\documentclass{amsart}

\usepackage{tikz}
\usepackage{pgfplots}
\usetikzlibrary{arrows,intersections, pgfplots.fillbetween}
\usepackage{amsmath,amssymb,amsfonts,amsthm}
\usepackage{hyperref}
\usepackage{fullpage}

\usepackage{algorithm}
\usepackage[noend]{algpseudocode}

\theoremstyle{plain}

\newtheorem{theorem}{Theorem}[section]
\newtheorem{proposition}{Proposition}[section]
\newtheorem{corollary}{Corollary}[section]
\newtheorem{lemma}[theorem]{Lemma}
\newtheorem{question}{Question}[section]
\theoremstyle{remark}

\newtheorem{definition}[theorem]{Definition}

\newcommand{\R}{\mathbb{R}}
\newcommand{\F}{\mathcal{F}}
\renewcommand{\P}{\mathbb{P}}
\newcommand{\E}{\mathbb{E}}
\newcommand{\Var}{\textnormal{Var}}
\newcommand{\Cov}{\textnormal{Cov}}

\newcommand{\Pbb}{\mathbb{P}}
\newcommand{\Rbb}{\mathbb{R}}
\newcommand{\Ebb}{\mathbb{E}}
\newcommand{\Nbb}{\mathbb{N}}
\newcommand{\Zbb}{\mathbb{Z}}

\newcommand{\diam}{\textnormal{diam}}
\newcommand{\dhaus}{d_{\textnormal{H}}}
\newcommand{\dvechaus}{\vec{d}_{\textnormal{H}}}
\newcommand{\dtrop}{d_{\textnormal{tr}}}
\newcommand{\cpt}{\textnormal{K}}
\newcommand{\supp}{\textnormal{supp}}
\newcommand{\CovNum}{\mathcal{N}}
\newcommand{\closed}{\textnormal{F}}
\newcommand{\kurouter}{\textnormal{Ls}}
\newcommand{\kurinner}{\textnormal{Li}}
\newcommand{\kurlimit}{\textnormal{Lt}}
\renewcommand{\emptyset}{\varnothing}

\newcommand{\Pcal}{\mathcal{P}}
\newcommand{\Gcal}{\mathcal{G}}
\newcommand{\Hcal}{\mathcal{H}}
\newcommand{\Bcal}{\mathcal{B}}

\newcommand{\Ucal}{\mathcal{U}}
\newcommand{\Ecal}{\mathcal{E}}

\newcommand{\comment}[1]{}

\begin{document}
	
\title[Fr\'echet mean set estimation in the Hausdorff metric]{Fr\'echet mean set estimation in the~\\Hausdorff metric, via relaxation}

\author{Mo\"ise Blanchard}
\address{Operations Research Center\\
	Massachusetts Institute of Technology,
	Cambridge, MA, USA}
\email{moiseb@mit.edu}

\author{Adam Quinn Jaffe}
\address{Department of Statistics\\
	University of California, Berkeley,
	Berkeley, CA, USA}
\email{aqjaffe@berkeley.edu}


\date{\today}

\begin{abstract}
	This work resolves the following question in non-Euclidean statistics: Is it possible to consistently estimate the Fr\'echet mean set of an unknown population distribution, with respect to the Hausdorff metric, when given access to independent identically-distributed samples?
	Our affirmative answer is based on a careful analysis of the ``relaxed empirical Fr\'echet mean set estimators'' which identify the set of near-minimizers of the empirical Fr\'echet functional and where the amount of ``relaxation'' vanishes as the number of data tends to infinity.
	On the theoretical side, our results include exact descriptions of which relaxation rates give weak consistency and which give strong consistency, as well as a description of an estimator which (assuming only the finiteness of certain moments and a mild condition on the metric entropy of the underlying metric space) adaptively finds the fastest possible relaxation rate for strongly consistent estimation.
	On the applied side, we consider the problem of estimating the set of Fermat-Weber points of an unknown distribution in the space of equidistant trees endowed with the tropical projective metric; in this setting, we provide an algorithm that provably implements our adaptive estimator, and we apply this method to real phylogenetic data.
\end{abstract}

\maketitle


\section{Introduction}\label{sec:intro}

The goal of the field of non-Euclidean statistics is to extend classical (Euclidean) statistical theory to geometric settings which are relevant from the point of view of applications.
For example, many applied statisticians work with data that live in spaces of matrices subject to non-linear constraints \cite{dryden2009non, fiori2009learning}, quotients of Euclidean spaces by certain group actions \cite{le2000frechet,Orbifold}, spaces of graphs \cite{GinestetGraphs,UnlabeledGraphs,MayerComputation,MayerTheory}, spaces of trees \cite{BHV,TreeMean}, spaces of persistence diagrams \cite{turner2014frechet,PersistenceUniqueness}, and, perhaps most importantly, Riemannian manifolds \cite{Bhattacharya,FrechetCircle,FrechetSphere,Chakraborty_2015_ICCV,KarcherSOn,StiefelStatistics,GeomStats}.
For these settings, one has to extend or modify most classical methodologies to accommodate the new aspects of geometry.

Among the fundamental objects of study in non-Euclidean statistics is the \textit{Fr\'echet mean}, which generalizes the Euclidean notions of mean and median to general metric spaces \cite{Huckemann2015, EvansJaffeSLLN, SchoetzSLLN}.
As such, Fr\'echet means are a canonical notion of central tendency for data living in an abstract metric space.
The primary difficulty in the theoretical study of Fr\'echet means is that, for general metric spaces, the Fr\'echet mean is a set rather than a single point.

These set-valued considerations divide the existing literature on Fr\'echet means into two different parts.
Some authors avoid the technicalities of the set-valued setting by limiting their scope to only consider probability measures whose Fr\'echet mean set is a singleton \cite{Bhattacharya,ManifoldInference,FrechetCLTs,ManifoldsI,ManifoldsII};
in this ``uniqueness'' setting, one can prove many limit theorems which are very similar to their classical (Euclidean) counterparts.
Other authors embrace the generality of the set-valued setting
\cite{Ziezold, Sverdrup,EvansJaffeSLLN,SchoetzSLLN}, although much less is known in this case.
This paper strengthens the recent results  \cite{EvansJaffeSLLN,SchoetzSLLN} concerning the set-valued setting, and also has applications to the uniqueness setting.


\medskip

We briefly describe here our setting. We defer additional definitions of related objects to Fr\'echet means to Subsection~\ref{subsec:defs}.
Throughout the remainder of this section, $p\ge 1$ is a fixed exponent and $(X,d)$ is a fixed metric space with the Heine-Borel property (or simply, a \textit{HB space}), that is, such that all $d$-closed and bounded subsets are $d$-compact.
We write $\cpt(X)$ for the collection of all non-empty $d$-compact subsets of $X$. 
We let $\Pcal(X)$ be the space of Borel probability measures on $(X,d)$ and for $r\geq 0$, we let $\Pcal_{r}(X)$ be the space of all $\mu\in\Pcal(X)$ satisfying $\E_{Y\sim \mu}[d^r(x,Y)] = \int_{X}d^{r}(x,y)\, d\mu(y) < \infty$ for some $x\in X$. Roughly speaking, $\Pcal_{r}(X)$ is the space of all distributions on $X$ which have $r$ moments finite.
We next assume that $\mu\in\Pcal(X)$ is a fixed but unknown Borel probability measure on $(X,d)$ and let $Y_1,Y_2,\ldots$ be an independent identically-distributed (IID) sequence with marginal distribution $\mu$, defined on a probability space $(\Omega,\F,\P_{\mu})$. We write $\E_{\mu}, \Var_{\mu}$, $\Cov_{\mu}$ for the expectation, variance, and covariance on this space, respectively, and we remark that all ``almost sure'' statements should be understood to be with respect to $\P_{\mu}.$
With these notations at hand, we are ready to define the Fr\'echet $p$-mean of $\mu$, denoted $F_p(\mu)$. Intuitively, $F_p(\mu)$ is the set of minimizers of the following optimization problem,
\begin{equation*}
    \min_{x\in X} \, \int_X d^p(x,y) d\mu(y).
\end{equation*}
However, it is known \cite[Lemma~2.6]{EvansJaffeSLLN} that we have
\begin{equation}\label{eqn:ineq-2}
	|d^p(x,y)-d^p(x',y)|\le pd(x,x')(d^{p-1}(x,y)+d^{p-1}(x',y))
\end{equation}
for all $x,x',y\in X$. Hence, the function $y\mapsto d^p(x,y)-d^p(x',y)$ is $\mu$-integrable for $\mu\in \Pcal_{p-1}(X)$ and $x,x'\in X$.
In particular, we can define the \textit{Fr\'echet functional} $W_p:\mathcal{P}_{p-1}(X)\times X^2\to \R$ via
\begin{equation*}
	W_p(\mu,x,x') := \int_{X}(d^p(x,y)-d^p(x',y))\,d\mu(y).
\end{equation*}
For $\mu\in \Pcal_{p-1}(X)$ we then define the \textit{Fr\'echet mean set} to be
\begin{equation*}
	F_p(\mu) := \{x\in X: W_p(\mu,x,x')\le 0\textrm{ for all } x'\in X\}.
\end{equation*}
In full generality, it is possible for the Fr\'echet mean set to be empty, to be a singleton, or to be a set with more than one point.
Our goal is to construct a sequence of estimators $\hat{F}^n_p:X^n\to \cpt(X)$ for $n\in\Nbb$ such that $\hat F^n_p = \hat F^n_p(Y_1,\ldots, Y_n)\to F_p(\mu)$ as $n\to\infty$ in some useful sense.

A few reasonable senses of convergence involve the \emph{one-sided Hausdorff distance} $\dvechaus$, defined via $\dvechaus(K,K') := \max_{x\in K}d(x,K') := \max_{x\in K}\min_{x'\in K'}d(x,x')$ for $K,K'\in \cpt(X)$.
These notions of convergence are stated as
\begin{equation}\label{eqn:one-sided-Hausdorff-1}
	\dvechaus(\hat F^n_p,F_p(\mu))\to 0
\end{equation}
and
\begin{equation}\label{eqn:one-sided-Hausdorff-2}
	\dvechaus(F_p(\mu),\hat F^n_p)\to 0
\end{equation}
as $n\to\infty$, where the convergence can either be taken in probability or almost surely.
Indeed, these convergences each have a clear statistical meaning (in an asymptotic sense): the first \eqref{eqn:one-sided-Hausdorff-1} is a guarantee of ``no false positives'' since it shows that every element of $\hat F^n_p$ is close to some element of $F_p(\mu)$, and the second \eqref{eqn:one-sided-Hausdorff-2} is a guarantee of ``no false negatives'' since it shows that every element of $F_p(\mu)$ is close to some element of $\hat F^n_p$.
One may also be interested in the stronger notion of convergence where both \eqref{eqn:one-sided-Hausdorff-1} and \eqref{eqn:one-sided-Hausdorff-2} hold; this is equivalent to
\begin{equation}\label{eqn:full-Hausdorff}
	\dhaus(F_p(\mu),\hat F^n_p)\to 0
\end{equation}
almost surely as $n\to\infty$, where the quantity $\dhaus(K,K') := \max\{\dvechaus(K,K'),\dvechaus(K',K)\}$ for $K,K'\in \cpt(X)$ is called the \emph{Hausdorff metric}. 
In statistical terms, the convergence \eqref{eqn:full-Hausdorff} guarantees both ``no false positives'' and ``no false negatives''.
Thus, let us say that an estimator $\hat F^n_p$ is \textit{weakly $\dhaus$-consistent} if \eqref{eqn:full-Hausdorff} holds in probability and that it is \textit{strongly $\dhaus$-consistent} if \eqref{eqn:full-Hausdorff} holds almost surely.

We have, as our motivation, the following two reasons for studying the property of $\dhaus$-consistency:
\begin{itemize}
	\item From a methodological point of view, we are of the opinion that $\dhaus$-consistency is an important property that any reasonable estimator $\hat F^n_p$ of $F_p(\mu)$ should possess.
	Indeed, this has already been observed by previous authors (\cite[p. 1118]{HuckemannLeafs}, \cite[Remark~2.5]{SchoetzSLLN}, and \cite[p. 60]{turner2014frechet}) who sought (explicitly or implicitly) some form of $\dhaus$-consistency.
	Roughly speaking, $\dhaus$-consistency guarantees a much more useful view of the estimand, compared to existing consistency properties studied in the literature.
	\item Many existing statistical results in the literature require the a priori ``uniqueness assumption'' that the Fr\'echet mean set consists of a single point.
	This has led to the conventional wisdom that one should seek geometric conditions under which the uniqueness assumption is guaranteed to hold (see, for example, \cite{PersistenceUniqueness,FrechetCircle, Afsari}).
	Yet, such assumptions are, typically, only used to guarantee that \eqref{eqn:one-sided-Hausdorff-1} implies \eqref{eqn:one-sided-Hausdorff-2}, hence that \eqref{eqn:one-sided-Hausdorff-1} implies that $\dhaus$-consistency \eqref{eqn:full-Hausdorff}.
	Thus, there is some hope that one can bypass complicated geometric considerations by directly studying $\dhaus$-consistency.
\end{itemize}
Generally speaking, we believe $\dhaus$-consistency deserves much further study in set-valued statistics.

Let us now recall the successes and limitations of the empirical Fr\'echet mean $F_p(\bar \mu_n)$ as an estimator of the population Fr\'echet mean $F_p(\mu)$, where $\bar\mu_n := \frac{1}{n}\sum_{i=1}^{n}\delta_{Y_i}$ is the empirical measure.
On the positive side, the main result of \cite{SchoetzSLLN} shows that $F_p(\bar \mu_n)$ satisfies \eqref{eqn:one-sided-Hausdorff-1} almost surely.
On the negative side, a result of \cite{EvansJaffeSLLN} shows, for finite metric spaces $(X,d)$, that $F_p(\bar \mu_n)$ satisfies \eqref{eqn:one-sided-Hausdorff-2} in probability only if $F_p(\mu)$ is a singleton, up to some trivialities.
In particular, $F_p(\bar \mu_n)$ is not even weakly $\dhaus$-consistent in the general set-valued setting.

Yet, not all hope is lost!
It certainly seems possible to cook up some other estimator $\hat F^n_p$ which is $\dhaus$-consistent.    
The major contribution of this work is to show that this is possible and that one such estimator is not too different from the empirical Fr\'echet mean set itself.
In this sense, our work can be seen as resolving an open question in previous literature, and as allowing one  to dispense with various unnecessary uniqueness hypothesis in the literature.
Additionally, it establishes that the empirical Fr\'echet mean set is (asymptotically) inadmissible with respect to any loss based on $\dhaus$.

In order to construct such an estimator, we now define a few additional concepts. For $\mu\in\Pcal_{p-1}(X)$, one defines the \textit{relaxed Fr\'echet mean set} via
\begin{equation*}
	F_p(\mu,\varepsilon) := \{x\in X: W_p(\mu,x,x') \leq \varepsilon \textrm{ for all } x'\in X\},
\end{equation*}
where $\varepsilon\ge 0$ is an arbitrary parameter called the \textit{relaxation scale}.
In words, the relaxed Fr\'echet mean set $F_p(\mu,\varepsilon)$ consists of all of the near-minimizers of the Fr\'echet functional, up to an additive error of $\varepsilon$.
Thus, we have $F_p(\mu) = F_p(\mu,0)$; we sometimes refer to $F_p(\mu)$ as the \emph{unrelaxed Fr\'echet mean set} or the \emph{vanilla Fr\'echet mean set}.
A sequence of relaxation scales $\{\varepsilon_n\}_{n\in\Nbb}$ applied to form the empirical relaxed Fr\'echet mean sets $\{F_p(\bar \mu_n,\varepsilon_n)\}_{n\in\Nbb}$ will be referred to as a \emph{relaxation rate}, and we often simply write $\varepsilon_n$ in place of $\{\varepsilon_n\}_{n\in\Nbb}$.
By a \textit{random relaxation rate} we mean a relaxation rate $\varepsilon_n$ in which $\varepsilon_n$ is a random variable for each $n\in\Nbb$; these can of course be deterministic or independent of $Y_1,Y_2,\ldots$, but the most interesting and important examples are \emph{adaptive}, in the sense that $\varepsilon_n$ is $\sigma(Y_1,\ldots, Y_n)$-measurable for all $n\in\Nbb$.

Relaxed Fr\'echet mean sets have received a bit of attention in the recent work \cite{SchoetzSLLN}.
Indeed, \cite[Corollary~5.1 and Corollary~5.2]{SchoetzSLLN} shows that if a relaxation rate satisfies $\varepsilon_n\to 0$, then $F_p(\bar \mu_n,\varepsilon_n)$ satisfies \eqref{eqn:one-sided-Hausdorff-1} almost surely.
Interestingly, it was observed \cite[Appendix~A.3]{SchoetzSLLN} in the simple example of $X=\{0,1\},p=2$, and $\mu= \frac{1}{2}\delta_0+\frac{1}{2}\delta_1$, that choosing the relaxation rate as $\varepsilon_n := n^{-1/4}$ implies that $F_2(\bar \mu_n,\varepsilon_n)$ satisfies  \eqref{eqn:one-sided-Hausdorff-2} almost surely.
In words, applying a vanishing relaxation rate does not disturb \eqref{eqn:one-sided-Hausdorff-1}, and applying a sufficiently slowly vanishing relaxation rate allows obtaining \eqref{eqn:one-sided-Hausdorff-2}.
Of course, in order to be as efficient as possible, one wants to choose the relaxation rate which is the fastest possible among all sufficiently slow relaxation rates.
Thus, we are motivated to study the following (stated informally):

\begin{question}
	For a HB space $(X,d)$, some $p\ge 1$, and an unknown sufficiently integrable $\mu\in\Pcal(X)$, what is the fastest relaxation rate $\varepsilon_n\to 0$ for which $F_p(\bar \mu_n,\varepsilon_n)$ is a $\dhaus$-consistent estimator of $F_p(\mu)$?
\end{question}

The upshot of this work is that, in a very generic setting, we are able to provide a complete answer to the motivating question.
That is, we exactly characterize which relaxation rates give rise to (weak and strong) $\dhaus$-consistency.

\section{Main results}

In this section, we state the important results of the paper. These results point to a simple recommendation for applied researchers: When using Fr\'echet means to study non-Euclidean data in the real world, if one aims to find the whole Fr\'echet mean set, it is advantageous to apply a small bit of relaxation.
We hope that the results of this subsection will guide practitioners in making an informed choice of the relaxation scale.

To state them more precisely, we need to introduce some mild form of regularity to impose on the underlying metric space.
For example, a common assumption since the work \cite{ManifoldsI,ManifoldsII} is to take $(X,d)$ to be a HB space in order to get \eqref{eqn:one-sided-Hausdorff-1} almost surely.
It turns out that we also need some control on the metric entropy of the balls in $(X,d)$, so we write $\CovNum(\varepsilon;X,d)$ or simply $\CovNum_X(\varepsilon)$ for the minimal number of $d$-closed balls of radius $\varepsilon\ge 0$ needed to cover $X$.
This leads us to the following:

\begin{definition}\label{def:HBD}
	A metric space $(X,d)$ is called a \textit{Dudley space} if it is compact and 
	\begin{equation*}
		\int_0^\infty \sqrt{\log \CovNum_X(\varepsilon)}d\varepsilon <\infty.
	\end{equation*}
	A metric space is called a \textit{Heine-Borel-Dudley (HBD) space} if all of its closed, bounded subsets are Dudley.
\end{definition}

Let us emphasize that most metric spaces encountered in practice in non-Euclidean statistics are HBD.
Indeed, most metric spaces of interest are finite-dimensional, in the sense that $(X,d)$ admits some number $k > 0$ such that covering numbers for compact sets $K\subseteq X$ satisfy $\log \CovNum_K(\varepsilon)  \lesssim k\log (1/\varepsilon)$ as $\varepsilon\to 0$, and it can be easily shown that such spaces are HBD.
In particular, this is true whenever $X$ is a Riemannian sub-manifold of finite-dimensional Euclidean space, whenever $X$ is a quotient of Euclidean space by an isometric group action, and more.
Additionally, some infinite-dimensional metric spaces of interest in non-parametric statistics are also HBD:
for example, the space of $\alpha$-H\"older functions in dimension $d\ge 1$ endowed with the uniform metric, when $\alpha>2d$ \cite[Theorem~8.2.1]{dudley2014uniform}.

Next, we introduce some notation for probability measures.
As before, we let $Y_1,Y_2,\ldots$ be an IID sequence with marginal distribution $\mu\in \Pcal(X)$, on a probability space $(\Omega,\F,\P_{\mu})$.
All ``almost sure'' and ``in probability'' statements are implicitly understood to be taken with respect to $\P_{\mu}$. We introduce an important quantity that will appear throughout our results.
For $p\geq 1$ and $\mu$ any Borel probability measure on $(X,d)$, we define the value
\begin{equation}\label{eqn:sigma-p}
	\sigma_p(\mu) := \sqrt{2} \sup_{x,x'\in F_p(\mu)}\sqrt{\Var_{\mu}(d^p(x,Y_1)-d^p(x',Y_1))}.
\end{equation}
We note that $\sigma_p(\mu)$ is finite as long as we have $\mu\in \Pcal_{2p-2}(X)$ (Lemma~\ref{lemma:finite_sigma}), and that $\sigma_p(\mu)$ is positive in most cases, for example when $\supp(\mu) = X$ (Lemma~\ref{lem:sigma_pos}), but also much more generally.
Our main results are comprehensive for the typical case of $\sigma_p(\mu) \in (0,\infty)$, but we later point out that the story becomes more complicated when $\sigma_p(\mu) = 0$.

Now we turn to the results.
In light of the classical central limit theorem (CLT), a natural guess for a useful relaxation rate is $\Theta(n^{-1/2})$. (We refer to Subsection~\ref{subsec:defs} for the definition of Bachman-Landau symbols.)
It turns out that these rates do not even yield convergence in probability, but we can cleanly bound their asymptotic probability of error as in the following.

\begin{theorem}\label{thm:gaussian_tail_one_sided}
	Let $(X,d)$ be a HBD space, $p\geq 1$, and $\mu\in\Pcal_{2p-2}(X)$.
	Then, there exists a constant $M_p(\mu)<\infty$ 
	such that for any $c\ge M_p(\mu)$ and
	any random relaxation rate $\varepsilon_n$ satisfying $\liminf_{n\to\infty}\varepsilon_n \sqrt n \geq c$ a.s., we have
	\begin{equation*}
		\sup_{\delta>0}\limsup_{n\to\infty}\Pbb_{\mu}(\dhaus(F_p(\bar \mu_n,\varepsilon_n),F_p(\mu))\ge\delta) \leq  \exp\left(-\frac{(c-M_{p}(\mu))^2}{\sigma_{p}^2(\mu)}\right),
	\end{equation*}
	with the convention the right side is 0 if $\sigma_p^2(\mu) = 0$.
\end{theorem}
 
The proof is given in Appendix~\ref{sec:proof_general_results}.
For deterministic relaxation rates $\varepsilon_n = cn^{-1/2}$, the preceding result shows that a certain asymptotic probability of error for the relaxed Fr\'echet mean set estimators can be made arbitrarily small by making the pre-factor $c$ arbitrary large, and that the decay of the asymptotic error probabilities is at most Gaussian-like.
Although the result does not include a lower bound on the left-hand side, we later prove (Theorem \ref{thm:weak_convergence}) in the case of $\sigma_p(\mu)>0$ that a deterministic relaxation rate $\varepsilon_n$ is weakly consistent if and only if $\varepsilon_n\in o(1)\cap\omega(n^{-1/2})$. Hence, relaxation scales in $\Theta(n^{-1/2})$ are generally not slow enough to get even weak $\dhaus$-consistency.

Now we move on to our second main result.
Because of the previous discussion, it is necessary to choose some relaxation rate which is slightly slower than $\Theta(n^{-1/2})$.
It turns out that a logarithmic correction is exactly what is needed so that the pre-factor separates the strongly $\dhaus$-consistent regime from the not strongly $\dhaus$-consistent regime.

\begin{theorem}\label{thm:consistency}
	Let $(X,d)$ be a HBD space, $p\geq 1$, and $\mu\in\Pcal_{2p-2}(X)$.
	Then, any random relaxation rate $\varepsilon_n$ satisfies the following:
	\begin{itemize}
		\item[(i)] If
		\begin{equation*}
			\liminf_{n\to\infty} \varepsilon_n\sqrt{\frac{n}{\log\log n}} > \sigma_p(\mu) \quad \textrm{ and }\quad \lim_{n\to\infty}\varepsilon_n=0 \quad \textrm{a.s.,}
		\end{equation*}
		then $F_p(\bar \mu_n,\varepsilon_n)$ is strongly $\dhaus$-consistent, and $\P(F_p(\mu) \subseteq F_p(\bar \mu_n,\varepsilon_n) \textrm{ for suff. large } n\in\Nbb) =1$.
		\item[(ii)] If
		\begin{equation*}
			\limsup_{n\to\infty} \varepsilon_n\sqrt{\frac{n}{\log\log n}} < \sigma_p(\mu) \quad \textrm{a.s.,}
		\end{equation*}
		then $F_p(\bar \mu_n,\varepsilon_n)$ is not strongly $\dhaus$-consistent.
		In fact, in this case it is strongly $\dhaus$-inconsistent, in the sense that $\P(\dvechaus(F_p(\mu),F_p(\bar\mu_n,\varepsilon_n))\to 0 \textrm{ as } n \to \infty) = 0$.
	\end{itemize}
\end{theorem}

The proof is given in Section~\ref{sec:proofs}. For the deterministic relaxation rates $\varepsilon_n = cn^{-1/2}(\log\log n)^{1/2}$, the preceding result guarantees:
\begin{itemize}
	\item[(i)] If $c>\sigma_p(\mu)$ then $F_p(\bar \mu_n,\varepsilon_n)$ is strongly $\dhaus$-consistent.
	\item[(ii)] If $c<\sigma_p(\mu)$ then $F_p(\bar \mu_n,\varepsilon_n)$ is not strongly $\dhaus$-consistent.
\end{itemize}
In words, this means that the strong $\dhaus$-consistency for relaxed empirical Fr\'echet mean set estimators experiences a sharp transition at the relaxation rate $\varepsilon_n = \sigma_p(\mu) n^{-1/2}(\log\log n)^{1/2}$. 
The ``critical'' relaxation rate of $\varepsilon_n = \sigma_p(\mu) n^{-1/2}(\log\log n)^{1/2}$ appears to be more complicated, and we do not know what behavior to expect in this case.

Let us now summarize and compare the results thus far.
In light of Theorem~\ref{thm:gaussian_tail_one_sided}, it is natural to consider the \textit{weak asymptotic error probability} for each population distribution $\mu\in \Pcal_{2p-2}(X)$ and each relaxation rate $\varepsilon_n$ defined via
\begin{equation*}
	\textrm{WE}_{\mu}(\varepsilon_n) = \sup_{\delta>0}\limsup_{n\to\infty}\P_{\mu}(\dhaus(F_p(\bar \mu_n,\varepsilon_n),F_p(\mu))\ge \delta).
\end{equation*}
Similarly, in light of Theorem~\ref{thm:consistency}, it is natural to consider the \textit{strong asymptotic error probability} for each $\mu\in \Pcal_{2p-2}(X)$ and each $\varepsilon_n$ defined via
\begin{equation*}
	\textrm{SE}_{\mu}(\varepsilon_n) = 1- \P_{\mu}\left(\dhaus(F_p(\bar \mu_n,\varepsilon_n),F_p(\mu)) \underset{n\to\infty}{\longrightarrow} 0\right).
\end{equation*}
(Note that $\textrm{WE}_{\mu}(\varepsilon_n)$ and $\textrm{SE}_{\mu}(\varepsilon_n)$ are shorthand for $\textrm{WE}_{\mu}(\{\varepsilon_n\}_{n\in\Nbb})$ and $\textrm{SE}_{\mu}(\{\varepsilon_n\}_{n\in\Nbb})$, respectively.)
Figure \ref{fig:asymp-error-probs} summarizes the results of Theorem \ref{thm:gaussian_tail_one_sided} and Theorem \ref{thm:consistency} in terms of these weak and strong error quantities.

\begin{figure}
    \begin{tikzpicture}
		\draw[ultra thick, -latex] (7,0) -- (7+5,0);
		\draw[ultra thick, -latex] (7,0) -- (7,3);
		
        \draw[ultra thick] (7-0.2,2) node[left]{$1$} -- (7+0.2,2);
        \draw[ultra thick] (7-0.2,0) node[left]{$0$} -- (7+0.2,0);
        \node at (7+5,-0.4) {$c$};
        \node at (7+2.5,-0.4) {$\sigma_p(\mu)$};

        \draw[ultra thick, color=blue, -o] (7,2) -- (7+2.68,2);
        \node[color=blue] at (7+4,1.4) {$\textrm{SE}_{\mu}\left(c\sqrt{\frac{\log\log n}{n}}\right)$};
	\draw[ultra thick, color=blue, o-latex] (7+2.4,0) -- (7+5,0);
        \draw[thick, dashed] (7+2.55,1.9) -- (7+2.55,0.1);
        \draw[thick, dashed] (7+2.55,2.5) -- (7+2.55,2.1);

		\draw[ultra thick, name path=xaxis1] (0,0) -- (2,0);
            \draw[ultra thick, -latex, name path=xaxis2] (2,0) -- (5,0);
		\draw[ultra thick, -latex] (0,0) -- (0,3);
		
        \draw[ultra thick] (-0.2,2) node[left]{$1$} -- (0.2,2);
        \draw[ultra thick] (-0.2,0) node[left]{$0$} -- (0.2,0);
        \node at (5,-0.4) {$c$};

        \node at (2,-0.4) {$M_p(\mu)$};

        

        \draw[opacity=0, -latex, name path=line] (0,2) -- (2,2);

        \node[color=red] at (4,1.4) {$\textrm{WE}_{\mu}\left(\frac{c}{\sqrt n}\right)$};
        
        \draw[opacity=0, -latex, domain=2:5, smooth, variable = \x, name path = gaussian] plot ({\x}, {2*exp(-0.5*(\x-2)*(\x-2)});

        \draw[ultra thick, color=red, -, domain=0:1, smooth, variable = \x] plot ({\x}, {0.5+(\x-1)^4});
        \draw[ultra thick, -, color=red, domain=1:2.5, smooth, variable = \x] plot ({\x}, {0.5});
        \draw[ultra thick, -, color=red, domain=2.5:3, smooth, variable = \x] plot ({\x}, {0.5-0.5*(\x-2.5)^2});
        \draw[ultra thick, -latex, color=red, domain=3:5, smooth, variable = \x] plot ({\x}, {0.375*exp(-1.5*(\x-3))});
        \draw[thick, dashed] (2,2.5) -- (2,0);

        \tikzfillbetween[of=line and xaxis1]{red, opacity=0.1};
        \tikzfillbetween[of=gaussian and xaxis2]{red, opacity=0.1};
	\end{tikzpicture}

    \caption{For $\mu\in\Pcal_{2p-2}(X)$, we plot the asymptotic error probabilities of the empirical relaxed Fr\'echet mean set estimators,	as a function of the pre-factor $c\ge 0$, when the relaxation rate is taken to be $cn^{-1/2}$ (left) or $cn^{-1/2}(\log \log n)^{1/2}$ (right).
    By Theorem~\ref{thm:gaussian_tail_one_sided}, we know that $c\mapsto\textrm{WE}_{\mu}(cn^{-1/2})$ decays at least as quickly as a Gaussian; the red region depicts the bound on the error, and the red curve depicts a cartoon of the weak error probability.
    By Theorem~\ref{thm:consistency}, we know that $c\mapsto \textrm{SE}_{\mu}(cn^{-1/2}(\log \log n)^{1/2})$ decays abruptly; the blue curve depicts the exact form of the strong error rate.}
	\label{fig:asymp-error-probs}
\end{figure}

With these results in hand, it still remains to show that one can $\dhaus$-consistently estimate the Fr\'echet mean set of the unknown population distribution $\mu$.
As a first guess, one may simply try to use the adaptive relaxation rate $\varepsilon_n := \sigma_p(\bar \mu_n)n^{-1/2}(\log \log n)^{1/2}$, or perhaps $\varepsilon_n := (1+\delta)\sigma_p(\bar \mu_n)n^{-1/2}(\log \log n)^{1/2}$ for any $\delta > 0$.
However, it can be shown in simple examples that the convergence $\sigma_p(\bar \mu_n) \to \sigma_p(\mu)$ can plainly fail.
Thus, this adaptive rate does not lead to strong consistency.

Instead, one has to be more careful about the method of estimating $\sigma_p(\mu)$.
For a better choice of estimator, let us introduce a relaxed version of \eqref{eqn:sigma-p} given as
\begin{equation*}
	\sigma_p(\mu,\varepsilon) := \sqrt{2}\sup_{x,x'\in F_p(\mu,\varepsilon)}\sqrt{\Var_{\mu}(d^p(x,Y_1)-d^p(x',Y_1))}
\end{equation*}
for $\mu\in \Pcal_{2p-2}(X)$ and $\varepsilon\ge 0$.
Then we have the following:

	\begin{theorem}\label{thm:adaptive-consistency}
		Let $(X,d)$ be a HBD space, $p\geq 1$, and $\mu\in \Pcal_{2p-2}(X)$.
		For any random relaxation rate $\varepsilon_{1,n}$ satisfying
		\begin{equation*}
			\lim_{n\to\infty}\varepsilon_{1,n} \sqrt{\frac{n}{\log\log n}} = \infty \quad \text{and} \quad \lim_{n\to\infty}\varepsilon_{1,n} = 0\quad \textrm{a.s.},
		\end{equation*}
		and for any $\delta > 0$, define the adaptively-chosen relaxation rate $\varepsilon_{2,n,\delta}$ via
		\begin{equation*}
			\varepsilon_{2,n,\delta}:=(1+\delta)\sigma_p\left(\bar \mu_n,\varepsilon_{1,n}\right)\sqrt{\frac{\log \log n}{n}}.
		\end{equation*}
		Then, $F_p(\bar \mu_n,\varepsilon_{2,n,\delta})$ is strongly $\dhaus$-consistent and we have
		\begin{equation*}
			\lim_{\delta \to 0}\lim_{n\to\infty} \varepsilon_{2,n,\delta} \sqrt{\frac{n}{\log \log n}} = \sigma_p(\mu) \quad \textrm{a.s.}.
		\end{equation*}
	\end{theorem}

The proof is given in Section~\ref{sec:proofs}. The preceding result motivates the construction of a certain ``two-step'' estimator of the population Fr\'echet mean set:
First, we compute an initial estimate of the Fr\'echet mean set using any suboptimal but consistent relaxation rate.
Second, we compute a better estimate of the Fr\'echet mean set using the optimal rate, where the pre-factor is determined from the initial estimate of the first step.
The result then shows that this two-step estimator is strongly $\dhaus$-consistent, and that it adaptively finds the fastest possible relaxation rate for strongly $\dhaus$-consistent estimation; in this sense, it is the optimal relaxation-based estimator.

We now comment on the computational aspects of our theoretical results.
First of all, we note that, while there are many known algorithms and heuristics to compute at least one point of the Fr\'echet mean set in many particular metric spaces of interest \cite{MayerComputation,TreeOptimality,Chakraborty_2015_ICCV,FrechetCircle,FrechetSphere,KarcherSOn,GeomStats}, it is much more difficult to compute the entire Fr\'echet mean set.
Second of all, we note that computing $\sigma_p(\mu)$ or its empirical relaxed counterpart $\sigma_p(\bar \mu_n,\varepsilon)$ is typically very difficult.
As such, it seems difficult in general to construct an explicit algorithm which provably implements the two-step estimator of Theorem~\ref{thm:adaptive-consistency}.
Nonetheless, we are aware of one explicit setting in which all of the computational considerations are resolved: the space of phylogenetic trees endowed with the tropical projective metric \cite{YoshidaLin, LinMonodYoshida, BarnhillYoshida}.
In Section~\ref{sec:phylo} we detail an extended analysis of the problem in this setting, and we define a practical adaptive-relaxed estimation procedure in Algorithm~\ref{alg:phylo-estimator}.
We believe an interesting direction of future study is to understand in which other metric spaces of interest one is able to implement such adaptive procedures.

As we previously remarked, our results provide exact characterizations of (weak and strong) $\dhaus$-consistency in the  case of $\sigma_p(\mu) > 0$.
However, in the case of $\sigma_p(\mu) = 0$, our results provide only sufficient conditions for $\dhaus$-consistency.
For example, suppose that $\sigma_p(\mu) = 0$; then, Theorem~\ref{thm:gaussian_tail_one_sided} shows that all rates in $\Omega(n^{-1/2})$ are weakly $\dhaus$-consistent (it can also be shown that we have $M_p(\mu)=0$ in this case), and Theorem~\ref{thm:consistency} shows that all rates in $\Omega(n^{-1/2}(\log \log n)^{1/2})$ are strongly $\dhaus$-consistent.
However, it turns out that neither of these conditions is necessary:
To illustrate this point (and to show that a totally general characterization of weak and strong $\dhaus$-consistency is likely very difficult to find), we provide Proposition~\ref{prop:example_square} and Proposition~\ref{prop:ex-two-point} which both satisfy $\sigma_{p}(\mu) = 0$ but where the boundaries between the consistent and inconsistent regimes exhibit very different behaviors.

Let us briefly comment on the method of proof of these results.
First of all, our results make heavy use of functional limit theorems; specifically, we need to find and apply a sufficiently powerful functional central limit theorem (fCLT) for Theorem~\ref{thm:gaussian_tail_one_sided} and a sufficiently powerful functional law of the iterated logarithm (fLIL) for Theorem~\ref{thm:consistency}.
Second of all, our results require a strengthened form of the (one-sided) Fr\'echet mean strong law of large numbers which was previously studied in \cite{EvansJaffeSLLN,SchoetzSLLN}.
We establish this strengthened result in Section~\ref{subsec:SLLN}, and we believe it will be of independent interest in non-Euclidean statistics and beyond.
Notably, this sharpened SLLN result implies the one-sided SLLN for $p$-medoids under the minimal assumption of a finite $(p-1)$th moment (Theorem~\ref{thm:medoid-SLLN}), which resolves the open question stated in \cite[Remark~2.11]{EvansJaffeSLLN}.

\medskip


There are several natural questions remaining at the end of this work.
Let us briefly review some of them, hoping that they will inspire similar work in the future.

First, our main results require $2p-2$ moments in the population distribution while the one-sided results of \cite{EvansJaffeSLLN,SchoetzSLLN} require only $p-1$ moments in the population distribution.
We believe that this is not an artifact of our proof, and that, if the $2p-2$ moments assumption is removed, then the requisite relaxation rates for weak and strong $\dhaus$-consistency will change.
To see this, note that our $2p-2$ moments assumption is used primarily to guarantee that the fLIL holds; and that even the classical one-dimensional LIL is known \cite{FellerLIL} to behave somewhat wildly when finite variance is not assumed.

Second, we note that it is also natural to consider relaxation rates in the even slower class $\Theta(n^{-1/2}(\log n)^{1/2})$.
From some heuristic calculations, we believe that such rates may lead to finite-sample concentration bounds for relaxed empirical Fr\'echet mean set estimators.
However, we expect that the finite-sample concentration of the $cn^{-1/2}(\log n)^{1/2}$-relaxed empirical Fr\'echet mean set estimators will be such that the choice of pre-factor $c > 0$ depends intricately on the geometry of $(X,d)$.
This should be contrasted with the case of relaxation rates in $\Theta(n^{-1/2}(\log \log n)^{1/2})$, where the dependence of the consistency of the $cn^{-1/2}(\log \log n)^{1/2}$-relaxed empirical Fr\'echet mean set estimators depends somewhat coarsely on the geometry of $(X,d)$.

Third, we pose the open question of comparing the two-step estimator with respect to the unrelaxed empirical Fr\'echet mean set estimator, in the uniqueness setting.
As the experiments in Section~\ref{sec:phylo} below show, the two-step estimator provides a quite conservative estimate of the population Fr\'echet mean set.
As such, it would be of interest to try to prove a rigorous notion which quantifies its ``inefficiency'' if uniqueness is assumed a priori.

Finally, we emphasize that the analysis used to prove our main results will likely be applicable to other settings.
Indeed, consider the case of non-convex stochastic optimization where the objective function is a sum of IID objectives.
According to \cite[Theorem~7.31(c)]{RockafellarWets}, there always exists a sufficiently slow $\varepsilon_n\to 0$ such that the set of $\varepsilon_n$-relaxed empirical minimizers converges (in a particular sense) to the set of unrelaxed population minimizers; while their result is totally deterministic, our results suggest that the randomness imposes a remarkable degree of regularity on the problem, in the sense that the fastest possible sufficiently slow relaxation rate should generically be $\Theta(n^{-1/2}(\log \log n)^{1/2})$.
We welcome subsequent work which makes this idea precise.

\section{Application to phylogenetics}
\label{sec:phylo}

In this section we detail an extended application of our theory to the problem of statistical inference with data coming in the form of trees.
Such data sets have become ubiquitous in phylogenetics, and there has been a huge development of mathematical theory for inference in such data sets \cite{BHV, TreeMean, TreeOptimality, BarnhillYoshida, YoshidaKNN, TropPhylo, YoshidaLin, LinMonodYoshida}.
Initial theory mainly focused on the Billera-Holmes-Vogtmann (BHV) treespace, although, due to computational difficulties, many authors have recently explored alternative geometries on tree spaces.

The works \cite{BarnhillYoshida, YoshidaKNN, TropPhylo, YoshidaLin} endow a suitable space of trees with the ``tropical projective metric'' and they consider the fundamental notion of central tendency to be the ``Fermat-Weber set'' which is equivalent to what we in this work call the Fr\'echet 1-mean set.
This perspective has significant computational advantages, but it also has some conceptual disadvantages in that the Fermat-Weber point appears to be non-unique in many applications.
Presently, we specialize our general results to this setting, we show that one can exactly implement an algorithm achieving the optimal relaxation rate, and we demonstrate the advantage of this approach on real and simulated data. All proofs of this section are deferred to Appendix~\ref{sec:phylogenomic}.
All figures of this section can be reproduced from the \texttt{Python} code available online at \href{https://github.com/moiseb/Relaxed_Frechet_Sets/}{\texttt{https://github.com/moiseb/Relaxed\_Frechet\_Sets/}} which provides an implementation of our main algorithm.

To begin, let us describe our basic objects, following the presentation of \cite{YoshidaLin}.
For $k\in\Nbb$, we write $\R^k/\R\mathbf{1}$ for the Euclidean space $\R^k$ quotiented by the action of $\R\mathbf{1}$, where we define $c\mathbf{1}\cdot (x_1,\ldots, x_k) := (x_1+c,\ldots, x_k+c)$ for $c\in\R$ and $(x_1,\ldots, x_k)\in\Rbb^k$; in other words, $\R^k/\R\mathbf{1}$ represents Euclidean space, modulo diagonal translations.
For $x = (x_1,\ldots, x_k)$ and $x' = (x_1',\ldots, x_k')$ in $\R^k$ we define
\begin{equation*}
    \dtrop(x,x') := \max\{x_i-x_i': 1\le i \le k\} - \min\{x_i-x_i': 1\le i \le k\}.
\end{equation*}
It turns out that $\dtrop$ is invariant under the action of $\Rbb\mathbf{1}$ on $\Rbb^k$ and that $\dtrop$ descends to a metric on $\Rbb^k/\Rbb\mathbf{1}$; by a slight abuse of notation, we also write $\dtrop$ for the induced metric on $\Rbb^k/\Rbb\mathbf{1}$, called the \textit{tropical projective metric}.
The authors of \cite{LinMonodYoshida, TropPhylo} show that $(\Rbb^k/\Rbb\mathbf{1},\dtrop)$ is a natural metric for phylogenetic inference.
They also show that $F_1$, called the set of \textit{Fermat-Weber points}, provides the natural notion of central tendency in this space, and also that the set of Fermat-Weber points is a polytope.
In particular, it can be exactly computed with the help of standard existing software.

Our theory specializes nicely to this case of $(X,d) = (\Rbb^k/\Rbb\mathbf{1},\dtrop)$ and $p=1$:

\begin{lemma}\label{lemma:tropical_HBD}
The space $(\Rbb^k/\Rbb\mathbf{1},\dtrop)$ is a HBD space.
\end{lemma}

On the statistical side of things, suppose $Y_1,Y_2,\ldots$ are IID data coming from an unknown population distribution $\mu\in\Pcal(\Rbb^k/\Rbb\mathbf{1})$, and, as before, write $(\Omega,\F,\P_{\mu})$ for the ambient probability space.
Theorem~\ref{thm:consistency} shows that there exist relaxed Fermat-Weber set estimators that are strongly $\dhaus$-consistent for estimating the population Fermat-Weber set, but, since the choice of pre-factor depends on the unknown population distribution, it is desirable instead to invoke Theorem~\ref{thm:adaptive-consistency} which shows that adaptively-chosen relaxation rates can be asymptotically optimal.

These considerations lead us to Algorithm~\ref{alg:phylo-estimator}, about which it is useful to make a few remarks.
First, the two procedures \textsc{FermatWeberValue} and \textsc{FermatWeberSet} that are used in Algorithm~\ref{alg:phylo-estimator} are straightforward adaptations of the results of \cite{YoshidaLin} which focus on the case of no relaxation.
Second, the multi-step nature of \textsc{AdaptRelaxFermatWeberSet} can be seen as a natural extension of the two-step nature of Theorem~\ref{thm:adaptive-consistency}; while two steps are sufficient to guarantee an asymptotically optimal pre-factor, we will see that multiple steps also satisfy asymptotic optimality while (empirically) increasing non-asymptotic performance. Third, while the computation of the adaptive pre-factor from Theorem~\ref{thm:adaptive-consistency} can, in general, be a difficult task, it turns out that, for this particular setting, the computation of $c_{s,n}$ in line 22 yields an upper bound which is asymptotically tight as $n\to\infty$.

\begin{algorithm}[t]
	\caption{In the tropical projective space $(\Rbb^{k}/\Rbb\mathbf{1},\dtrop)$, we define a procedure \textsc{AdaptRelaxFermatWeberSet} which (by Theorem~\ref{thm:phylo-alg}) adaptively finds the optimal relaxation rate for strongly $\dhaus$-consistent estimation.
    Here, we require two procedures from linear programming and polyhedral geometry, namely, \textsc{ExtremePoints} which takes as input a polyhedron and outputs its set of extreme points, and \textsc{ConvexHull} which takes as input a set of points and returns the subset corresponding to the extreme points of their convex hull. }\label{alg:phylo-estimator}
\linespread{1.0}\selectfont

	\begin{algorithmic}[1]
            \Procedure{FermatWeberValue}{}
    		\State \textbf{input:} data $Y_1,\ldots, Y_n\in \Rbb^{k}$
    		\State \textbf{output:} optimal objective $m_1(\bar \mu_n)\ge 0$
            \State $m_n \leftarrow \textbf{minimize} \,\frac{1}{n}\sum_{l=1}^n c_l$ \textbf{over} $(v,c)\in\Rbb^{k\times n}$
            \textbf{with} $c_l \geq v_i-v_j - (Y_l)_i + (Y_l)_j,  l\in[n],\; i,j\in[k]$
            \State \textbf{return} $m_n$
    	\EndProcedure

            \vspace{1mm}
            
            \Procedure{FermatWeberSet}{}
    		\State \textbf{input:} data $Y_1,\ldots, Y_n\in \Rbb^{k}$ and relaxation scale $\varepsilon_n \ge 0$
    		\State \textbf{output:} extreme point set $\{v_j: j\in I_n\}$ of $F_1(\bar \mu_n,\varepsilon_n)$
            \State $m_n \leftarrow$ \textsc{FermatWeberValue}$(Y_1,\ldots, Y_n)$
            \State $S\leftarrow \{(v,c)\in\Rbb^k\times\Rbb^n:\frac{1}{n}\sum_{l=1}^n c_l \leq m_n + \varepsilon_n,\; c_l \geq v_i-v_j - (Y_l)_i + (Y_l)_j,\; l\in[n],\; i,j\in[k] \}$
            \State $\{(v_j,c_j)\in \Rbb^k\times\Rbb^n: j\in I_n\} \leftarrow$ \textsc{ExtremePoints}$(S)$
            \State \textbf{return} $\{v_j: j\in I_n\}$
    	\EndProcedure

            \vspace{1mm}
  
		\Procedure{AdaptRelaxFermatWeberSet}{}
		\State \textbf{input:} data $Y_1,\ldots, Y_n\in \Rbb^{k}$ and small constant $\delta>0$
		\State \textbf{output:} subset $F_n^{\ast}\subseteq \Rbb^{k}$
		
		\State $m_n \leftarrow$ \textsc{FermatWeberValue}$(Y_1,\ldots, Y_n)$
        \State $\varepsilon_{1,n,\delta} \leftarrow m_n n^{-1/2}\log\log n$
        \State $s\leftarrow 0$
        \Repeat
                \State $s \leftarrow s+1$
                
		      \State $\{v_{s,j}: j\in I_{s,n}\} \leftarrow$ \textsc{FermatWeberSet}$(Y_1,\ldots, Y_n,\varepsilon_{s,n,\delta})$
                \State $c_{s,n} \leftarrow$ \textbf{maximize} $(\frac{1}{n}\sum_{i=1}^{n}|\dtrop(v_{s,j},Y_i)-\dtrop(v_{s,j'},Y_i)|^2)^{1/2}$ \textbf{over} $j,j'\in I_{s,n}$
                \State $\varepsilon_{s+1,n,\delta} \leftarrow (2+\delta)^{1/2} c_{s,n}n^{-1/2}(\log\log n)^{1/2}$ 
        \Until{$\varepsilon_{s+1,n,\delta} \geq \varepsilon_{s,n,\delta}$}

		\State \textbf{return} $\textsc{ConvexHull}(\{v_{s,j}: j\in I_{s,n}\})$
		\EndProcedure
	\end{algorithmic}
\end{algorithm}

Our main result in this setting is the following.

\begin{theorem}\label{thm:phylo-alg}
For any $\mu\in\Pcal(\Rbb^k/\Rbb\mathbf{1})$, the procedure $\textsc{AdaptRelaxFermatWeberSet}(Y_1,\ldots, Y_n)$ is a strongly $\dhaus$-consistent estimator of $F_1(\mu)$.
Moreover, the terminal relaxation rate $\varepsilon_{s,n,\delta}$  is asymptotically optimal in that it satisfies
\begin{equation*}
    \lim_{\delta\to 0}\lim_{n\to\infty}\varepsilon_{s,n,\delta}\sqrt{\frac{n}{\log \log n}} = \sigma_1(\mu) \quad \textnormal{a.s.}.
\end{equation*}
\end{theorem}

Notice that the preceding result does not make any integrability assumption about $\mu$.
This suggests that Fermat-Weber sets and their relaxations are desirable from the point of view of robust non-Euclidean statistics. 

It is also useful to mention some heuristic modifications which can be applied to Algorithm~\ref{alg:phylo-estimator}; while these modifications render our optimality and consistency proofs invalid, we find that they make no difference in practice.
One modification concerns line 12:
Instead of returning the set $\{v_j: j\in I_n\}$, one can return the set $\textsc{ExtremePoints}(\textsc{ConvexHull}(\{v_j: j\in I_n\}))$.
This shrinks the resulting set of vertices, which significantly improves the speed of the combinatorial optimization problem in line 22.
Another modification concerns line 23: We find that one can simply take $\delta = 0$.

While the full details of the proof of Theorem~\ref{thm:phylo-alg} are given in Appendix~\ref{sec:phylogenomic}, let us give some intuition for the computation.
The key is the following, which shows that, on the Fermat-Weber set of a distribution, distances to points from the support are affine functionals:

\begin{lemma}\label{lemma:linear_frechet_mean}
    For $\mu\in\Pcal(\Rbb^k/\Rbb\mathbf 1)$ and $z\in\supp(\mu)$, the map $\dtrop(\,\cdot\,, z):F_1(\mu)\to \R$ is affine.
\end{lemma}

Thanks to this result, we can prove that the definition of $\sigma_1(\mu)$ is exactly maximizing a convex function over a compact convex set, hence it can be exactly computed given the extreme points of $F_1(\mu)$.
For example, given the empirical Fr\'echet mean set $F_1(\bar\mu_n)$ and the set of its extreme points $E_n = \textsc{ExtremePoints}(F_1(\bar\mu_n))$, we have exactly
\begin{equation*}
    \sigma_1(\bar\mu_n) = \sqrt{\max_{v,v'\in E_n} \frac{2}{n}\sum_{i=1}^n  (\dtrop( v, Y_i)-\dtrop( v', Y_i))^2}.
\end{equation*}
We use a similar strategy to compute an asymptotically tight upper bound on $\sigma_1(\bar\mu_n,\varepsilon_{s,n,\delta})$.

In the next two subsections, we explore the empirical performance of relaxation-based Fermat-Weber sets on simulated (Subsection~\ref{subsec:toy}) and real (Subsection~\ref{subsec:flu}) data.

\subsection{Simulated data}\label{subsec:toy}

We consider the tropical projective space $\Rbb^3/\Rbb\mathbf{1}$.
Because of the quotient, every point $\Rbb^3/\Rbb\mathbf{1}$ can be represented by a point in $\R^2$, if we simply set the first coordinate to be equal to zero.
Hence, in what follows, we identify $\Rbb^3/\Rbb\mathbf{1}$ with $\Rbb^2$.

Now we consider the example given in \cite[Example~7]{YoshidaLin}, where the population distribution is
\begin{equation*}
    \mu = \frac{1}{3}\delta_{(0,0,0)} + \frac{1}{3}\delta_{(0,3,1)} + \frac{1}{3}\delta_{(0,2,5)},
\end{equation*}
and for which the Fermat-Weber set is the triangle with vertices $(0,1,1),(0,2,1),$ and $(0,2,2)$.
As before, we let $Y_1,Y_2,\ldots$ denote IID samples from $\mu$, and we consider the question of estimating $F_1(\mu)$ from the data $Y_1,\ldots, Y_n$ alone.

In Figure~\ref{fig:toy-ex}, we consider 3 possible estimators, computed on simulated data for number of samples equal to $n\in\{50,100,200,500,1000,2000\}$.
The first is the unrelaxed Fermat-Weber set.
The second is the relaxed Fr\'echet mean set estimator with optimal relaxation rate $2\cdot 3^{-1/2}\cdot n^{-1/2}(\log \log n)^{1/2}$; note that we can exactly compute the critical pre-factor to be $\sigma_1(\mu) = 2\cdot 3^{-1/2}$ since in this case we know the population distribution, but that in general we do not have access to this.
The third is the estimator \textsc{AdaptRelaxFermatWeberSet} from Algorithm~\ref{alg:phylo-estimator} which, by Theorem~\ref{thm:phylo-alg}, adaptively finds the asymptotically optimal relaxation rate.

\begin{figure}
    \centering
    \includegraphics[scale=0.45]{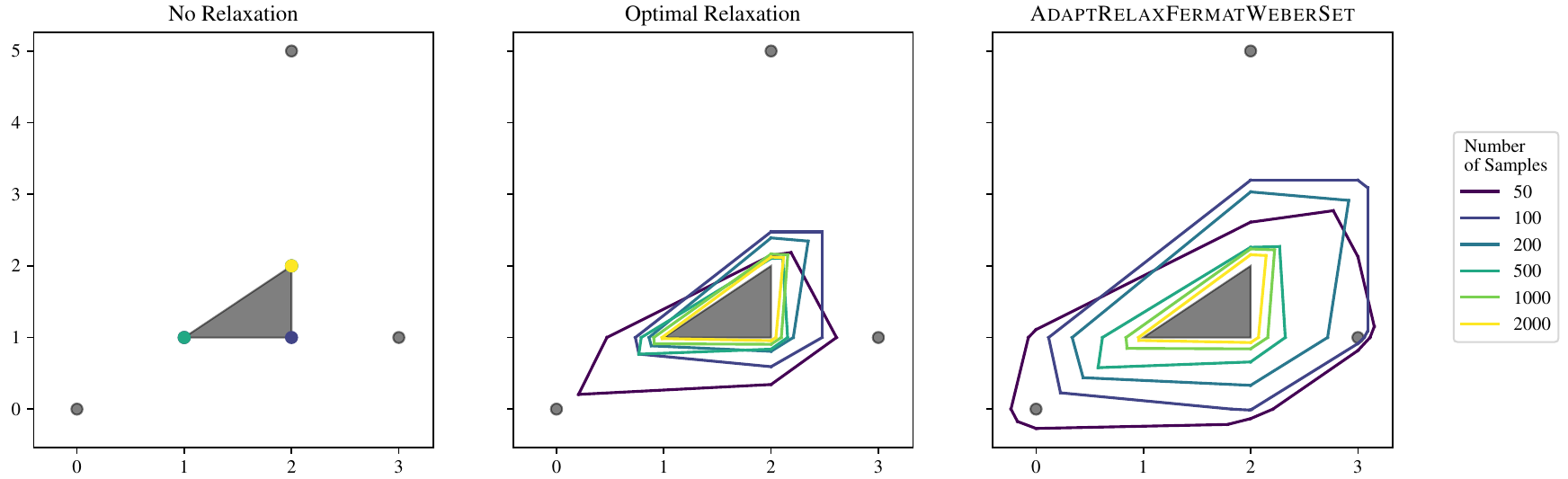}
    \caption{Comparison of Fermat-Weber set estimators for the toy example of Subsection~\ref{subsec:toy}. The unrelaxed Fermat-Weber set is typically a singleton (left), the optimally-relaxed Fermat-Weber set provides an accurate estimator (center), and the output of \textsc{AdaptRelaxFermatWeberSet} gives a conservative outer estimate (right).} 
    \label{fig:toy-ex}
\end{figure}

We make the following observations about this example. 
First, the empirical unrelaxed Fermat-Weber set is typically a single point: It jumps around the vertices of the population Fermat-Weber set, but it does not, for a fixed value of $n$, appear to give a reasonable estimator.
Second, we note that the empirical adaptively-relaxed Fermat-Weber set converges somewhat slowly to its population counterpart, when compared to the empirical optimally-relaxed Fermat-Weber set.
We found this behavior to be extremely stable across multiple trials of the same experiment.

\subsection{Influenza data}\label{subsec:flu}

In this subsection, we consider estimating the Fermat-Weber set of an unknown population distribution of evolutionary trees of hemagglutinin genome sequences.
We take the data set from \cite{TropPhylo}, which processes the publicly-available GI-SAID EpiFlu$^{\textnormal{TM}}$ data from 1995 in New York state.
We refer the reader to \cite[Section~6.3]{TropPhylo} for details of the pre-processing used to construct the data set.
For our purposes, the data consists of tens of thousands of phylogenetic trees, each of which has $4$ or $5$ leaves.

To describe the inferential problem more carefully, we consider the tropical projective space
\begin{equation*}
    \left(\Rbb^{{N\choose 2}}/\Rbb\mathbf{1},\dtrop\right).
\end{equation*}
By an \textit{equidistant $N$-leaf tree} we mean a distance matrix $w$ on the set $[N] := \{1,2,\ldots, N\}$ satisfying $w(i,k) \leq \max(w(i,j),w(j,k))$ for all $i,j,k\in[N]$. (See \cite[Lemma~4.3.6] {maclagan2021introduction} or \cite[Proposition~3]{TropPhylo}.)
By restricting to only, say, the lower triangle, we can identify the space of all equidistant $N$-leaf trees with a subset $\Ucal_N$ of the tropical projective space.
It is known \cite[Section~2]{BHV} that $\mathcal{U}_N$ is a union of $(2N-3)!! = (2N-3)(2N-5)\cdots 3\cdot 1$ orthants each of dimension $N-2$, and that each orthant corresponds to a unique binary tree topology.
Our data consist of samples $Y_1,\ldots, Y_n\in\Ucal_N$, and we assume that these are IID samples from an unknown population distribution $\mu$ on $\Ucal_N$.
Our goal is to estimate $F_1(\mu)$.

At this point, we can think of the Fermat-Weber set $F_1$ as being computed in either of $(\Rbb^{{N \choose 2}}/\Rbb\mathbf{1}, \dtrop)$ or $(\Ucal_N,\dtrop)$.
In general, Fr\'echet means do not behave well with respect to subspace restriction, but in this case the two problems are closely related.
To state a precise result, let us write $F_{1}^{S}(\mu,\varepsilon)$ for the $\varepsilon$-relaxed Fermat-Weber set, computed in the space $(S,\dtrop)$, where $S$ is an arbitrary closed subset of the tropical projective space.
Then we have the following:

\begin{theorem}\label{thm:FW-proj-trees}
    For $N\in\Nbb$ and $\mu\in\Pcal(\Rbb^{N \choose 2} / \Rbb \mathbf{1})$ with $\mu(\Ucal_N) = 1$, and for any $\varepsilon\geq 0$, we have
    \begin{equation*}
        F_{1}^{\Ucal_N}(\mu,\varepsilon) = F_{1}^{\Rbb^{N\choose 2} / \Rbb \mathbf{1}}(\mu,\varepsilon) \cap \Ucal_N.
    \end{equation*}
\end{theorem}

The preceding result shows that intersecting $F_{1}^{\Rbb^{N\choose 2} / \Rbb \mathbf{1}}$ with $\Ucal_N$ provides a way to calculate $F_{1}^{\Ucal_N}$.
In other words, we can implement a version of \textsc{AdaptRelaxFermatWeberSet} by intersecting with $\Ucal_N$ whenever necessary.
In the next parts, we compute the Fermat-Weber set of samples from the influenza data, using these two perspectives.

\subsubsection{Estimation in tropical projective space}

First we consider estimating the Fermat-Weber set of our tree data in $(\Rbb^{N\choose 2} / \Rbb \mathbf{1},\dtrop)$, the ambient tropical projective space.
For this part, we consider the 4-leaf data set, i.e., $N=4$, so the data naturally lies in dimension ${4\choose 2} = 6$; by setting a specified coordinate to 0, we can identify $\Rbb^{6}/\Rbb\mathbf{1}$ with $\R^5$.

We compute the output of both the unrelaxed Fermat-Weber set and the procedure \textsc{AdaptRelaxFermatWeberSet} for a number of data points $n\in\{5,10,20,30,40,50\}$, and the results are shown in Figure~\ref{fig:real_data_tropical_projective_space}.
In order to visualize these estimators, we plot the projection of the resulting 5-dimensional polytopes onto 2-dimensional subspaces chosen uniformly at random.
The results for 3 random subspaces are shown in Figure~\ref{fig:real_data_tropical_projective_space}.
We observe that the output of \textsc{AdaptRelaxFermatWeberSet} seems to provide a quite conservative outer estimate of $F_1(\mu)$ compared to the small (but somewhat unstable) unrelaxed Fermat-Weber set.

\begin{figure}[t]
    \centering
    \includegraphics[scale=0.5]{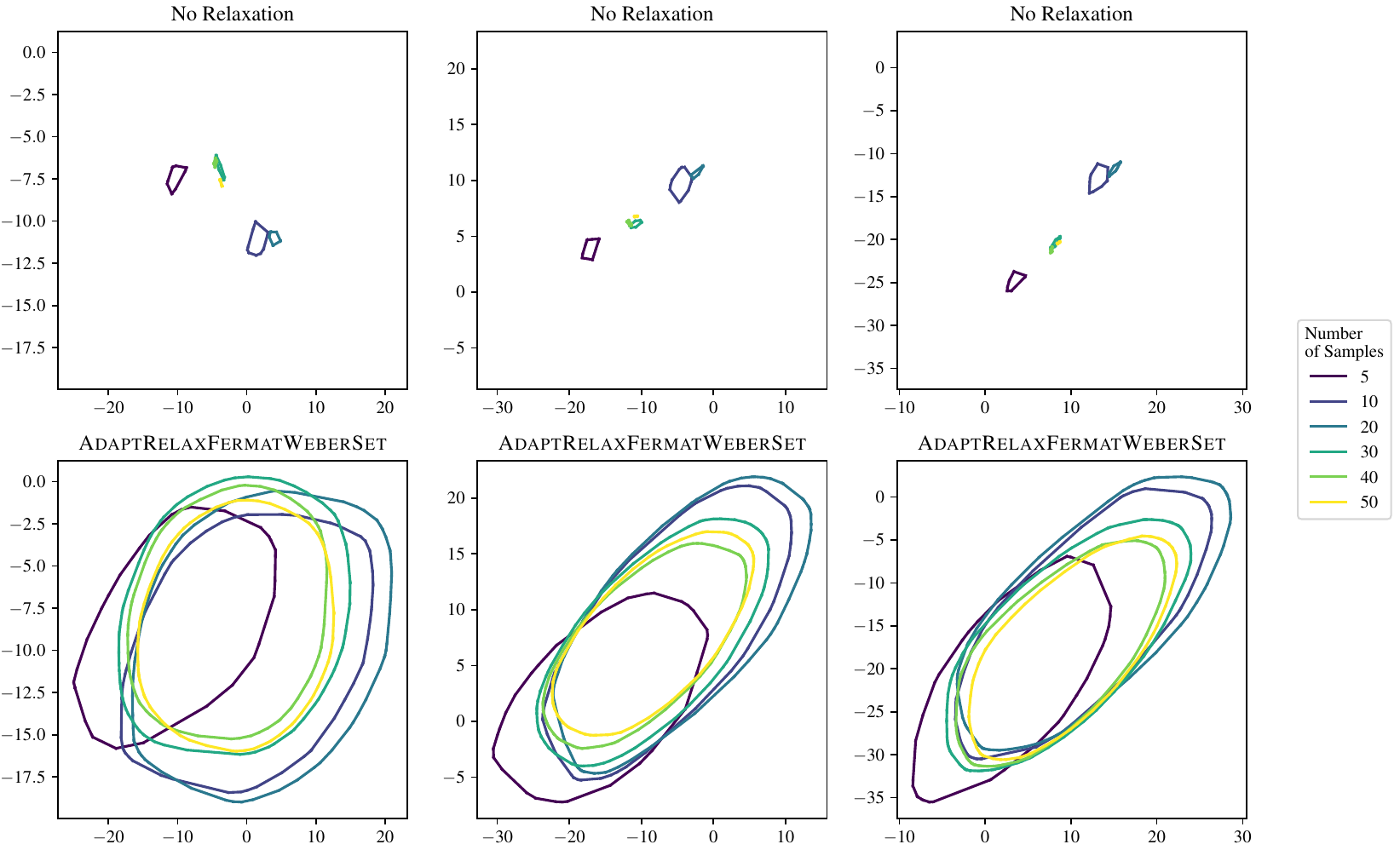}
    \caption{Comparison of Fermat-Weber set estimators on random samples from the 4-leaf influenza data of Subsection \ref{subsec:flu}. Because the estimates are $5$-dimensional polytopes, we plot the unrelaxed Fermat-Weber set (top) and the output of \textsc{AdaptRelaxFermatWeberSet} (bottom) using random 2-dimensional projections (one column per projection).}
    \label{fig:real_data_tropical_projective_space}
\end{figure}

\subsubsection{Estimation in equidistant tree space}\label{subsubsec:equi}

Next we consider estimating the Fermat-Weber set of our tree data in $(\Ucal_N,\dtrop)$, the space of equidistant trees.
For this part, we consider the 5-leaf data set.
This means that $N=5$, so we can think of our data as lying in a union of $(2\cdot 5-3)!! = 105$ orthants each of dimension $5-2 = 3$.

We compute the output of both the unrelaxed Fermat-Weber set and the procedure \textsc{AdaptRelaxFermatWeberSet} for $n=12$ data points.
For the unrelaxed Fermat-Weber set, we find that it contains a single equidistant tree.
For \textsc{AdaptRelaxFermatWeberSet}, we find that the output only intersects 3 orthants of $\Ucal_5$, which means that only 3 different tree topologies are represented in the estimated set;
in order to visualize this, we sample an extreme point uniformly at random from each polytope that results from intersecting the estimated set with each non-trivial topology.
The results can be seen in Figure~\ref{fig:my_label}.
While the estimator from \textsc{AdaptRelaxFermatWeberSet} can be more difficult to interpret, it robustly identifies a few notable qualitative features: 4 and 5 share the most recent common ancestor, 1 and 2 share the most distant common ancestor, and 3 lies somewhere in between.

\begin{figure}[t]
    \begin{tabular}{|c|c|}
        \hline
        \includegraphics[scale=0.5]{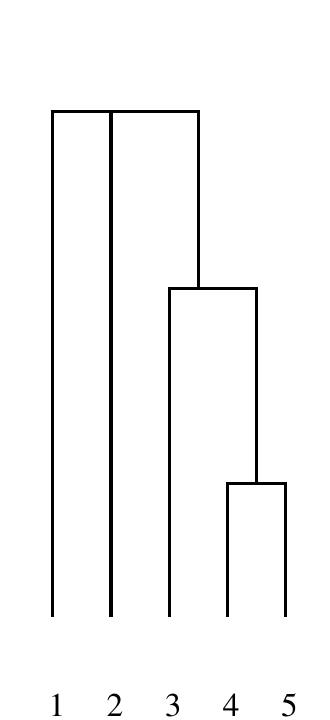} & 
        \includegraphics[scale=0.5]{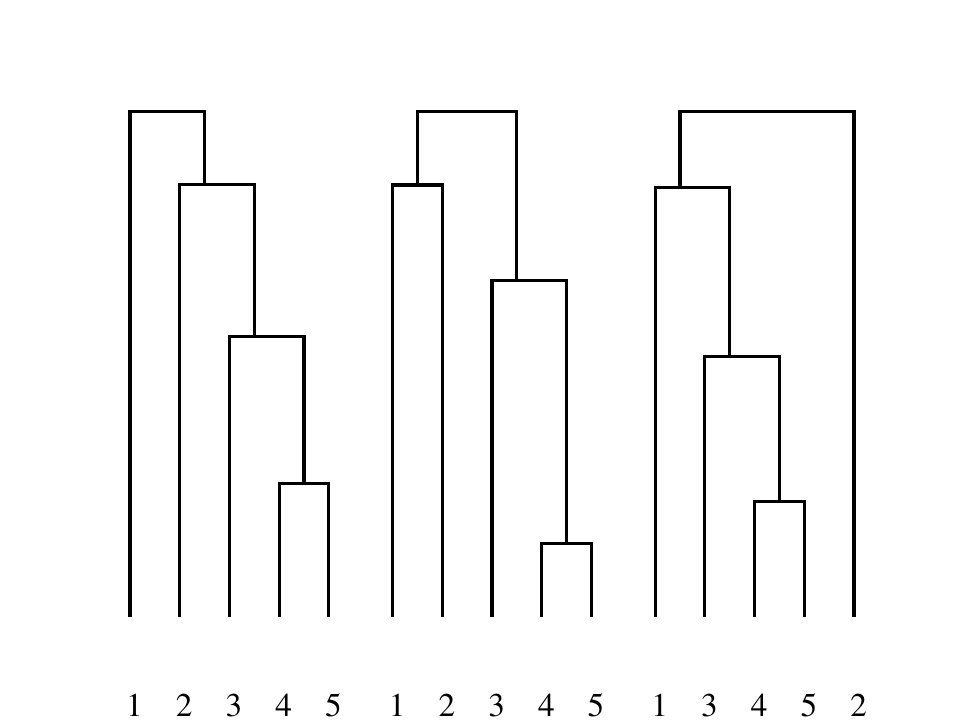} \\
        \hline
    \end{tabular}
    \caption{Estimates of the population tropical Fermat-Weber tree from the influenza data of Subsection~\ref{subsec:flu}. The unrelaxed Fermat-Weber point is a singleton (left). The estimate from \textsc{AdaptRelaxFermatWeberSet} intersects three different tree topologies, and we plot one extremal tree from each topology, sampled uniformly at random (right).}
    \label{fig:my_label}
\end{figure}

\subsection{Discussion}\label{subsec:discussion}

In this example of tropical projective space, we have seen that the abstract optimality result of Theorem~\ref{thm:adaptive-consistency} can be implemented as \textsc{AdaptRelaxFermatWeberSet} in Algorithm~\ref{alg:phylo-estimator} in order to estimate an unknown population Fermat-Weber set.
On simulated and real data, we have seen that this estimator provides a different view of the estimand compared to the unrelaxed empirical Fermat-Weber set.

Let us make some basic remarks.
On the positive side, we see that relaxation methods provide increased stability for the estimation problem.
This is contrasted with the unrelaxed empirical Fermat-Weber set, which can be highly sensitive to even a single data point.
On the negative side, it seems that even the asymptotically optimal procedure \textsc{AdaptRelaxFermatWeberSet} provides a very conservative outer estimate.
In this way, relaxation methods can be seen as a complement to, rather than a replacement of, unrelaxed estimators.

We also make some comments about computation.
For one, the main bottleneck of Algorithm~\ref{alg:phylo-estimator} is the combinatorial optimization problem in line 22; its time complexity is quadratic in the number of extreme points of the previous Fermat-Weber set estimator, and this number of extreme points appears to grow exponentially with the dimension of the problem.
For example, the computation of \textsc{AdaptRelaxFermatWeberSet} in Part~\ref{subsubsec:equi} above already has roughly 3,400 extreme points.
Another comment is that, in many settings one needs $n$ to be very large in order for the asymptotically adaptively optimal rate in $\Theta(n^{-1/2}(\log \log n)^{1/2})$ to become smaller than the asymptotically sub-optimal rate of $m_1(\bar \mu_n)n^{-1/2}\log\log n$.
For these reasons, we believe that a useful heuristic in practice is just to compute the relaxed empirical Fermat-Weber set for the relaxation rate $m_1(\bar \mu_n)n^{-1/2}\log\log n$.

\section{Proof of results}\label{sec:proofs}

In this section, we give some proofs for results stated above, as well as some additional results.
More specifically, Subsection~\ref{subsec:defs} gives the precise definitions of our basic objects, Subsection~\ref{subsec:prob} studies some basic probabilistic preliminaries, Subsection~\ref{subsec:SLLN} provides a strengthened form of the one-sided SLLN, and Subsection~\ref{subsec:main_proofs} gives the proofs of some of our main results.


\subsection{Basic definitions}\label{subsec:defs}

In this subsection we precisely describe the setting of our work.
Most of our notation is taken from \cite{EvansJaffeSLLN}, and some of these concepts have already been defined in Section~\ref{sec:intro}. 
We assume throughout that $(X,d)$ is a metric space.
By a \textit{Heine-Borel (HB)} space we mean that $(X,d)$ is a metric space and that the $d$-closed balls $B_r(x):= \{y\in X: d(x,y)\le r\}$ are compact for all $x\in X$ and $r\ge 0$.
We also write $B^{\circ}_r(x):= \{y\in X: d(x,y)< r\}$ for the $d$-open balls for all $x\in X$ and $r\ge 0$. Next, we write $\closed(X)$ for the collection of $d$-closed (ferm\'e) subsets of $X$ and $C(X)$ for the collection of continuous functions on $X$.
We also recall that $\CovNum_K(\varepsilon)$ denotes the $\varepsilon$-covering number of a compact subset $K\subseteq X$, and we recall the notions of \emph{Dudley space} and \emph{Heine-Borel-Dudley (HBD)} space given in Definition~\ref{def:HBD}.

We write $\Pcal(X)$ for the set of Borel probability measures on $X$, where $X$ is given the topology metrized by $d$, and we write $\tau_w$ for the weak topology on $\Pcal(X)$.
Now let $q\ge 0$.
Recall \cite[Equation~(8)]{EvansJaffeSLLN} that the constant $c_q := \max\{1,2^{q-1}\}$ is such that we have
\begin{equation}\label{eqn:ineq-1}
	d^q(x,y)\le c_q(d^q(x,x')+d^q(x',y))
\end{equation}
for all $x,x',y\in X$.
In particular, a probability measure $\mu\in \Pcal(X)$ satisfies $\int_{X}d^q(x,y)\,d\mu(y) < \infty$ for all $x\in X$ if and only if it satisfies $\int_{X}d^q(x,y)\,d\mu(y) < \infty$ for some $x\in X$;
we write $\Pcal_q(X)\subseteq \Pcal(X)$ for the set of $\mu\in \Pcal(X)$ with this integrability condition.
We write $\tau_w^q$ for the topology on $\Pcal_q(X)$ where $\{\mu_n\}_{n\in\Nbb}$ and $\mu$ in $\Pcal_q(X)$ have $\mu_n\to \mu$ in $\tau_w^q$ if and only if we have both $\mu_n\to \mu$ in $\tau_w$ and $\int_{X}d^q(x,y)\, d\mu_n(y) \to\int_{X}d^q(x,y)\, d\mu(y)$ for some (hence, all \cite[Lemma~2.1]{EvansJaffeSLLN}) $x\in X$.
It is known that $\tau_w^p$ for $p\ge 1$ is metrized by the $p$-Wasserstein metric.

Next, we need to define slightly generalized notions of relaxed and restricted Fr\'echet means.
Indeed, let us define, for $\mu\in \Pcal_{p-1}(X), C\in \closed(X)$, and $\varepsilon\ge 0$, the set
\begin{equation*}
	F_p(\mu,C, \varepsilon) := \{x\in C: W_p(\mu,x,x')\le \varepsilon\textrm{ for all } x'\in C\}.
\end{equation*}
For convenience, we use the notation $F_p(\mu,\varepsilon):=F_p(\mu,X,  \varepsilon)$ for the relaxed (unrestricted) Fr\'echet means. In particular, $F_p(\mu,0) = F_p(\mu)$.
We also note that $F_p^{\ast}(\mu):= F_p(\mu,\supp(\mu),0)$ is an important object of study (see \cite{LinearMedoids,UltraFastMedoids,EvansJaffeSLLN}) often called the \emph{$p$-medoid} of $\mu$.

It is worth mentioning that the notions of Fr\'echet means become slightly simplified when one assumes $\mu\in \Pcal_{p}(X)$ instead of merely $\mu\in \Pcal_{p-1}(X)$.
For instance, one can then define the univariate Fr\'echet functional $W_p:\mathcal{P}_{p}(X)\times X\to \R$ via
\begin{equation*}
	W_p(\mu,x) := \int_{X}d^p(x,y)\,d\mu(y),
\end{equation*}
and it follows that we have the identity $W_p(\mu,x,x') = W_p(\mu,x)-W_p(\mu,x')$ for all $\mu\in \Pcal_p(X)$ and $x,x'\in X$.
Consequently, the Fr\'echet mean set can be equivalently written
\begin{equation*}
	F_p(\mu) = \{x\in X: W_p(\mu,x)\le W_p(\mu,x')\textrm{ for all } x'\in X\}
\end{equation*}
when we have $\mu\in \Pcal_p(X)$; in this case, we also write $m_p(\mu) := \inf_{x\in X}W_p(\mu,x)$.
However, in light of the classical strong law of large numbers (whereby a 1st moment assumption yields strong convergence of the Fr\'echet 2-mean), it is natural to work only under the assumption $\mu\in \Pcal_{p-1}(X)$.

For any set $K\subseteq X$ and $\delta\geq 0$, we denote by $K^\delta:=\bigcup_{x\in K} B_\delta(x)$ the (closed) $\delta$-thickening of $K$.
Observe that if $(X,d)$ is a Heine-Borel space and $K\subseteq X$ is compact, then $K^{\delta}$ is compact for all $\delta \ge 0$.

Last, we adopt standard Bachman-Landau asymptotic notations. For two sequences $(u_n)_{n\in\Nbb}$ and $(v_n)_{n\in\Nbb}$, we write $u_n=O(n)$ if there exists $M\geq 0$ and an index $n_0$ such that for all $n\geq n_0$, $|u_n|\leq M v_n$. We write $u_n=\Theta(v_n)$ if $u_n=O(v_n)$ and $v_n=O(u_n)$. We also say that $u_n=o(v_n)$ if for any $\epsilon>0$ there exists $n_\epsilon$ such that $|u_n|\leq \epsilon v_n$ for all $n\geq n_\epsilon$, and that $u_n=\omega(v_n)$ if for any $M\geq 0$ there exists $n_M$ such that $u_n>M v_n$ for all $n\geq n_M$.

\subsection{Probabilistic setting}\label{subsec:prob}

In this subsection we study the probabilistic results that underly our main theorems. All proofs are given in Appendix~\ref{sec:probabilistic_setting}.
In general, these are functional limit theorems which quantify the rate of convergence of the empirical Fr\'echet functional to the population Fr\'echet functional. Throughout this subsection, let $(X,d)$ be a Heine-Borel space and fix $p\ge 1$.
For $\mu\in \Pcal_{p-1}(X)$, let $(\Omega,\F,\P_{\mu})$ be a probability space on which we have an IID sequence $Y_1,Y_2,\ldots$ each with law $\mu$; write $\E_{\mu}$, $\Var_{\mu}$, and $\Cov_{\mu}$ for the expectation, variance, and covariance on this space, respectively. From these elements we define the random functions $Z_{i}:X^2\to\R$ via 
\begin{equation*}
	Z_{i}(x,x') := d^p(x,Y_i)-d^p(x',Y_i) - \Ebb_\mu [d^p(x,Y_i)-d^p(x',Y_i)] 
	= d^p(x,Y_i)-d^p(x',Y_i) - W_p(\mu,x,x')
\end{equation*}
for $i\in\Nbb$. Last, if $\mu\in\Pcal_{2p-2}(X)$, we define the function $R_{\mu}:X^2\times X^2\to \R$ via
\begin{equation*}
	R_{\mu}(x,x',x'',x''') := \Cov(Z_{1}(x,x'),Z_{1}(x'',x''')).
\end{equation*}
For simplicity, we also write $R_{\mu}(x,x') := R_{\mu}(x,x',x,x') = \Var(Z_{1}(x,x'))$. We prove in Appendix~\ref{sec:probabilistic_setting} that these quantities are well-defined. We first develop some continuity results for these functions.

\begin{lemma}\label{lemma:continuity-Z-R}
    Let $(X,d)$ be a metric space, $p\ge 1$, and $\mu\in \Pcal_{p-1}(X)$ (resp. $\mu\in \Pcal_{2p-2}(X)$).
	Then, the function $Z_i:X^2\to \R$ (resp. $R_{\mu}:K^4\to\R$) is continuous for the metric $D((x,x'),(x'',x''')) = d(x,x'')+d(x',x''')$ (resp. $D((x_1,x_1',x_1'',x_1'''), (x_2,x_2',x_2'',x_2''')) = d(x_1,x_2) + d(x_1',x_2') + d(x_1'',x_2'') + d(x_1''',x_2''') $). 
\end{lemma}

The proof is given in Appendix~\ref{sec:probabilistic_setting}.
Next, we prove a uniform convergence result from $R_{\bar\mu_n}$ to $R_\mu$ on compacts.
\begin{lemma}\label{lemma:uniform_convergence_covariance}
	Let $(X,d)$ be a HB space, $p\geq 1$ and $\mu\in \Pcal_{2p-2}(X)$. Then, for any compact set $K\subseteq X$, the functions $R_{\bar\mu_n}:K^2\to \Rbb$ and $R_\mu:K^2\to\Rbb$ satisfy $\|R_{\bar\mu_n} - R_\mu\|_\infty \to 0 \; (a.s.)$.
\end{lemma}
The proof is given in Appendix~\ref{sec:probabilistic_setting}.
We now give limit theorems for the random continuous functions $G_n:X^2\to\R$ defined as
\begin{equation*}
	G_{n} = \sum_{i=1}^n Z_{i},
\end{equation*}
for $n\in\Nbb$.	Of course, classical limit theorems allow us to analyze the convergence  of the sequence $\{G_n(x,x')\}_{n\in\Nbb}$ for any fixed $x,x'\in X$.
However, for our later purposes, this will not be enough; we need sufficiently powerful functional limit theorems which allow us to analyze the convergence of the sequence $\{G_n\}_{n\in\Nbb}$ across the whole domain $X^2$ simultaneously.

We begin with the functional central limit theorem.
Because of the integrability assumption $\mu\in \Pcal_{2p-2}(X)$, the function $R_{\mu}$ is a positive semi-definite kernel on the space $X^2$.
We then denote by $\Gcal_{\mu}$ the mean-zero Gaussian measure on $C(X^2)$ whose covariance function is $R_{\mu}$.
We get the following:

\begin{proposition}\label{prop:functional_CLT}
	Let $(X,d)$ be a HBD space, $p\ge 1$, and $\mu\in\Pcal_{2p-2}(X)$.
	Then the random functions $\{n^{-1/2}G_{n}\}_{n\in\Nbb}$ converge in distribution to $\Gcal_{\mu}$ with respect to the topology of uniform convergence on compact sets.
\end{proposition}

Next we study the functional law of the iterated logarithm (fLIL), which requires some notation.
For each compact subset $K\subseteq X$ and $\mu\in \Pcal_{2p-2}(X)$, we write $R_{K,\mu}$ and $G_{K,n}$ for the restrictions $R_{K,\mu}:=R_{\mu}|_{K\times K}$ and $G_{K,n}:=G_{n}|_{K\times K}$, respectively, and we let $\Hcal_{K,\mu}\subseteq C(K\times K)$ denote the reproducing kernel Hilbert space (RKHS) with kernel $R_{K,\mu}$; we write $\|\cdot\|_{K,\mu}$ and $\langle\cdot,\cdot\rangle_{K,\mu}$ for the norm and inner product of $\Hcal_{K,\mu}$, and write $\Bcal_{K,\mu} := \{f\in \mathcal{H}_{K,\mu}: \|f\|_{K,\mu}\le 1\}\subseteq C(K\times K)$ for the unit ball of $\mathcal{H}_{K,\mu}$.

\begin{proposition}\label{prop:functional_LIL}
	Let $(X,d)$ be a HBD space and $p\ge 1$, and $\mu\in\Pcal_{2p-2}(X)$.
	Then for each compact set $K\subseteq X$, the random functions $\{(2n\log\log n)^{-1/2}G_{K,n}\}_{n\in\Nbb}$ form a relatively compact set with closure $\Bcal_{K,\mu}$, 
	with respect to the topology of uniform convergence on $K$.
\end{proposition}

It turns out that our critical pre-factor $\sigma_p(\mu)$ introduced in \eqref{eqn:sigma-p} is closely related to the RKHS corresponding to $R_{\mu}$.
The following result makes this relationship precise:

\begin{lemma}\label{lemma:formula_cp(mu)}
	Let $(X,d)$ be a metric space, $p\geq 1$ and $\mu\in\Pcal_{2p-2}(X)$. Then, for any compact $K\subseteq X$ with $F_p(\mu)\subseteq K$,
	\begin{equation*}
		\sigma_p(\mu) = \sqrt{2} \sup_{f\in \Bcal_{K,\mu}} \sup_{x,x'\in F_p(\mu)} f(x,x').
	\end{equation*}
\end{lemma}

The proof is given in Appendix~\ref{sec:probabilistic_setting}.
Last, we provide the following, which is an asymptotic version of Lemma~\ref{lemma:formula_cp(mu)}, which proof is also given in Appendix~\ref{sec:probabilistic_setting}.

\begin{lemma}\label{lemma:sigma_continuity}
	Let $(X,d)$ be a metric space, $p\ge 1$, and $\mu\in\Pcal_{2p-2}(X)$.
	Suppose also that we have compact subsets $\{K_{\delta}\}_{0 <\delta \le 1}$ of $X$ with $\bigcap_{0 < \delta \le 1}K_{\delta} = F_p(\mu)$.
	Then
	\begin{equation*}
		\sqrt 2 \cdot \sup_{f\in \Bcal_{K_1,\mu}}\left\{ \max_{x\in F_p(\mu)} \max_{x'\in K_\delta} f(x,x') \right\} \underset{\delta\to 0}{\longrightarrow} \sigma_p(\mu).
	\end{equation*}
\end{lemma}

\subsection{Strong law of large numbers}\label{subsec:SLLN}

In our study of relaxed Fr\'echet means with adaptive relaxation rates, we need one-sided SLLN for relaxed constrained Fr\'echet means. All proofs of this subsection are given in Appendix~\ref{sec:SLLN}. As before, let $(X,d)$ be a metric space, fix $\mu\in \mathcal{P}(X)$, and let $(\Omega,\F,\P_{\mu})$ be a probability space on which is defined a sequence $Y_1,Y_2,\ldots$ of IID random variables with common distribution $\mu$.
Then we first show the compactness of relaxed constrained Fr\'echet means.

\begin{lemma}\label{lemma:Fp-cpt}
	If $(X,d)$ is a HB space, $p\ge 1$, and $(\mu,C,\varepsilon) \in \Pcal_{p-1}(X)\times \closed(X)\times [0,\infty)$, the set $F_p(\mu,C,\varepsilon)$ is non-empty and compact.
\end{lemma}
The proof is given in Appendix~\ref{sec:SLLN}.
Next, we have the following.

\begin{theorem}\label{thm:mean-SLLN}
	Let $(X,d)$ be a HB space, $p\ge 1$, and $\mu\in\Pcal_{p-1}(X)$.
	Then, any random relaxation rate with $\varepsilon_n\to 0$ satisfies $\dvechaus(F_p(\bar \mu_n,\varepsilon_n),F_p(\mu))\to 0$ almost surely.
\end{theorem}

We will need Theorem~\ref{thm:mean-SLLN} in our later study of relaxed Fr\'echet means with adaptive relaxation rates. However, in Appendix~\ref{sec:SLLN} (Theorem C.2), we give a more general version of this result for restricted and relaxed Fr\'echet mean sets, which is of independent interest. From this general result, we can deduce a strong consistency result for medoids, which resolves the open question of \cite[Remark~2.11]{EvansJaffeSLLN}.
By the \textit{$p$-medoid set} of $\mu$ we mean the set $F_p^{\ast}(\mu) := F_p(\mu,\supp(\mu),0)$.
Then we have the following, which generalizes the result of \cite{EvansJaffeSLLN}:

\begin{theorem}\label{thm:medoid-SLLN}
	Let $(X,d)$ be a HB space, $p\ge 1$, and $\mu\in\Pcal_{p-1}(X)$.
	Then, we have almost surely $	\dvechaus(F_p^{\ast}(\bar \mu_n),F_p^{\ast}(\mu))\to 0$.
\end{theorem}

\subsection{Proofs of main general results}
\label{subsec:main_proofs}

We start by showing that $\sigma_p(\mu)$ is finite whenever $\mu\in\Pcal_{2p-2}(X)$.

\begin{lemma}\label{lemma:finite_sigma}
	Let $X$ be a HB space, $p\geq 1$, and $\mu\in\Pcal_{2p-2}(X)$.
	Then, $\sigma_p(\mu)<\infty$.
\end{lemma}

The proof is given in Appendix~\ref{sec:proof_general_results}. Next, we discuss conditions under which $\sigma_p(\mu)$ is positive.
First of all, observe that $\sigma_p(\mu) = 0$ if and only if $d(x,Y_1) = d(x',Y_1)$ holds almost surely for all $x,x'\in F_p(\mu)$.
However, this condition is impossible to check if $F_p(\mu)$ is not already known; thus, we provide the following sufficient condition:

\begin{lemma}\label{lem:sigma_pos}
	Let $(X,d)$ be a metric space.
	If $\mu\in \Pcal_{p-1}(X)$ with $\supp(\mu) = X$, then either $\#F_p(\mu) = 1$ or $\sigma_p(\mu) > 0$.
\end{lemma}

The proof is given in Appendix~\ref{sec:proof_general_results}. Now we can prove our main results.
We also defer to Appendix~\ref{sec:SLLN} the proof of Theorem \ref{thm:gaussian_tail_one_sided} that gives a description of what happens in the regime of $\Omega(n^{-1/2})$ relaxation. Complementing this result, we also have the following result aiming to characterize weak $\dhaus$-consistency. The proof is given in Appendix~\ref{sec:proof_general_results}.

\begin{theorem}\label{thm:weak_convergence}
	Let $(X,d)$ be a HBD space and $p\geq 1$, and consider any random relaxation rate $\varepsilon_n$.
	Then for every $\mu\in \Pcal_{2p-2}(X)$, we have:
	\begin{itemize}
		\item[(i)] If $\varepsilon_n\in o_p(1)$ and $\varepsilon_n\in \omega_p(n^{-1/2})$, then $F_p(\bar\mu_n,\varepsilon_n)$ is weakly $\dhaus$-consistent.
		\item[(ii)] If $\varepsilon_n$ is deterministic, $\sigma_p(\mu) > 0$, and $F_p(\bar \mu_n,\varepsilon_n)$ is weakly $\dhaus$-consistent, then $\varepsilon_n\in\omega(n^{-1/2})$.
	\end{itemize}
\end{theorem}

We now argue that the requirements of the previous theorem cannot be dropped. We start by the condition $\sigma_p(\mu) > 0$. A simple example is when the Fr\'echet mean set is a singleton $F_p(\mu) = \{\bar x\}$, so $\varepsilon_n \equiv 0$ suffices for strong (hence weak) consistency.
Indeed, in this case, for any bounded subset $S\subseteq X$,
\begin{equation*}
	\dvechaus(F_p(\mu),S) = d(\bar x, S) \leq \dvechaus(S,F_p(\mu)).
\end{equation*}
Therefore, the strong one-sided convergence $\dvechaus(F_p(\bar\mu,\varepsilon_n),F_p(\mu))\to 0$ a.s., which holds whenever $\varepsilon_n\to 0$, is sufficient for strong Hausdorff convergence.

Unfortunately, when $\sigma_p(\mu)=0$, the condition $\varepsilon_n\to 0$ may not be sufficient either. In the next result, we give an explicit example where some rate in $o(n^{-1/2})$ gives strong (hence weak) consistency.
For the sake of simplicity, we focus on deterministic relaxation rates.

\begin{proposition}\label{prop:example_square}
	Let $X = \{(\cos \theta,\sin\theta):\theta\in\{0\} \cup [\pi/2,3\pi/2]\}$ be equipped with the geodesic distance $d((\cos \theta,\sin\theta),(\cos \theta',\sin\theta')) =  \frac{2}{\pi}\min_{k\in\Zbb}|\theta-\theta'-2k\pi|$, let $p> 1$, and consider the distribution $\mu = \frac{1}{2}\delta_{(0,1)}+ \frac{1}{2}\delta_{(0,-1)}$.
	Then $(X,d)$ is a Dudley space, $\sigma_p(\mu) = 0$, and for any deterministic relaxation rate $\varepsilon_n$ we have the following:
	\begin{enumerate}
		\item[(W)] $F_p(\bar\mu_n,\varepsilon_n)$ is weakly $\dhaus$-consistent if and only if $\varepsilon_n = o(1)\cap  \omega(n^{-1})$.
		\item[(S)] $F_p(\bar \mu_n,\varepsilon_n)$ is strongly $\dhaus$-consistent if $\liminf_{n\to\infty}\varepsilon_n n/\log \log n > 2p/(p-1)$ and not strongly $\dhaus$-consistent if $\limsup_{n\to\infty}\varepsilon_n n/\log \log n < 2p/(p-1)$.
	\end{enumerate}
\end{proposition}

The proof is given in Appendix~\ref{sec:proof_general_results}. As a last remark, in Theorem \ref{thm:weak_convergence}, the condition $\varepsilon_n = o(1)$ may not be necessary for weak convergence. This condition is unnecessary when the functional $W_p(\mu,\cdot)$ has a gap from its minimum value $m_p(\mu)$ to non-optimal values. For instance, consider the simple case of a discrete finite space $X$. Then, letting $\delta^* = \min_{x\notin F_p(\mu)}W_p(\mu,x) - m_p(\mu) > 0$, the usual strong law of large number readily implies that $F_p(\bar\mu_n,\delta)$ is weakly and strongly $\dhaus$-consistent if $\delta<\delta^*$ and not weakly nor strongly $\dhaus$-consistent if $\delta>\delta^*$. The following result gives a more precise description of the frontier for weakly and strongly consistent relaxation rates in the simple case of a two-point space and non-uniform measure.

\begin{proposition}\label{prop:ex-two-point}
	Let $X=\{x_1,x_2\}$ be equipped with the discrete metric $d(x,y)=\mathbf{1}[x=y]$, let $p\geq 1$, and consider the distribution $\mu = q \delta_{x_1}+ (1-q)\delta_{x_2}$ for $q\in(0,\frac{1}{2})\cup (\frac{1}{2},1)$.
	Then $(X,d)$ is a Dudley space, $\sigma_p(\mu) = 0$, and for any deterministic relaxation rate $\varepsilon_n$ we have the following:
	\begin{enumerate}
		\item[(W)] $F_p(\bar\mu_n,\varepsilon_n)$ is weakly $\dhaus$-consistent if and only if $\varepsilon_n = |1-2q| -  \omega(n^{-1/2})$.
		\item[(S)] $F_p(\bar \mu_n,\varepsilon_n)$ is strongly $\dhaus$-consistent if one has $\liminf_{n\to\infty}(|1-2q|-\varepsilon_n)n^{1/2}(\log \log n)^{-1/2} > 2\sqrt{2q(1-q)}$, and it is not strongly $\dhaus$-consistent (in fact, it is strongly $\dhaus$-inconsistent) if $\limsup_{n\to\infty}(|1-2q|-\varepsilon_n)n^{1/2}(\log \log n)^{-1/2} < 2\sqrt{2q(1-q)}$.
	\end{enumerate}
\end{proposition}

The proof is given in Appendix~\ref{sec:proof_general_results}. We are now ready to prove the strong consistency result, Theorem \ref{thm:consistency}.



\begin{proof}[Proof of Theorem \ref{thm:consistency}]
The key observation in both cases is that our choice of relaxation parameter, along with some simple arithmetic, shows that we have
\begin{equation*}
	\{W_p(\bar \mu_n,x,x')\le \varepsilon_n\} = \left\{\frac{G_n(x,x')}{\sqrt{n \log\log n}} \le \sqrt\frac{n}{\log\log n}(\varepsilon_n -W_p(\mu,x,x'))\right\}
\end{equation*}
for all $x,x'\in X$ and $n\in\Nbb$. In particular, if $x'\in F_p(\mu)$, then
\begin{equation}\label{eqn:key2}
	\left\{(n \log\log n)^{-1/2}G_n(x,x') > n^{1/2}(\log\log n)^{-1/2}\varepsilon_n \right\} \subseteq \{W_p(\bar \mu_n,x,x')> \varepsilon_n\}.
\end{equation}
Conversely, if $x\in F_p(\mu)$, then
\begin{equation}\label{eqn:key3}
	\left\{(n \log\log n)^{-1/2}G_n(x,x')\le  n^{1/2}(\log\log n)^{-1/2}\varepsilon_n \right\} \subseteq  \{W_p(\bar \mu_n,x,x')\le \varepsilon_n\}
\end{equation}
for all $x'\in X$.
Now we proceed to the main proof.

We begin with the easier statement (ii). Suppose that the relaxation parameters $\varepsilon_n$ satisfy the asymptotics $ \limsup_{n\to\infty}\varepsilon_n n^{1/2} (\log\log n)^{-1/2} < \sigma_p(\mu)$ almost surely.
In particular, for any $\delta>0$, there exists $\eta>0$ such that with probability at least $1-\delta$,
\begin{equation*}
	\limsup_{n\to\infty}\varepsilon_n\sqrt{\frac{n}{\log\log n}} < \sigma_p(\mu) - \eta.
\end{equation*}
Denote by $E$ this event. For simplicity, we will use the notation $K^s =(F_p(\mu))^s$ for the $s$-thickening of $F_p(\mu)$. These are compact by the Heine-Borel property of $X$ and because $F_p(\mu)$ is compact. By Lemma \ref{lemma:formula_cp(mu)}, there exists $f\in \Bcal_{K^1,\mu}$, and $x,x'\in F_p(\mu)$ such that $\sqrt 2\cdot f(x,x')>\sigma_p(\mu) - \eta/3$.
Because $f$ is continuous, there exists $0<r<1$ such that the closed ball $B_r(x)$ satisfies $\sqrt 2\cdot f(z,x') > \sigma_p(\mu) - 2\eta/3$ for all $z\in B_r(x)$. Note that we selected $r<1$ so that $B_r(x)\subseteq K^1$. Now by Proposition~\ref{prop:functional_LIL}, there exists an event $F$ with $\P_{\mu}(F) = 1$, on which the following is true:
There is a (random) subsequence $\{n_k\}_{k\in\Nbb}$ such that we have $(2n_k\log\log n_k)^{-1/2}G_{K^1,n_k}\to f$ in the topology of uniform norm convergence on $K^1$. 
Also recall that $G_{K^1,n}$ is just the restriction of $G_n$ to $K^1\times K^1$.
Therefore,
\begin{align*}
	\limsup_{n\to\infty} \min_{z\in B_r(x)}(n\log\log n)^{-1/2}G_n(z,x') 
	&\ge\liminf_{k\to\infty} \min_{z\in B_r(x)}(n_k\log\log n_k)^{-1/2}G_{n_k}(z,x') \\
	&\ge \min_{z\in B_r(x)}\sqrt 2\cdot f(z,x') \ge \sigma_p(\mu) - 2\eta/3 > \sigma_p(\mu) - \eta
\end{align*}
on $F$.
We combine this with \eqref{eqn:key2} to see that on $E\cap F$, $
\min_{z\in B_r(x)}W_p(\bar \mu_n,z,x')>\varepsilon_n $ for infinitely many $n\in\Nbb$.
Also, we have
\begin{equation*}
	\left\{\min_{z\in B_r(x)}W_p(\bar \mu_n,z,x')>\varepsilon_n\right\} 
	\subseteq \{B_r(x)\cap F_p(\bar\mu_n,\varepsilon_n)=\emptyset\}
	\subseteq \{\dhaus(F_p(\bar \mu_n,\varepsilon_n),F_p(\mu)) \ge r\}.
\end{equation*}
Therefore, on $E\cap F$, we have $   \limsup_{n\to\infty}\dhaus(F_p(\bar \mu_n,\varepsilon_n),F_p(\mu)) \ge r.$ Since $\P_{\mu}(E\cap F) \geq 1-\delta $, this shows $\P_{\mu}(\dvechaus(F_p(\bar \mu_n,\varepsilon_n),F_p(\mu)) \to 0) \leq \delta.$	This holds for any $\delta>0$. Therefore, we proved $\P_{\mu}(\dvechaus(F_p(\bar \mu_n,\varepsilon_n),F_p(\mu)) \to 0) =0$.
As a result, $F_p(\bar \mu_n,\varepsilon_n)$ is not strongly $\dhaus$-consistent.

For the harder statement (i), we suppose that almost surely, $\varepsilon_n\to 0$ and
\begin{equation*}
	\liminf_{n\to\infty}\varepsilon_n \sqrt{\frac{n}{\log\log n}} > \sigma_p(\mu)\quad a.s.
\end{equation*}
First, by Theorem~\ref{thm:mean-SLLN}, since almost surely $\varepsilon_n \to 0$, the event $E_1=\{\dvechaus(F_p(\bar\mu_n),F_p(\mu))\to 0\}$ has probability one. Now fix $\delta>0$. By hypothesis on the parameters $\varepsilon_n$, there exists $\eta>0$ such that with probability at least $1-\delta$,
\begin{equation*}
	\liminf_{n\to\infty}\varepsilon_n \sqrt{\frac{n}{\log\log n}} > \sigma_p(\mu) + \eta.
\end{equation*}
Denote by $F_1$ this event. Now by Lemma \ref{lemma:sigma_continuity} we have
\begin{equation*}
	\sqrt 2\cdot \sup_{f\in \Bcal_{K^1,\mu}}\left\{ \max_{x\in F_p(\mu)} \max_{x'\in K^s} f(x,x') \right\} \underset{s\to 0}{\longrightarrow} \sigma_p(\mu).
\end{equation*}
In particular, there exists $0<s<1$ such that
\begin{equation*}
	\sqrt 2\cdot 	\sup_{f\in \Bcal_{K^1,\mu}}\left\{ \max_{x\in F_p(\mu)} \max_{x'\in K^s} f(x,x') \right\} \leq  \sigma_p(\mu) + \frac{\eta}{2}.
\end{equation*}
By Proposition~\ref{prop:functional_LIL}, on an event $E_2$ of full probability, $\{(2n\log\log n)^{-1/2}G_{K^1,n}\}$ is relatively compact and its set of limits is exactly $\Bcal_{K^1,\mu}$. Therefore, and since $F_p(\mu)$ and $K^s$ are compact, there exists $f\in\Bcal_{K^1,\mu}$ such that
\begin{align*}
	\limsup_{n\to\infty}\max_{x\in F_p(\mu)} \max_{x'\in K^s} \frac{G_n(x,x')}{\sqrt{n\log\log n}} &= \sqrt 2\cdot  \max_{x\in F_p(\mu)} \max_{x'\in K^s} f(x,x')\\
	&\leq \sqrt 2\cdot \sup_{f\in \Bcal_{K^1,\mu}}\left\{ \max_{x\in F_p(\mu)} \max_{x'\in K^s} f(x,x') \right\} 
	\leq \sigma_p(\mu) + \frac{\eta}{2}.
\end{align*}
Hence, on $E_2\cap F_1$, for $n$ sufficiently large,
\begin{equation*}
	\max_{x\in F_p(\mu)} \max_{x'\in K^s} \frac{G_n(x,x')}{\sqrt{n\log\log n}} < \sigma_p(\mu) + \frac{3\eta}{4} < \varepsilon_n \sqrt{\frac{n}{\log\log n}}.
\end{equation*}
Hence, on $E_1\cap E_2\cap F_1$, for $n$ sufficiently large, $F_p(\bar\mu_n)\subseteq K^s$ so using \eqref{eqn:key3}, we obtain
\begin{equation*}
		E_1\cap E_2\cap F_1 \subseteq E_1\cap \{\textrm{for suff. large }n\in\Nbb,F_p(\mu)\subseteq F_p(\bar \mu_n,\varepsilon_n)\} 
		\subseteq \{\dhaus(F_p(\mu),F_p(\bar \mu_n,\varepsilon_n))\to 0\}.
\end{equation*}
Thus, $\Pbb_{\mu}(\dhaus(F_p(\mu),F_p(\bar \mu_n,\varepsilon_n))\to 0) \geq \P_{\mu}(E_1\cap E_2\cap F_1) \geq 1-\delta.$ Since this holds for any $\delta>0$, the result is proved.
\end{proof}

We are now ready to prove Theorem \ref{thm:adaptive-consistency} which shows that using adaptive rate to estimate $\sigma_p(\mu)$, one can obtain near-optimal consistent relaxation rates.


\begin{proof}[Proof of Theorem \ref{thm:adaptive-consistency}]
Let $\mu\in\Pcal_{2p-2}(X)$ and fix a sub-optimal relaxation scale $\{\varepsilon_{1,n}\}_{n\in\Nbb}$ satisfying $\varepsilon_{1,n} = \omega(n^{-1/2}(\log\log n)^{1/2})$ and $\varepsilon_{1,n} = o(1)$ almost surely. Then,
\begin{equation*}
	\liminf_{n\to\infty} \varepsilon_{1,n} \sqrt{\frac{n}{\log\log n}} > \sigma_p(\mu),\quad a.s.
\end{equation*}
As a result, Theorem \ref{thm:consistency} implies that $\dvechaus(F_p(\bar\mu_n,\varepsilon_{1,n}),F_p(\mu))\to 0$ and that for sufficiently large $n\in\Nbb$, $F_p(\mu) \subseteq F_p(\bar\mu_n,\varepsilon_{1,n})$, both holding on an event $E$ of full probability.
Then,
\begin{align*}
	E &\subseteq \{\text{for suff. large }n\in\Nbb, F_p(\mu) \subseteq  F_p(\bar\mu_n,\varepsilon_{1,n}) \}\\
	&\subseteq \left\{\text{for suff. large }n\in\Nbb,  \sigma_p(\bar\mu_n,\varepsilon_{1,n}) \geq \sqrt 2 \sup_{x,x'\in F_p(\mu)} \sqrt{R_{\bar\mu_n}(x,x')} \right\}.
\end{align*}
Now denote by $K^\eta := (F_p(\mu))^\eta$ the $\eta$-thickening of $F_p(\mu)$ for $\eta > 0$.
By Lemma \ref{lemma:uniform_convergence_covariance}, on an event $F$ of full probability, one has $R_{\bar\mu_n}\to R_\mu$ uniformly on $K^1\times K^1$, hence
\begin{equation*}
	\lim_{n\to\infty}\sup_{x,x'\in F_p(\mu)} \sqrt{R_{\bar\mu_n}(x,x')} = \sup_{x,x'\in F_p(\mu)} \sqrt{R_\mu(x,x')} = \sigma_p(\mu).
\end{equation*}
As a result, we obtain
\begin{equation}\label{eq:useful_3}
	E\cap F \subseteq \left\{ \liminf_{n\to\infty} \sigma_p(\bar\mu_n,\varepsilon_{1,n}) \geq \sigma_p(\mu) \right\}.
\end{equation}
Next, on the event $E$, since $\dvechaus(F_p(\bar\mu_n,\varepsilon_{1,n}),F_p(\mu))\to 0$, for sufficiently large $n\in N$, we have $F_p(\bar\mu_n,\varepsilon_{1,n})\subseteq K^\eta$.
As a result,
\begin{align*}
	E\cap F &\subseteq F\cap\left\{ \text{for suff. large }n\in\Nbb, \sigma_p(\bar\mu_n,\varepsilon_{1,n})\leq  \sqrt 2 \sup_{x,x'\in K^\eta} \sqrt{R_{\bar\mu_n}(x,x')}  \right\}\\
	&\subseteq \left\{ \limsup_{n\to\infty} \sigma_p(\bar\mu_n,\varepsilon_{1,n}) \leq \sqrt 2 \sup_{x,x'\in K^\eta} \sqrt{R_\mu(x,x')} \right\}
\end{align*}
By continuity of $R_\mu$ on $X^2$ (see Lemma \ref{lemma:continuity-Z-R}), we have $ \sup_{x,x'\in K^\eta} \sqrt{2 R_\mu(x,x')} \to \sigma_p(\mu)$ as $\eta\to 0$. As a result, the previous equation shows that
\begin{equation*}
	E\cap F \subseteq \left\{ \limsup_{n\to\infty} \sigma_p(\bar\mu_n,\varepsilon_{1,n}) \leq \sigma_p(\mu) \right\}.
\end{equation*}
Together with \eqref{eq:useful_3}, we have $\sigma_p(\bar\mu_n,\varepsilon_{1,n})\to\sigma_p(\mu)$ on $E\cap F$.
Thus, on $E\cap F$ which has full probability,
\begin{equation*}
	\lim_{n\to\infty} \varepsilon_{2,n,\delta}\sqrt{\frac{n}{\log\log n}} = (1+\delta)\sigma_p(\mu).
\end{equation*}
Hence, this implies the desired optimality property when $\delta\to 0$.
Further, by Theorem \ref{thm:consistency}, we have that $F_p(\bar\mu_n,\varepsilon_{2,n,\delta})$ is strongly $\dhaus$-consistent. This ends the proof.
\end{proof}

\subsection{Proofs for the phylogenetic application}
\label{subsec:proof_phylo}

We defer the proof of Lemma~\ref{lemma:tropical_HBD}, Lemma~\ref{lemma:linear_frechet_mean}, and Theorem~\ref{thm:phylo-alg} to Appendix~\ref{sec:phylogenomic}, and we provide now a proof of Theorem~\ref{thm:FW-proj-trees} which is of independent interest:

\begin{proof}[Proof of Theorem~\ref{thm:FW-proj-trees}]
It suffices to show $F_1(\mu,\varepsilon) \cap \mathcal U_N \neq \emptyset$, since then the Fr\'echet functional attains the same minimum on both the space of equidistant phylogenetic trees $\Ucal_N$ and the ambient tropical projective space.
To do this, we write $E = \{\{i,j\}:1\leq i<j\leq N\}$ for the set of all pairs of leaves.
Then we consider the optimization problem
    \begin{equation}\label{eq:opt_problem}
    \begin{cases}
        \textnormal{minimize} &\sum_{e\in E} (u_e-u_{\{1,2\}})\\
        \textnormal{over}& u\in F_1(\mu).
    \end{cases}
    \end{equation}
(Note that the objective is invariant under the action of $\Rbb\mathbf{1}$, as it must be.)
Recall that $F_1(\mu)$ is compact (Lemma~\ref{lemma:Fp-cpt}) and the objective is continuous, so there exists a minimizer $u^{\ast}$ of \eqref{eq:opt_problem}.

We claim $u^{\ast}\in\Ucal_N$.
Assume for the sake of contradiction that this is not the case, so that there exist distinct $i,j,k\in[N]$ with $u^\ast_{\{i,k\}}> \max(u^\ast_{\{i,j\}}, u^\ast_{\{j,k\}})$.
Now define $u^{\ast\ast}\in\Rbb^{E}$ via
\begin{equation*}
    u^{\ast\ast}_e := \begin{cases}
        u_e &\textnormal{ if } e\neq \{i,k\}, \\
        \max(u^\ast_{\{i,j\}}, u^\ast_{\{j,k\}}) &\textnormal{ if } e= \{i,k\}.
    \end{cases}
\end{equation*}
Observe that, for any $w\in \Ucal_N$, we have
    \begin{align*}
        u^\ast_{\{i,k\}} - w_{\{i,k\}} > \max(u^\ast_{\{i,j\}}, u^\ast_{\{j,k\}}) - \max(w_{\{i,j\}}, w_{\{j,k\}})&\geq \min (u^\ast_{\{i,j\}}- w_{\{i,j\}} , u^\ast_{\{j,k\}} -w_{\{j,k\}} ).
    \end{align*}
This means that, for both $u\in \{u^{\ast},u^{\ast\ast}\}$, the map $e\mapsto u_e-w_e$ is minimized at a coordinate different than $e=\{i,k\}$.
Since $u^{\ast\ast}_{\{i,k\}} < u^\ast_{\{i,k\}}$ and since $u^\ast$ and $u^{\ast\ast}$ coincide on all coordinates other than $e=\{i,k\}$, we conclude
    \begin{equation*}
        \dtrop(u^\ast,w) = \max_{e\in E} (u^\ast_e - w_e) - \min_{e\in E} (u^\ast_e - w_e) \geq \max_{e\in E} (u^{\ast\ast}_e - w_e) - \min_{e\in E} (u^{\ast\ast}_e - w_e) = \dtrop(u^{\ast\ast},w).
    \end{equation*}
    for all $w\in\Ucal_N$.
    Now note that, for all $u\in\Ucal_N$:
    \begin{equation*}
        \begin{split}
            W_1(\mu,u^{\ast\ast},u) &= \int_{\Ucal_N}(\dtrop(u^{\ast\ast},w)-\dtrop(u,w))\, d\mu(w) \\
            &= \int_{\Ucal_N}(\dtrop(u^{\ast},w)-\dtrop(u,w))\, d\mu(w) = W_1(\mu, u^{\ast}, u) \le \varepsilon.
        \end{split}
    \end{equation*}
    This implies $u^{\ast\ast}\in F_1(\mu)$.
    However, we have
    \begin{equation*}
        \sum_{e\in E} (u^\ast_e-u^\ast_{\{1,2\}}) > \sum_{e\in E} u^{\ast\ast}_e-u^{\ast\ast}_{\{1,2\}},
    \end{equation*}
    which contradicts the optimality of $u^\ast$ for the optimization problem \eqref{eq:opt_problem}.
    This shows $u^\ast\in F_1(\mu,\varepsilon)\cap \Ucal_N$, which ends the proof of the result.
\end{proof}

	
\subsection*{Acknowledgments}
	
We would like to thank Guillaume Goujard for facilitating this collaboration. We also thank the anonymous reviewers who significantly improved the quality of the paper with their suggestions.

\subsection*{Funding}

This material is based upon work for which  MB was partly funded by Office of Naval Research grant N00014-18-1-2122 and AQJ was funded by National Science Foundation grant DGE 1752814.

\bibliographystyle{alpha} 
\bibliography{FrechetMeanRelaxation}       

\begin{appendix}
	
	\begin{center}
		\textsc{Appendices}
	\end{center}

First, in Appendix~\ref{sec:mble}, we establish the required measurability of the estimators studied in the main body.
Then, we complete the omitted proofs from the main body, for the probabilistic setting (Appendix~\ref{sec:probabilistic_setting}), the results on strong laws of large numbers (Appendix~\ref{sec:SLLN}), for the main general consistency results (Appendix~\ref{sec:proof_general_results}), and for the phylogenetic application (Appendix~\ref{sec:phylogenomic}).
\comment{Last, we provide in Appendix~\ref{sec:special cases} some examples of the way that non-parametric assumptions can allow one to obtain sub-optimal but consistent relaxation rates which are easier to work with, including some simplifications of the two-step estimator described in Theorem~2.4.}

\comment{
We recall two simple inequalities on powers of a distance. First given a metric space $(X,d)$, for $q\geq 0$, the constant $c_q := \max\{1,2^{q-1}\}$ is such that we have
\begin{equation}\label{eqn:ineq-1}
	d^q(x,y)\le c_q(d^q(x,x')+d^q(x',y))
\end{equation}
for all $x,x',y\in X$. Next, for $p\geq 1$ we have
\begin{equation}\label{eqn:ineq-2}
	|d^p(x,y)-d^p(x',y)|\le pd(x,x')(d^{p-1}(x,y)+d^{p-1}(x',y))
\end{equation}
for all $x,x',y\in X$. 
}

We introduce some notions of set-valued convergence that will be useful to state our strong-law-of-large-number results.
First of all, for any $d$-closed (ferm\'e) subset $C\in \closed(X)$, define
\begin{equation*}
	d(x,C) := \min_{x'\in C}d(x,x')
\end{equation*}
for the \textit{metric projection of $x$ onto $C$}; we take $d(\cdot,\emptyset)\equiv \infty$ by convention, and note that $d(\cdot,C):X\to [0,\infty]$ is continuous for all $C\in \closed(X)$. For a sequence $\{C_n\}_{n\in\Nbb}$ in $\closed(X)$, we write
\begin{align*}
	\kurouter_{n\in\Nbb}C_n &:= \left\{x\in X: \liminf_{n\to\infty}d(x,C) = 0\right\} \\
	\kurinner_{n\in\Nbb}C_n &:= \left\{x\in X: \limsup_{n\to\infty}d(x,C) = 0\right\}
\end{align*}
for their \emph{Kuratowski upper limit} and \emph{Kuratowski lower limit}, respectively.
We always have $\kurinner_{n\in\Nbb}C_n \subseteq \kurouter_{n\in\Nbb}C_n$, and, if $\kurinner_{n\in\Nbb}C_n = \kurouter_{n\in\Nbb}C_n$, then we write $\kurlimit_{n\in\Nbb}C_n$ for their common value and we say that $\{C_n\}_{n\in\Nbb}$ \emph{converges in the Kuratowski sense}.
Let us also write $\cpt(X)$ for the collection of non-empty $d$-compact subsets of $X$, and let us recall the definitions of $\dvechaus$ and $\dhaus$ given in Section~\ref{sec:intro}.
We note in passing (see \cite{EvansJaffeSLLN}) that convergence in the Hausdorff metric is much stronger than convergence in the Kuratowski sense.

As a remark, most objects in this paper depend on the choice of exponent $p\ge 1$. 
In general, objects appearing in the statements of our main theorems $(F_p(\mu), \sigma_p(\mu)$, etc.) display their dependence on $p$ while objects appearing only in the body of the paper or in proofs $(Z_i, R_{\mu}$, etc.) do not display their dependence on $p$.

\section{Measurability concerns}\label{sec:mble}

In order to study any probabilistic questions about relaxed empirical Fr\'echet mean sets and their convergence in the Hausdorff metric, we must establish some basic measurability properties.
For example, we must show that the events appearing in our main theorems, like $\{\dhaus(F_p(\bar \mu_n,\varepsilon_n),F_p(\mu)) \ge \delta\}$ and $\{\dhaus(F_p(\bar \mu_n,\varepsilon_n),F_p(\mu)) \to 0\}$, are in fact measurable.
This is similar to the work of \cite[Section~3]{EvansJaffeSLLN}, except our work is slightly simplified since we only care about the case of HB spaces.

For this section, let us adopt some notation.
As always, fix $p \ge 1$ and let $(\Omega,\F,\P_{\mu})$ be a probability space on which is defined a sequence $Y_1,Y_2,\ldots$ of $X$-valued random variable with common distribution $\mu\in \Pcal_{p-1}(X)$.
Note that we do not assume that $\F$ is $\P$-complete.
Then, recall that by the \emph{Effros $\sigma$-algebra} $\Ecal(X)$ on $\closed(X)$ we mean the $\sigma$-algebra generated by the sets $\{C\in \closed(X): C\cap B_r^{\circ}(x) \neq \emptyset\}$ where $x$ ranges over all points in $X$ and $r$ ranges over all non-negative real numbers.
An $\Ecal(X)$-measurable map $\Omega\to\closed(X)$ is called a \emph{random set}.
We also see that $\Ecal(X)$ induces a $\sigma$-algebra on $\cpt(X)$ which (by a slight abuse of notation) we also denote $\Ecal(X)$, so an $\Ecal(X)$-measurable map $\Omega\to\cpt(X)$ will also be called a random set.

Our first result shows that the relaxed restricted Fr\'echet mean sets are in fact random sets.
Note that this result requires us to invoke 
Lemma~\ref{lemma:Fp-cpt},
showing the compactness of relaxed restricted Fr\'echet means, which is proved in the next section; this is not a fault in the logic, but we have chosen this ordering for the sake of exposition.

\begin{lemma}\label{lemma:relaxed_constrained_frechet_means_measurable}
	Let $(X,d)$ be a HB space, and $p\ge 1$.
	Then, for any $n\in\Nbb$, any random set $C:\Omega\to\closed(X)$, and any random relaxation scale $\varepsilon :\Omega\to [0,\infty)$, the relaxed restricted Fr\'echet mean set $F_p(\bar \mu_n,C,\varepsilon):\Omega\to \cpt(X)$ is a random set.
\end{lemma}

\begin{proof}
	Since $F_p(\bar\mu_n,C,\varepsilon)$ is compact from Lemma~\ref{lemma:Fp-cpt},
	we need to show that the map $F_p(\bar\mu_n,C,\varepsilon):\Omega\to \cpt(X)$ is $\Ecal(X)$-measurable.
	Because $X$ is HB, let $\{x_k\}_{k\in \Nbb}$ be a countable dense set in $X$.
	We first show that the function $f(x):=\sup_{x'\in C} W_p(\bar\mu_n,x,x') = W_p(\bar\mu_n,x) - \inf_{x'\in C}W_p(\bar\mu_n, x')$ is Lipschitz on $B_r(z)$. To do so, it suffices to show that $x\mapsto W_p(\bar\mu_n,x)$ is Lipschitz on $B_r(z)$.
	Indeed, for any $x_1,x_2\in B_r(z)$, we use \eqref{eqn:ineq-2} then \eqref{eqn:ineq-1} to obtain
	\begin{align*}
		|W_p(\bar\mu_n,x_1)-W_p(\bar\mu_n,x_2)|&\leq p d(x_1,x_2) \frac{1}{n}\sum_{k=1}^n(d^{p-1}(x_1,Y_k) + d^{p-1}(x_2,Y_k))\\
		&\leq 2p c_{p-1}(W_p(\bar\mu_n,z) + r^{p-1}) d(x_1,x_2).
	\end{align*}
	As a result, $f$ is Lipschitz (with random Lipschitz constant) on $B_r(z)$.
	
	Next, let $C_j= \{x \in X: d(x,C)<2^{-j}\}$ for any $j\in \Nbb$. We write       
	\begin{align*}
		&F_p(\bar \mu_n,C,\varepsilon)^{-1}(\{K\in \cpt(X): K\cap B_r^\circ(z)\neq\emptyset\}) \\
		&= \{F_p(\bar\mu_n,C,\varepsilon)\cap B_r^\circ(z)\neq\emptyset \}\\
		&= \{\exists x\in C\cap B_r^\circ(z),f(x) \leq \varepsilon \}\\
		&= \bigcup_{i\in\Nbb}\{ \exists x\in  C\cap  B_{r-2^{-i}}(z),f(x) \leq \varepsilon \} \\
		&= \bigcup_{i\in\Nbb}\bigcap_{j\in\Nbb}\{ \exists x\in  C_j\cap B_{r-2^{-i}}^\circ(z), f(x)< \varepsilon +2^{-j}\},
	\end{align*}
	but it remains to justify the last equality.
	Indeed, the ``$\subseteq$'' inclusion is direct from the fact that $C\subseteq C_j$.
	For the ``$\supseteq$'' inclusion, suppose that for all $j\in\Nbb$, there exists $x^j \in C_j\cap B_{r-2^{-i}}^\circ(z)$, with $f(x^j)<\varepsilon+2^{-j}$.
	Because $B_{r-2^{-i}}(z)$ is compact from the HB property, there is $x\in  B_{r-2^{-i}}(z)$ and a subsequence $\{j_k\}_{k\in\Nbb}$ such that $x^{j_k}\to x$.
	By construction, $d(x^{j_k},C)<2^{-j_k}$, hence $x\in C$ because $C$ is closed.
	Finally, by continuity of $f$, we have $f(x)\leq \varepsilon$.
	Therefore, it suffices to show that $\{ \exists x\in  C_j\cap B_{r-2^{-i}}^\circ(z), f(x)< \varepsilon +2^{-j}\}$ is measurable for each $i,j\in\Nbb$.
	
	To show this, note that $C_{j}\cap B_{r-2^{-i}}^\circ(z)$ is open, hence $\{x_k: k\in \Nbb,x_k\in C_j\cap B_{r-2^{-i}}^\circ(z)\}$ is dense in $C_j\cap B_{r-2^{-i}}^\circ(z)$.
	Therefore, we have
	\begin{align*}
		&\{ \exists x\in  C_j\cap B_{r-2^{-i}}^\circ(z), f(x)< \varepsilon +2^{-j}\}\\
		&=\bigcup_{\substack{k\in\Nbb \\ x_k\in  B_{r-2^{-i}}^\circ(z)}} \left(\{x_k\notin C_j\} \cup \{f(x_k)< \varepsilon + 2^{-j}\} \right)\\
		&=\bigcup_{\substack{k\in\Nbb \\ x_k\in  B_{r-2^{-i}}^\circ(z)}} \left(\{C\cap B_{2^{-j}}^\circ(x_k)=\emptyset\} \cup \{\forall x'\in C, W_p(\bar\mu_n, x_k, x') <\varepsilon + 2^{-j} \} \right)
	\end{align*}
	by continuity of $f$.
	By definition of $\Ecal(X)$, it further suffices to check that $\{\forall x'\in C, W_p(\bar\mu_n, x_k, x') <\varepsilon + 2^{-j} \}$ is measurable for all $j,k\in\Nbb$.
	
	To do this, let $ r_n = \max\{d(x_k,Y_m):1 \le m \le n\}$ and $R_n = (2c_p)^{1/p} r_n$.
	We claim that we have
	\begin{equation}\label{eqn:meas-3}
		\begin{split}
			\{\forall &x'\in C, W_p(\bar\mu_n, x_k, x') <\varepsilon + 2^{-j} \} \\
			&=  \bigcup_{\ell\in\Nbb} \left\{\forall x'\in C_{\ell}\cap B_{R_n}^\circ(x_k), W_p(\bar\mu_n,x_k,x') \leq  \varepsilon + 2^{-j} - 2^{-\ell}\right\}.
		\end{split}
	\end{equation}
	for all $j,k\in\Nbb$.
	To prove this, we will need the following bound:
	For all $x'\notin B_{R_n}^\circ(x_k)$, we use \eqref{eqn:ineq-1} to get
	\begin{equation}\label{eqn:meas-2}
		\begin{split}
			W_p(\bar\mu_n,x_k,x') &= \frac{1}{n}\sum_{m=1}^n (d^p(x_k,Y_m) - d^p(x',Y_m))\\
			&\leq \frac{1}{n}\sum_{m=1}^n \left(d^p(x_k,Y_m) - \left(\frac{1}{c_p}d^p(x',x_k) - d^p(x_k,Y_m)\right)\right)\\
			&\leq 2r_n^p - \frac{R_n^p}{c_p} = 0.
		\end{split}
	\end{equation}
	Now we can prove \eqref{eqn:meas-3}.
	To see the ``$\supseteq$'' inclusion, suppose that the right side holds for some $\ell\in\Nbb$.
	Now to show that the left side holds, for arbitrary $x'\in C$, we consider two cases:
	If $x'\in B_{R_n}^\circ(x_k)$, then the right side along with $C\subseteq C_{\ell}$ implies $W_p(\bar \mu_n,x_k,x') \le \varepsilon+2^{-j}-2^{-\ell} < \varepsilon + 2^{-j}$.
	Otherwise, $x'\notin B_{R_n}^\circ(x_k)$, and we have $W_p(\bar \mu_n,x_k,x')\le 0<\varepsilon+2^{-j}$ by \eqref{eqn:meas-2}.
	To see the ``$\subseteq$'' inclusion, let us suppose that the left side holds.
	Since $B_{R_n}(x_k)$ is compact by the HB property, there exists $x^\star_k\in C\cap B_{R_n}(x_k)$ such that $W_p(\bar\mu_n,x_k,x^\star_k) = \sup_{x'\in C\cap B_{R_n}(x_k)} W_p(\bar\mu_n,x_k,x')$.
	So, by the left side of \eqref{eqn:meas-3}, there exists some $\delta > 0$ such that $\sup_{x'\in C\cap B_{R_n}(x_k)} W_p(\bar\mu_n, x_k, x') \leq \varepsilon + 2^{-j} - \delta$. 
	From before, we know that $x'\mapsto W_p(\bar\mu_n,x_k,x')$ is Lipschitz on the ball $B_{R_n}^\circ(x_k)$; we denote its (random) Lipschitz constant as $M_n$.
	In particular, with $\eta_n = \delta/(2M_n)$ and $\ell =  \lceil \log_2(2/\delta)\rceil $, we obtain
	\begin{align*}
		\sup_{x'\in C_{\ell}\cap B_{R_n}^\circ(x_k)} &W_p(\bar\mu_n, x_k, x') \\
		&\le \sup_{x'\in C\cap B_{R_n}^\circ(x_k)} W_p(\bar\mu_n, x_k, x') + M_n\eta_n \\
		&\leq \varepsilon + 2^{-j} - \delta + M_n\eta_n \\
		&= \varepsilon + 2^{-j} - \delta/2 \\
		&\leq \varepsilon + 2^{-j} -2^{-\ell}
	\end{align*}
	This proves the second inclusion and establishes \eqref{eqn:meas-3}.
	Therefore, it suffices to show that
	\begin{equation*}
		\left\{\forall x'\in C_{\ell}\cap B_{R_n}^\circ(x_k), W_p(\bar\mu_n,x_k,x') \leq  \varepsilon + 2^{-j} - 2^{-{\ell}}\right\}
	\end{equation*}
	is measurable for all $j,k,\ell\in\Nbb$.
	
	Finally, to do this, let us use that $\{x_m: m\in\Nbb, x_m\in C_{\ell}\cap B_{R_n}^{\circ}(x_k)\}$ is dense in $C_{\ell}\cap B_{R_n}^{\circ}(x_k)$ and that $x'\mapsto W_p(\bar \mu_n,x_k,x')$ is continuous.
	Then we get
	\begin{align*}
		&\left\{\forall x'\in C_{\ell}\cap B_{R_n}^\circ(x_k), W_p(\bar\mu_n,x_k,x') \leq  \varepsilon + 2^{-j} - 2^{-\ell}\right\} \\
		&= \left\{\forall (m\in\Nbb,x_m\in C_{\ell}\cap B_{R_n}^\circ(x_k)), W_p(\bar\mu_n,x_k,x_m) \leq  \varepsilon + 2^{-j} - 2^{-\ell}\right\} \\
		&= \bigcap_{m\geq 1} (\{x_m \notin  C_{\ell}\cap B_{R_n}^\circ(x_k)\} \cup\{W_p(\bar\mu_n,x_k,x_m)  \leq  \varepsilon + 2^{-j} - 2^{-\ell}\}) \\
		&=\bigcap_{m\geq 1} (\{C\cap B_{2^{-\ell}}^\circ(x_m) = \emptyset \} \cup \{d(x_m,x_k) \geq R_n\} \cup\{W_p(\bar\mu_n,x_k,x_m)  \leq  \varepsilon + 2^{-j} - 2^{-\ell}\})
	\end{align*}
	Recalling the definitions of $R_n$ and $W_p(\bar\mu_n,x_k,x_m)$, we observe that the last term is measurable. 
	Thus, the proof is complete.
\end{proof}

Our next result shows that the one-sided Hausdorff distance is a measurable function of random sets.
\begin{lemma}\label{lemma:dhaus_measurable}
	Let $(X,d)$ be a HB space.
	Then, for any random sets $K,K':\Omega\to\cpt(X)$, the map $\dvechaus(K,K'):\Omega\to \R$ is Borel-measurable.
\end{lemma}

\begin{proof}
	For any $n\in\Nbb$, let $K_n = \{x\in X: d(x, K)<2^{-n}\}$.
	We claim for all $\beta > 0$ that
	\begin{equation}\label{eqn:mble-1}
		\left\{ \max_{x\in K}  d(x,K') < \beta\right\}=\bigcup_{n\in\Nbb} \left\{ \sup_{x\in K_n} d(x,K') < \beta\right\}.
	\end{equation}
	To prove this, note that one direction is immediate from the inclusion $K\subseteq K_n$.
	For the other direction, note that if $\max_{x\in K} d(x,K') < \beta$, then for any $n\in\Nbb$ such that $2^{-n}< \beta - \max_{x\in K_n} d(x,K')$ and any $y\in K_n$, there exists $z\in K$ with $d(y,z)< 2^{-n}$.
	As a result
	\begin{equation*}
		d(y,K') \leq d(y,z) + d(z,K') < 2^{-n} + \max_{x\in K} d(x,K') < \beta,
	\end{equation*}
	as needed.
	Now, fix $\alpha\in\Rbb$, and use \eqref{eqn:mble-1} to get:
	\begin{align*}
		\{\dvechaus(K,K') \leq \alpha\} &= \bigcap_{k\in\Nbb}\left\{ \dvechaus(K,K') < \alpha + 2^{-k}\right\}\\
		&= \bigcap_{k\in\Nbb} \left\{ \max_{x\in K}  d(x,K') < \alpha + 2^{-k}\right\}\\
		&= \bigcap_{k\in\Nbb} \bigcup_{n\in\Nbb} \left\{ \sup_{x\in K_n} d(x,K') < \alpha + 2^{-k}\right\}.
	\end{align*}
	Thus, if we show that
	\begin{equation*}
		\left\{ \sup_{x\in K_n} d(x,K') < \alpha + 2^{-k}\right\}
	\end{equation*}
	is measurable for all $k,n\in\Nbb$, then it follows that $\dvechaus(K,K'):\Omega\to\R$ is Borel-measurable.
	To do this, we use that $X$ is HB to get a dense sequence $\{x_i\}_{i\in\Nbb}$ of $X$.
	Since $K_n = \bigcup_{x\in K}B_{2^{-n}}^{\circ}(x)$ is open, it follows that $\{x_i:i\in\Nbb, x_i\in K_n\}$ is dense in $K_n$.
	Thus, since the function $x\in X\mapsto d(x,K')$ is continuous we have $ \sup_{x\in K_n}  d(x,K') =  \sup\{d(x_i,K'):i\in\Nbb,x_i\in K_n\} $.
	Therefore,
	\begin{align*}
		&\left\{ \sup_{x\in K_n} d(x,K') < \alpha + 2^{-k}\right\} \\
		&= \bigcap_{i\in\Nbb} \left(\{x_i\notin K_n\}\cup\left\{ d(x_i,K') < \alpha + 2^{-k}\right\}\right) \\
		&= \bigcap_{i\in\Nbb} \left(\{d(x_i,K)\geq 2^{-n}\}\cup\left\{ d(x_i,K') < \alpha + 2^{-k}\right\}\right)\\
		&=\bigcap_{i\in\Nbb} \left(\{ K \cap B_{2^{-n}}^{\circ}(x_i) = \emptyset\}\cup\left\{ K'\cap B_{\alpha + 2^{-k}}^{\circ}(x_i) \neq\emptyset \right\}\right).
	\end{align*}
	By definition of $\Ecal(X)$, the last term is measurable, hence we showed that the map $\dvechaus(K, K'):\Omega\to \Rbb$ is Borel-measurable.
\end{proof}

These results together yield the measurability of most events of interest in the paper. We leave the remaining details to the interested reader, but we note that this primarily follows from Lemma~\ref{lemma:relaxed_constrained_frechet_means_measurable} and two applications of Lemma~\ref{lemma:dhaus_measurable}, which imply that $\dhaus(F_p(\bar \mu_n,\varepsilon),F_p(\mu)):\Omega\to\R$ is Borel-measurable for any random relaxation scale $\varepsilon \ge 0$.
For example, this implies that for any random relaxation rate $\varepsilon_n$, the event $\{\dhaus(F_p(\bar \mu_n,\varepsilon_n),F_p(\mu))\to 0\}$ is measurable.

The only remaining measurability concern is to show that the adaptive relaxation rate $\varepsilon_{2,n,\delta}$ of Theorem~\ref{thm:adaptive-consistency}
is in fact Borel-measurable.
This is established in the following.

\begin{lemma}
	Let $(X,d)$ be a HB space, and $p\ge 1$.
	Then, for any $n\in\Nbb$ and any random relaxation scale $\varepsilon:\Omega\to [0,\infty)$, the parameter $\sigma_p(\bar \mu_n,\varepsilon):\Omega\to [0,\infty)$ is Borel-measurable.
\end{lemma}

\begin{proof}
	By Lemma~\ref{lemma:relaxed_constrained_frechet_means_measurable}, it suffices to show that for any random set $K:\Omega\to \cpt(X)$ and any $\alpha \in\R$, the event $\{\sup_{x,x'\in K} R_{\bar \mu_n}(x,x') < \alpha\}$ is measurable.
	The approach is similar to the proofs of the preceding measurability results:
	We write $K_m := \bigcup_{x\in K}B_{2^{-m}}^{\circ}(x)$ for each $m\in\Nbb$, and, since $(X,d)$ is a HB space, we can also get a sequence $\{x_i\}_{i\in\Nbb}$ which is dense in $X$.
	Then using the fact that $R_{\bar \mu_n}$ is Lipschitz on compact sets (Lemma~\ref{lem:Z-R-cts}), we have
	\begin{align*}
		&\left\{\sup_{x,x'\in K} R_{\bar \mu_n}(x,x') < \alpha\right\} \\
		&= \bigcup_{m\in\Nbb}\left\{\sup_{x,x'\in K_m} R_{\bar \mu_n}(x,x') < \alpha\right\} \\
		& = \bigcup_{m\in \Nbb}\bigcup_{i\in\Nbb} \left\{\sup_{x,x'\in K_m} R_{\bar \mu_n}(x,x') \leq \alpha - 2^{-i}\right\} \\ 
		&= \bigcup_{m\in\Nbb}\bigcup_{i\in\Nbb}\left\{\sup\{R_{\bar \mu_n}(x_k,x_{\ell}): k,\ell\in\Nbb, x_k,x_{\ell}\in K_m\} \leq  \alpha - 2^{-i}\right\} \\
		&= \bigcup_{m\in\Nbb}\bigcup_{i\in \Nbb}\bigcap_{k,\ell\in\Nbb}(\{x_k,x_{\ell}\notin K_m\}\cup\{R_{\bar \mu_n}(x_k,x_{\ell}) \leq \alpha - 2^{-i}\}).
	\end{align*}
	Note that $\{R_{\bar \mu_n}(x_k,x_{\ell}) \leq \alpha - 2^{-i}\}$ is easily shown to be measurable, and that
	\begin{equation*}
		\{x_k,x_{\ell}\notin K_m\} = \{K\cap B_{2^{-m}}^{\circ}(x_{k}) = \emptyset\}\cap\{K\cap B_{2^{-m}}^{\circ}(x_{\ell}) = \emptyset\}
	\end{equation*}
	is measurable by the definition of $\Ecal(X)$.
	This finishes the proof.
\end{proof}

\section{Proofs for the probabilistic setting (Subsection~\ref{subsec:prob})}
\label{sec:probabilistic_setting}

The following result clarifies an important condition for pointwise convergence of the Fr\'echet functional.

\begin{lemma}\label{lem:UI}
	Suppose that $(X,d)$ is a metric space, $p\ge 1$, $\alpha \ge 1$ is an integer, and that $\{\mu_n\}_{n\in\Nbb}$ and $\mu$ in $\Pcal_{\alpha(p-1)}(X)$ have $\mu_n\to \mu$ in the weak topology.
	If for some $x_0\in X$,
	\begin{equation*}
		\int_{X}d^{\alpha (p-1)}(x_0,y) \, d\mu_n(y) \to \int_{X}d^{\alpha (p-1)}(x_0,y)\, d\mu(y),
	\end{equation*}
	then for all $x,x'\in X$,
	\begin{equation*}
		\int_{X}(d^p(x,y)-d^p(x',y))^{\alpha}\, d\mu_n(y) \to \int_{X}(d^p(x,y)-d^p(x',y))^{\alpha}\, d\mu(y).
	\end{equation*}
\end{lemma}

\begin{proof}
	Fix $x_0\in X$, and let $x,x'\in X$ be arbitrary.
	Let $Y_n$ and $Y$ denote random elements of $X$ with laws $\mu_n$ and $\mu$, respectively, for each $n\in\Nbb$, defined on the same probability space with expectation $\E$.
	Then we have $Y_n\to Y$ in distribution, so, by the continuous mapping theorem \cite[Lemma~5.3]{kallenberg1997foundations}, we have $(d^p(x,Y_n)-d^p(x',Y_n))^{\alpha}\to (d^p(x,Y)-d^p(x',Y))^{\alpha}$ in distribution.
	Also note by assumption that 
	\begin{equation*}
		\E[d^{\alpha (p-1)}(x_0,Y_n)] = \int_{X} d^{\alpha (p-1)}(x_0,y)\, d\mu_n(y) 
	\end{equation*}
	converges as $n\to\infty$. Therefore, from \cite[Lemma~5.11]{kallenberg1997foundations}, the sequence $\{d^{\alpha (p-1)}(x_0,Y_n)\}_{n\in\Nbb}$ is uniformly integrable.
	Now use \eqref{eqn:ineq-2} then \eqref{eqn:ineq-1} to get
	\begin{equation*}
		\begin{split}
			|d^{p}(x,Y_n)-d^p(x',Y_n)| &\le pd(x,x')(d^{p-1}(x,Y_n)+d^{p-1}(x',Y_n)) \\
			&\le pc_{p-1}d(x,x')(d^{p-1}(x,x_0)+d^{p-1}(x',x_0)+2d^{p-1}(x_0,Y_n)).
		\end{split}
	\end{equation*}
	Consequently, by convexity we have
	\begin{equation*}
		\begin{split}
			|d^{p}&(x,Y_n)-d^p(x',Y_n)|^{\alpha} \\
			&\le (pc_{p-1})^{\alpha}3^{\alpha-1} d^{\alpha}(x,x')(d^{\alpha(p-1)}(x,x_0) +d^{\alpha(p-1)}(x',x_0)+2^{\alpha}d^{\alpha(p-1)}(x_0,Y_n)) \\
		\end{split}
	\end{equation*}
	In particular, the sequence $\{|d^{p}(x,Y_n)-d^{p}(x',Y_n)|^{\alpha}\}_{n\in\Nbb}$ is uniformly integrable.
	Applying \cite[Lemma~5.11]{kallenberg1997foundations} once more proves the result.
\end{proof}

The following simple lemma shows that $R_\mu$ is well defined.

\begin{lemma}\label{lemma:R_well_defined}
	Let $(X,d)$ be a metric space, $p\geq 1$ and $\mu\in\Pcal_{2p-2}(X)$. Then, for any $x,x',x'',x'''\in X$, the random variable $Z_{1}(x,x')Z_{1}(x'',x''')$ is integrable.
\end{lemma}

\begin{proof}
	First, note that $2p-2\geq p-1$, hence $\mu\in\Pcal_{p-1}(X)$ and as a result $Z_1$ is well defined and both $d^p(x,Y_1)-d^p(x',Y_1)$, $d^p(x'',Y_1)-d^p(x''',Y_1)$ are integrable. It remains to check that $(d^p(x,Y_1)-d^p(x',Y_1))(d^p(x'',Y_1)-d^p(x''',Y_1))$ is integrable. Next, fix $x_0\in X$ and note that
	\begin{align*}
		|d^p(x,Y_1) - d^p(x',Y_1)| &\leq pd(x,x') (d^{p-1}(x,Y_1) + d^{p-1}(x',Y_1))\\
		&\leq pd(x,x') c_{p-1} (2d^{p-1}(x_0,Y_1) + d^{p-1}(x,x_0) + d^{p-1}(x',x_0)).
	\end{align*}
	Similarly, we obtain an upper bound for $|d^p(x'',Y_1)-d^p(x''',Y_1)|$. Now note that the random variable $d(x_0,Y_1)^{2p-2}$ is integrable since $\mu\in\Pcal_{2p-2}(X)$. Hence, combining the two upper bounds shows that $|d^p(x,Y_1)-d^p(x',Y_1)||d^p(x'',Y_1)-d^p(x''',Y_1)|$ is integrable.
\end{proof}

We will need the following continuity result.

\begin{lemma}\label{lem:Z-R-cts}
	Let $(X,d)$ be a metric space, $p\ge 1$, and $\mu\in \Pcal_{p-1}(X)$.
	Then, for each compact set $K\subseteq X$, the function $Z_i:K\times K\to \R$ is almost surely Lipschitz continuous for the metric $D((x,x'),(x'',x''')) = d(x,x'')+d(x',x''')$.
	Moreover, for any $x_0\in K$, its Lipschitz constant $M_i$ satisfies
	\begin{equation*} M_i\le pc_{p-1}\left(2d^{p-1}(x_0,Y_i)+2\Ebb[d^{p-1}(x_0,Y_i)]+4(\diam(K))^{p-1}\right).
	\end{equation*}
	If also $\mu\in \Pcal_{2p-2}(X)$, then $\E[M_i^2] < \infty$, and the function $R_{\mu}:K^4\to\R$ is Lipschitz with respect to the metric $D((x_1,x_1',x_1'',x_1'''), (x_2,x_2',x_2'',x_2''')) = d(x_1,x_2) + d(x_1',x_2') + d(x_1'',x_2'') + d(x_1''',x_2''') $, with Lipschitz constant $M_\mu$ bounded by
	\begin{equation*}
		M_\mu \leq 24p^2 c_{p-1}^2\diam(K)(\Ebb_\mu[d^{2p-2}(x_0,Y_1)] + 2(\diam(K))^{2p-2}).
	\end{equation*}
\end{lemma}

\begin{proof}[Proof of Lemma~\ref{lem:Z-R-cts}]
	For any $x,x',x'',x'''\in K$, we use \eqref{eqn:ineq-2} to get
	\begin{align*}
		|Z_i(x,x')-Z_i(x'',x''')| &\le pd(x,x'')\left(d^{p-1}(x,Y_i)+d^{p-1}(x'',Y_i)\right)\\
		&+ pd(x',x''')\left(d^{p-1}(x',Y_i)+d^{p-1}(x''',Y_i)\right) \\
		&+ pd(x,x'')\left(\E[d^{p-1}(x,Y_i)]+\E[d^{p-1}(x'',Y_i)]\right) \\
		&+ pd(x',x''')\left(\E[d^{p-1}(x',Y_i)]+\E[d^{p-1}(x''',Y_i)]\right).
	\end{align*}
	Now fix any $x_0\in K$, and use \eqref{eqn:ineq-1} to further this bound as
	\begin{align*}
		&|Z_i(x,x')-Z_i(x'',x''')| \\
		&\leq D((x,x''),(x',x'''))p c_{p-1}\left(2d^{p-1}(x_0,Y_i)+2\Ebb[d^{p-1}(x_0,Y_i)]+4(\diam(K))^{p-1}\right).
	\end{align*}
	This proves that $Z_i$ is almost surely Lipschitz continuous with Lipschitz constant $M_i$ bounded by the desired value.
	To see that $\E[M_i^2] < \infty$ when $\mu\in \Pcal_{2p-2}(X)$, simply note
	\begin{equation*}
		M_i^2 \le 3p^2 c_{p-1}^2\left(4d^{2p-2}(x_0,Y_i)+4\Ebb[d^{p-1}(x_0,Y_i)]^2+16(\diam(K))^{2p-2}\right),
	\end{equation*}
	hence $\E[M_i^2] < \infty$. Next, for any $x_1,x_1',x_1'',x_1''',x_2,x_2',x_2'',x_2'''\in K$ one has
	\begin{align*}
		&|Z_1(x_1,x_1')Z_1(x_1'',x_1''') - Z_1(x_2,x_2')Z_1(x_2'',x_2''')|\\
		&\leq |Z_1(x_1,x_1')||Z_1(x_1'',x_1''') - Z_1(x_2'',x_2''')| + |Z_1(x_2'',x_2''')||Z_1(x_1,x_1') - Z_1(x_2,x_2')|\\
		&\leq M_1^2( d(x_1,x_1')   D((x_1'',x_1'''),(x_2'',x_2''')) + d(x_2'',x_2''')D((x_1,x_1'),(x_2,x_2')))\\
		&\leq D((x_1,x_1',x_1'',x_1'''), (x_2,x_2',x_2'',x_2''')) \diam(K) M_1^2,
	\end{align*}
	where we used the fact that $Z_1(x,x)=0$. As a result, since $\Ebb[M_1^2]<\infty$, this shows that $R_\mu$ is Lipschitz on $K^4$ with Lipschitz constant $M_\mu$ bounded by
	\begin{align*}
		M_\mu &\leq \diam(K) \Ebb_\mu[M_1^2]\\
		&\leq  \diam(K)\cdot 3p^2 c_{p-1}^2(4\Ebb_\mu[d^{2p-2}(x_0,Y)] + 4\Ebb_\mu[d^{p-1}(x_0,Y)]^2 + 16(\diam(K))^{2p-2})\\
		&\leq 24p^2c_{p-1}^2 \diam(K) (\Ebb_\mu[d^{2p-2}(x_0,Y)] + 2(\diam(K))^{2p-2}).
	\end{align*}
	In the last inequality, we used Cauchy-Schwarz to get $\Ebb_\mu[d^{p-1}(x_0,Y)]^2 \leq \Ebb_\mu[d^{2p-2}(x_0,Y)]$.
\end{proof}

In particular, the proof of Lemma~\ref{lem:Z-R-cts} directly implies Lemma~\ref{lemma:continuity-Z-R} from the main body of the paper.

\begin{proof}[Proof of Lemma \ref{lemma:uniform_convergence_covariance}]
	We first start by observing that since $(X,d)$ is separable, on an event $E$ of full probability, $\bar\mu_n$ converges to $\mu$ for the weak topology.
	Next, fix $x_0\in K$. By the law of large numbers, on an event $F$ of full probability one has $\frac{1}{n}\sum_{i=1}^{n}d^{p-1}(x_0,Y_i) \to \Ebb_{\mu}d^{p-1}(x_0,Y_1)$ and $\frac{1}{n}\sum_{i=1}^{n}d^{2p-2}(x_0,Y_i) \to \Ebb_{\mu}d^{2p-2}(x_0,Y)$.
	By Lemma \ref{lem:UI}, on $E\cap F$, for all $x,x'\in K$, one has $\frac{1}{n}\sum_{i=1}^{n}(d^p(x,Y_i)-d^p(x',Y_i))\to \Ebb_\mu(d^p(x,Y)-d^p(x',Y))$ and $\frac{1}{n}\sum_{i=1}^{n}(d^p(x,Y_i)-d^p(x',Y_i))^2)\to \Ebb_\mu[(d^p(x,Y)-d^p(x',Y))^2]$. In particular, on $E\cap F$, for any $x,x'\in K$, we have
	\begin{equation}\label{eq:useful_2}
		R_{\bar\mu_n}(x,x')\to R_\mu(x,x').
	\end{equation}
	In the rest of the proof, we suppose that $E\cap F$ is met. Because $K^2$ is compact, let $(x_n,x_n')\in K^2$ such that $\|R_{\bar\mu_n}-R_\mu\|_\infty = |R_{\bar\mu_n}(x_n,x_n') - R_\mu(x_n,x_n')|$. To show that $\|R_{\bar\mu_n}-R_\mu\|_\infty\to 0$, it suffices to show that for any sequence $\{n_k\}_{k\in\Nbb}$, there is a subsequence $\{k_j\}_{j\in\Nbb}$ such that $\|R_{\bar\mu_{n_{k_j}}}-R_\mu\|_\infty\to 0$.
	Fix the sequence $\{n_k\}_{k\in \Nbb}$.
	By compactness of $K$, there exists a subsequence $\{k_j\}_{j\in \Nbb}$ and $x,x'\in K$ such that $x_{n_{k_j}}\to x$ and $x_{n_{k_j}}'\to x'$. Next, we use the estimates from Lemma~\ref{lem:Z-R-cts}
	showing that $R_{\bar\mu_n}$ is Lipschitz to obtain
	\begin{align*}
		&|R_{\bar\mu_n}(x_n,x_n')-R_{\bar\mu_n}(x,x')| \leq 2D((x_n,x_n'),(x,x')) \diam(K)\\
		&\quad\quad  \cdot 24p^2 c_{p-1}^2 \left(\frac{1}{n}\sum_{i=1}^n d^{2p-2}(x_0,Y_i) + 2(\diam(K))^{2p-2}\right).
	\end{align*}
	Since $(x_{n_{k_j}},x_{n_{k_j}}')\to (x,x')$, the above equation with the convergence hypothesis from $F$ show that $|R_{\bar\mu_{n_{k_j}}}(x_{n_{k_j}},x_{n_{k_j}}')-R_{\bar\mu_{n_{k_j}}}(x,x')|\to 0$. Now note that
	\begin{equation*}
		\begin{split}
			\|R_{\bar\mu_{n_{k_j}}} - R_\mu\|_\infty \leq |R_{\bar\mu_{n_{k_j}}}(x_{n_{k_j}},x_{n_{k_j}}')-R_{\bar\mu_{n_{k_j}}}(x,x')| + |R_{\bar\mu_{n_{k_j}}}(x,x')-R_\mu(x,x')|.
		\end{split}
	\end{equation*}
	Using \eqref{eq:useful_2} implies $\|R_{\bar\mu_{n_{k_j}}} - R_\mu\|_\infty\to 0$ on $E\cap F$. Thus $\|R_{\bar\mu_n}-R_\mu\|_\infty\to 0$ on $E\cap F$.
\end{proof}

\begin{proof}[Proof of Proposition~\ref{prop:functional_CLT}]
	First, we recall the main result of \cite[Theorem 1]{jain1975central} which provides a sufficient condition for an IID sequence of $C(S)$-valued random variables $X_1,X_2,\ldots$ to satisfy the central limit theorem, where $(S,\rho)$ is a Dudley space.
	Precisely, the result states that if $\Ebb[f(X_1)]=0$ for any $f\in (C(S))^*$ (where $(C(S))^{\ast}$ denotes the space of continuous linear functionals on $C(S)$), if $\sup_{t\in S}\Ebb[X_1^2(t)]<\infty$, and if there exists a random variables $M_1$ with $\Ebb[M_1^2]<\infty$ and
	\begin{equation*}
		|X_1(s)-X_1(t)|\leq M_1 \rho(s,t) \quad \textrm{a.s.}, 
	\end{equation*}
	then the sequence of random variables $\{n^{-1/2}(X_1+\ldots +X_n)\}_{n\in\Nbb}$ convergences in distribution, with respect to the topology of uniform convergence, to the Gaussian measure on $S$ whose covariance kernel is given by $L(s,t) := \Cov(X_1(s),X_1(t))$.
	In this proof for any fixed compact $K\subseteq X$, we check that these conditions hold with the space $K\times K$, which is compact with respect to the metric given by $D((x,x'),(x'',x'''))=d(x,x'')+d(x',x''')$ for any $x,x',x'',x'''\in K$.
	By the hypothesis that $(X,d)$ is HBD, it follows that $(K,d)$ is a Dudley space.
	We then check that $(K\times K,D)$ a Dudley space; this is immediate from the observation that for any $\varepsilon>0$, $\CovNum_{K\times K}(\varepsilon) \leq (\CovNum_K(\varepsilon))^2.$	Now, Lemma~\ref{lem:Z-R-cts} 
	exactly guarantees that $Z_1$ is Lipschitz continuous with Lipschitz constant $M_1$ satisfying $\E_{\mu}[M_1^2] < \infty$.
	Also, because $Z_1$ is Lipschitz, $\|{Z_1}|_{K\times K}\|_{\infty} \leq Z_1(x_0,x_0) + 2M_1\diam(K) = 2M_1\diam(K)$. In particular, we have $\Ebb_{\mu}[\|{Z_1}|_{K\times K}\|_\infty^2]<\infty$ which implies that $\sup_{x,x'\in K} \Ebb_{\mu}[Z_1^2(x,x')]<\infty$. Additionally, we have $\Ebb_{\mu}[\|{Z_1}|_{K\times K}\|_\infty]<\infty$.
	Thus, the Riesz representation theorem together with Fubini's theorem imply that for any positive continuous linear functional $f\in (C(K^2))^*$, we have $\Ebb_{\mu}[f(Z_1)]=0$, and as a result, for any continuous linear functional $f\in (C(K^2))^*$, we have $\Ebb_{\mu}[f(Z_1)]=0$.
	We have now checked the conditions from the main result \cite[Theorem 1]{jain1975central} which implies that the sequence $Z_1,Z_2,\ldots$ satisfies the central limit theorem; this implies that on the compact $K^2$, the sequence $\{n^{-1/2}G_n\}_{n\in\Nbb}$ converges in distribution to $\Gcal_\mu$. 
	
	We checked the convergence on compacts of $X\times X$ of the form $K\times K$.
	For a generic compact $\tilde K\subseteq X\times X$, let $K = \bigcup_{(x,x')\in\tilde K}\{x,x'\} \subseteq X$. Because $\tilde K$ is compact hence bounded, one can check that $K$ is also bounded and closed.
	Thus, because $X$ is Heine-Borel, $K$ is compact and $\tilde K\subseteq K\times K$.
	Therefore, the convergence in distribution of $\{n^{-1/2}G_n\}_{n\in\Nbb}$ to $\Gcal_\mu$ on $K\times K$ implies the convergence on $\tilde K$.
	The proof of the proposition is now complete.
\end{proof}

\begin{proof}[Proof of Proposition~\ref{prop:functional_LIL}]
	For a Banach space $(B,\|\cdot\|)$, and an IID sequence $X_1,X_2,\ldots$ of $B$-valued random variables, we say that (the law of) $X_1$ satisfies the LIL if $\{(2n\log\log n)^{-1/2}(X_1+\ldots+X_n)\}_{n\in\Nbb}$ is a relatively compact sequence in $B$ and if its closure is the unit ball from the corresponding RKHS.
	\cite[Corollary 1.3]{ledoux1988characterization} shows that the LIL holds for a an IID sequence of $B$-valued random variables $X_1,X_2,\ldots$ whenever $X_1$ satisfies the central limit theorem (CLT), in that $\{n^{-1/2}(X_1+\ldots+X_n)\}_{n\in\Nbb}$ converges in distribution to some $B$-valued Gaussian random variable, and whenever we additionally have $\Ebb[\|X_1\|^2/\log\log \|X_1\|]<\infty$.         
	In our case, we consider the random variable $Z_1|_{K\times K}$ on the set $C(K\times K)$, which is a Banach space when equipped with the topology of uniform convergence.
	In Proposition~\ref{prop:functional_CLT}, we showed that on any compact $K\subseteq X$, the random variable $Z_1|_{K\times K}$ satisfies the CLT and that $\Ebb[\|Z_1|_{K\times K}\|_\infty^2]<\infty$.
	As a result, the conditions from \cite[Corollary 1.3]{ledoux1988characterization} are met, which implies that the LIL holds for our purposes.
	Precisely, this shows that $\{(2n\log\log n)^{-1/2} G_{K,n}\}_{n\in \Nbb}$ is relatively compact and has $\Bcal_{K,\mu}$ as its closure.
\end{proof}

\begin{proof}[Proof of Lemma~\ref{lemma:formula_cp(mu)}]
	Let $f\in \Bcal_{K,\mu}$ and $x,x'\in F_p(\mu)$. Then, by the reproducing property of the RKHS $\Hcal_{K,\mu}$,
	\begin{align*}
		f(x,x') = \langle f, R_{K,\mu}(x,x',\cdot,\cdot)\rangle_{K,\mu} &\leq \|f\|_{K,\mu}\|R_{K,\mu}(x,x',\cdot,\cdot)\|_{K,\mu}\\
		&\leq \sqrt{R_{K,\mu}(x,x',x,x')}\\
		&= \sqrt{\Var_\mu(d^p(x,Y_1) - d^p(x',Y_1))} \leq \frac{\sigma_p(\mu)}{\sqrt 2},
	\end{align*}
	where in the first inequality, we used the Cauchy-Schwarz inequality. Thus,
	\begin{equation*}
		\sqrt 2\cdot \sup_{f\in \Bcal_{K,\mu}} \sup_{x,x'\in F_p(\mu)} f(x,x') \leq \sigma_p(\mu).
	\end{equation*}
	Next, for any $x,x'\in F_p(\mu)$, let $f_{x,x'} = \frac{R_{K,\mu}(x,x',\cdot,\cdot)}{\|R_{K,\mu}(x,x',\cdot,\cdot)\|_{K,\mu}}$. By construction, $f_{x,x'}\in \Bcal_{K,\mu}$ and
	\begin{equation*}
		f_{x,x'}(x,x') = \langle f_{x,x'}, R_{K,\mu}(x,x',\cdot,\cdot)\rangle_{K,\mu} = \|R_{K,\mu}(x,x',\cdot,\cdot)\|_{K,\mu}
		= \sqrt{\Var_\mu(d^p(x,Y_1) - d^p(x',Y_1))}.
	\end{equation*}
	Taking the supremum over $x,x'\in F_p(\mu)$ gives the other inequality
	\begin{equation*}
		\sqrt 2\cdot \sup_{f\in \Bcal_{K,\mu}} \sup_{x,x'\in F_p(\mu)} f(x,x') \geq \sigma_p(\mu),
	\end{equation*}
	which ends the proof.
\end{proof}

We will also need the following result, which is certainly standard.

\begin{lemma}\label{lem:RKHS-ball-cpt}
	Let $(X,d)$ be a metric space, $p\ge 1$, and $\mu\in \Pcal_{2p-2}(X)$.
	For any compact $K\subseteq X$, the set $\Bcal_{K,\mu} \subseteq C(K\times K)$ is compact in the topology of uniform convergence on $K$.
\end{lemma}

\begin{proof}
	The result follows from the Arzela-Ascoli theorem if we show that $\Bcal_{K,\mu}$ is uniformly bounded and uniformly equicontinuous.
	To see uniform boundedness, take arbitrary $f\in \Bcal_{K,\mu}$ and $x,x'\in K$, and use the reproducing property and Cauchy-Schwarz, to get
	\begin{align*}
		|f(x,x')| &= |\langle f, R_{K,\mu}(x,x',\cdot,\cdot)\rangle_{K,\mu}| \\
		&\le \| f\|_{K,\mu} \|R_{K,\mu}(x,x',\cdot,\cdot)\|_{K,\mu} \\
		&\le \|R_{K,\mu}(x,x',\cdot,\cdot)\|_{K,\mu} \\
		&= \sqrt{R_{K,\mu}(x,x')}
	\end{align*}
	Now write $M$ for the Lipschitz constant of $R_{K,\mu}$ on $K^4$, which exists by Lemma~\ref{lem:Z-R-cts}
	and fix any $x_0,x_0'\in K$.
	it follows that
	\begin{align*}
		f(x,x')^2 &\le R_{K,\mu}(x,x') \\
		&\le R_{K,\mu}(x_0,x_0')+M\,D((x,x'),(x_0,x_0')) \\
		&\le R_{K,\mu}(x_0,x_0')+2M\diam(K).
	\end{align*}
	Since the right side does not depend on $f\in B_{K,\mu}$ or $x,x'\in K$, we have shown uniform boundedness.
	For uniform equicontinuity, we use a similar approach, so take arbitrary $f\in \Bcal_{K,\mu}$ and $x,x',x'',x'''\in K$, and use the reproducing property and Cauchy-Schwarz, to get
	\begin{align*}
		|f(x,x')-f(x'',x''')| &= |\langle f,R_{K,\mu}(x,x',\cdot,\cdot)-R_{K,\mu}(x'',x''',\cdot,\cdot)\rangle_{K,\mu}| \\
		&\le \|f\|_{K,\mu}\|R_{K,\mu}(x,x',\cdot,\cdot)-R_{K,\mu}(x'',x''',\cdot,\cdot)\|_{K,\mu} \\
		&\le \|R_{K,\mu}(x,x',\cdot,\cdot)-R_{K,\mu}(x'',x''',\cdot,\cdot)\|_{K,\mu}  \\
		&= \sqrt{R_{K,\mu}(x,x')+R_{K,\mu}(x'',x''')-2R_{K,\mu}(x,x',x'',x''')}\\
		&\leq \sqrt{2MD((x,x'),(x'',x'''))}.
	\end{align*}
	This ends the proof that $\Bcal_{K,\mu}$ is uniformly equicontinuous and completes the proof of the lemma.
\end{proof}

\begin{proof}[Proof of Lemma~\ref{lemma:sigma_continuity}]
	By Lemma~\ref{lem:RKHS-ball-cpt}, 
	the set $\Bcal_{K_1,\mu}$ is compact in the topology of uniform convergence.
	In particular, there exists a sequence $\delta_i\to 0$ and $(f_i,x_i,x_i')\in \Bcal_{K_1,\mu}\times F_p(\mu)\times K_{\delta_i}$ as well as $(f,x,x')\in B_{K_1,\mu}\times F_p(\mu)\times K_1$ such that $f_i\to f$, $x_i\to x$, $x_i'\to x'$ and such that $\sqrt 2\cdot f_i(x_i,x_i')$ converges to the limsup of the left-hand side of the above equation as $\delta\to 0$. By continuity of $f$, we have $f_i(x_i,x_i')\to f(x,x')$. Since $\delta_i\to 0$, we also have $x'\in F_p(\mu).$ Therefore, using Lemma~\ref{lemma:formula_cp(mu)} 
	gives $\sqrt 2\cdot f_i(x_i,x_i')\to \sqrt 2 \cdot f(x,x')\leq\sigma_p(\mu)$. The other inequality is a direct immediate from $K_\delta\subseteq F_p(\mu)$ and Lemma~\ref{lemma:formula_cp(mu)}.
\end{proof}

\section{Proofs for the strong law of large numbers (Subsection~\ref{subsec:SLLN})}
\label{sec:SLLN}

A key tool in the proofs of our main results is a suitably powerful form of a strong law of large numbers (SLLN) for relaxed Fr\'echet means in the one-sided Hausdorff distance.
While this problem has already been studied in the recent works \cite{SchoetzSLLN,EvansJaffeSLLN}, we provide a strengthening of these results which we believe will be of independent interest.
In particular, this work resolves a conjecture of Evans-Jaffe in the affirmative (see \cite[Remark~3.11]{EvansJaffeSLLN}).

Let us briefly comment on the existing SLLNs.
The first is \cite{SchoetzSLLN} in which the author proved the SLLN for a $(p-1)$th moment assumption and deterministic relaxation rates.
The second is \cite{EvansJaffeSLLN} in which the authors proved the SLLN for a $p$th moment assumption and possibly random relaxation rates.
In the course of this work, it became necessary to use the SLLN for a $(p-1)$th moment assumption and possibly random relaxation rates; the main result of this section is that the SLLN indeed holds in this strengthened setting.

To begin, we prove a single result which contains most of the hard work of this section.
Indeed, it will immediately imply the remaining results we wish to prove.

\begin{lemma}\label{lem:Fp-cpt-opt}
	Let $(X,d)$ be a HB space and $p\ge 1$.
	Suppose $\{(\mu_n,C_n,\varepsilon_n)\}_{n\in\Nbb}$ is a sequence in $\mathcal{P}_{p-1}(X)\times \closed(X)\times[0,\infty)$ with $\mu_n\to \mu$ in $\tau^{p-1}_w$, $\kurlimit_{n\to\infty}C_n = C$, and $\varepsilon_n\to \varepsilon$ for some $(\mu,C,\varepsilon)\in\mathcal{P}_{p-1}(X)\times \closed(X)\times[0,\infty)$.
	Then, for any sequence $\{x_n\}_{n\in\Nbb}$ with $x_n\in F_p(\mu_n,C_n,\varepsilon_n)$, there exists a subsequence $\{n_j\}_{j\in\Nbb}$ and a point $x\in F_p(\mu,C,\varepsilon)$ with $x_{n_j}\to x$.
\end{lemma}

\begin{proof}
	Our approach is to fix $o\in C$ and then to derive some lower bounds on the Fr\'echet functionals $W_p$ whenever the argument ranges of a suitable large ball.
	To set things up, use  $C\subseteq \kurinner_{n\in\Nbb}C_n$ to get that there exists a subsequence $\{n_k\}_{k\in\Nbb}$ and a sequence $\{o_k\}_{k\in\Nbb}$ with $o_k\in C_{n_k}$ such that $o_k\to o$.
	
	Next, we derive two probabilistic inequalities that will be used to break up the integral for $W_p$ into different regimes.
	For the first inequality, note that $\mu_{n_k}\to\mu$ in $\tau_w^{p-1}$ implies \cite[Definition~6.8]{Villani} that we have
	\begin{equation*}
		\lim_{M\to\infty}\limsup_{k\to\infty}\int_{X\setminus B_{M}(o)}d^{p-1}(o,y)d\mu_{n_k}(y)\to 0.
	\end{equation*}
	Now observe that, for each $M>0$, we have $X\setminus B_{M}(o_k)\subseteq X\setminus B_{M/2}(o)$ for sufficiently large $k\in\Nbb$.
	Thus, we apply \eqref{eqn:ineq-1} to get
	\begin{align*}
		\limsup_{k\to\infty}&\int_{X\setminus B_{M}(o_k)}d^{p-1}(o_k,y)d\mu_{n_k}(y) \\
		&\le \limsup_{k\to\infty}\int_{X\setminus B_{M}(o_k)}c_{p-1}(d^{p-1}(o_k,o)+d^{p-1}(o,y))d\mu_{n_k}(y)\\
		&\le c_{p-1}\limsup_{k\to\infty}\left(d^{p-1}(o_k,o)+\int_{X\setminus B_{M}(o_k)}d^{p-1}(o,y)d\mu_{n_k}(y)\right)\\
		&= c_{p-1}\limsup_{k\to\infty}\int_{X\setminus B_{M}(o_k)}d^{p-1}(o,y)\,d\mu_{n_k}(y)\\
		&\le c_{p-1}\limsup_{k\to\infty}\int_{X\setminus B_{\frac{1}{2}M}(o)}d^{p-1}(o,y)\,d\mu_{n_k}(y).
	\end{align*}
	The two previous two displays together yield
	\begin{equation*}
		\lim_{M\to\infty}    \limsup_{k\to\infty}\int_{X\setminus B_{M}(o_k)}d^{p-1}(o_k,y)d\mu_{n_k}(y) = 0.
	\end{equation*}
	For the second inequality, note that $\mu_{n_k}\to \mu$ in $\tau_w^{p-1}$ implies that $\{\mu_{n_k}\}_{k\in\Nbb}$ is relatively compact in $\tau_w$, hence tight by Prokhorov.
	This means we have
	\begin{equation*}
		\lim_{M\to\infty}\sup_{k\in\Nbb}\mu_{n_k}(X\setminus B_{M}(o)) \to 0,
	\end{equation*}
	and the same argument as above shows that this implies
	\begin{equation*}
		\lim_{M\to\infty}\limsup_{k\to\infty}\mu_{n_k}(X\setminus B_{M}(o_k)) \to 0.
	\end{equation*}
	Combining these two inequalities, we see that it is possible to choose $M>0$ large enough so that for all sufficiently large $k\in\Nbb$ we have
	\begin{equation}\label{eqn:moment-bdd-4}
		\int_{X\setminus B_{M}(o_k)}d^{p-1}(o_k,y)d\mu_{n_k}(y) \le \frac{1}{p2^{2p+2}}
	\end{equation}
	and
	\begin{equation}\label{eqn:moment-bdd-3}
		\mu_{n_k}(X\setminus B_{M}(o_k)) \le \frac{1}{p2^{2p}} \le \frac{1}{2}
	\end{equation}
	We can of course also assume $M\ge 1/16$ and $M^p> \sup_{k\in\Nbb}\varepsilon_{n_k}$.
	
	Moving on, we derive two deterministic inequalities, each of which will be applied in one of two different regimes.
	First, use the elementary inequality \eqref{eqn:ineq-1} to get
	\begin{equation}\label{eqn:moment-bdd-2}
		\begin{split}
			d^p(x_{n_k},y)-d^p(o_{k},y) &\ge c_p^{-1}d^p(x_{n_k},o_k) - 2d^p(o_k,y) \\
			&\ge \frac{d^p(x_{n_k},o_k)}{2^{p-1}} - 2d^p(o_k,y) \\
		\end{split}
	\end{equation}
	for all $y\in X$ and $k\in\Nbb$.
	Second, use \eqref{eqn:ineq-2} then \eqref{eqn:ineq-1} to get
	\begin{equation*}
		\begin{split}
			&|d^p(x_{n_k},y)-d^p(o_k,y)| \\
			& \le pd(x_{n_k},o_k)(d^{p-1}(x_{n_k},y)+d^{p-1}(o_k,y)) \\
			& \le pd(x_{n_k},o_k)(c_{p-1}(d^{p-1}(x_{n_k},o_k)+d^{p-1}(o_k,y))+d^{p-1}(o_k,y)) \\
			&= pc_{p-1}d^p(x_{n_k},o_k)+p(1+c_{p-1})d(x_{n_k},o_k)d^{p-1}(o_k,y) \\
			&\leq p2^{p-1}d^p(x_{n_k},o_k)+p2^pd(x_{n_k},o_k)d^{p-1}(o_k,y)
		\end{split}
	\end{equation*}
	for all $y\in X$ and $k\in\Nbb$.
	In particular, this implies
	\begin{equation}\label{eqn:moment-bdd-1}
		d^p(x_{n_k},y)-d^p(o_k,y) \ge -p2^{p-1}d^p(x_{n_k},o_k)-p2^pd(x_{n_k},o_k)d^{p-1}(o_k,y)
	\end{equation}
	for all $y\in X$ and $k\in\Nbb$.
	
	Now we put these pieces together and assume that for the sake of contradiction that we have $\limsup_{k\to\infty}d(x_{n_k},o_k) \ge 17M$.
	In particular, we have $d(x_{n_k},o_k) \ge 16M$ for infinitely many $k\in\Nbb$.
	On the one hand, this means we can use \eqref{eqn:moment-bdd-2} and \eqref{eqn:moment-bdd-3} to get
	\begin{align*}
		\int_{B_{M}(o_k)}(d^p(x_{n_k},y)-d^p(o_k,y))\, d\mu_{n_k}(y) &\geq \left(\frac{d^p(x_{n_k},o_k)}{2^{p-1}} - 2M^p \right)\mu_{n_k}(B_{M}(o_k)) \\
		&\geq \left(\frac{d^p(x_{n_k},o_k)}{2^{p-1}} - 2M^p \right)\frac{1}{2} \\
		&=\frac{d^p(x_{n_k},o_k)}{2^{p}} - M^p
	\end{align*}
	for infinitely many $k\in\Nbb$.
	On the other hand, we can use \eqref{eqn:moment-bdd-1}, \eqref{eqn:moment-bdd-3} and \eqref{eqn:moment-bdd-4} to get
	\begin{align*}
		\int_{X\setminus B_{M}(o_k)}&(d^p(x_{n_k},y)-d^p(o_k,y))\, d\mu_n(y) \\
		&\ge -p 2^{p-1} d^p(x_{n_k},o_k)\mu_n(X\setminus B_{M}(o_k))  - p2^p d(x_{n_k},o_k) \int_{X\setminus B_{M}(o_k)}d^{p-1}(o_k,y)\,d\mu_{n_k}(y) \\
		&\geq  -\frac{d^p(x_{n_k},o_k)}{2^{p+1}} - \frac{d(x_{n_k},o_k)}{2^{p+2}}\\
		&\geq  -\frac{d^p(x_{n_k},o_k)}{2^{p+1}} - \frac{d^p(x_{n_k},o_k)}{2^{p+2}}
	\end{align*}
	for infinitely many $k\in\Nbb$, 
	where in the last line we used $M\ge 1/16$ to see that $d(x_{n_k},o_k) \ge 16M$ implies $d(x_{n_k},o_k) \le d^p(x_{n_k},o_k)$.
	Combining the previous two displays, we get the following:
	For infinitely many $k\in\Nbb$, we have
	\begin{align*}
		W_p(\mu_{n_k},x_{n_k},o_k) &= \int_{B_M(o_k)}(d^p(x_{n_k},y)-d^p(o_k,y))\, d\mu_{n_k}(y) \\
		&\qquad+ \int_{X\setminus B_{M}(o_k)}(d^p(x_{n_k},y)-d^p(o_k,y))\, d\mu_{n_k}(y) \\
		&\ge\frac{d^p(x_{n_k},o_k)}{2^{p}} - M^p   -\frac{d^p(x_{n_k},o_k)}{2^{p+1}} - \frac{d^p(x_{n_k},o_k)}{2^{p+2}} \\
		&=\frac{d^p(x_{n_k},o_k)}{2^{p+2}} - M^p \\
		&= \left(\frac{16^p}{2^{p+2}}-1\right)M^p \\
		&= (2^{3p-2}-1)M^p \\
		&\ge M^p.
	\end{align*}
	In summary, we have shown that
	\begin{equation*}
		\limsup_{k\to\infty}W_p(\mu_{n_k},x_{n_k},o_k) \ge M^p.
	\end{equation*}
	However, this contradicts our assumptions of $x_{n_k}\in F_p(\mu_{n_k},C_{n_k},\varepsilon_{n_k})$ and $M^p > \sup_{k\in\Nbb}\varepsilon_{n_k}$.
	Therefore, we must have $\limsup_{k\to\infty}d(x_{n_k},o_k)< 17M$.
	In particular, this implies that we have $d(x_{n_k},o) \le 18M$ for sufficiently large $k\in\Nbb$.
	
	In other words, we have shown that $x_{n_k}\subseteq B_{18M}(o)$ for sufficiently large $k\in\Nbb$.
	Since $(X,d)$ is a Heine-Borel space, this means there exists some $\{k_j\}_{j\in\Nbb}$ and some $x\in B_{18M}(o)$ with $x_{n_{k_j}}\to x$.
	So, it only remains to show that $x\in F_p(\mu,C,\varepsilon)$.
	To see this, first note that $\kurouter_{n\in\Nbb}C_n\subseteq C$ implies $x\in C$.
	Then, take any $z\in C$ and use $C\subseteq \kurinner_{n\in\Nbb}C_n$, to get some $\{z_j\}_{j\in\Nbb}$ with $z_j\in C_{n_{k_j}}$ such that $z_j\to z$.
	We compute, using the joint continuity of $W_p$ (see \cite[Lemma~3.7]{EvansJaffeSLLN}),
	\begin{align*}
		W_p(\mu,x,o) &= \lim_{j\to\infty}W_p(\mu_{n_{k_j}},x_{n_{k_j}},o_{k_j}) \\
		&\le \lim_{j\to\infty}(W_p(\mu_{n_{k_j}},z_{n_{k_j}},o_{k_j})+\varepsilon_{n_{k_j}}) = W_p(\mu,z,o) + \varepsilon,
	\end{align*}
	and this implies $x\in F_p(\mu,C,\varepsilon)$.
	This finishes the proof.
\end{proof}

\begin{proof}[Proof of Lemma~\ref{lemma:Fp-cpt}]
	First let us show non-emptiness.
	Fix $o\in C$, and get $\{x_n\}_{n\in\Nbb}$ in $C$ with
	\begin{equation*}
		W_p(\mu,x_n,o) \downarrow \inf_{x\in C}W_p(\mu,x,o)
	\end{equation*}
	as $n\to\infty$.
	Then set
	\begin{equation*}
		\varepsilon_n := W_p(\mu,x_n,o)-\inf_{x\in C}W_p(\mu,x,o)
	\end{equation*}
	and note by construction that we have $x_n\in F_p(\mu,C,\varepsilon_n)$.
	By Lemma~\ref{lem:Fp-cpt-opt}, there exists $\{n_j\}_{j\in\Nbb}$ and $x\in F_p(\mu,C,0)$ with $x_{n_j}\to x$.
	In particular, this shows that $F_p(\mu,C,0)$ is non-empty, hence that $F_p(\mu,C,\varepsilon) \supseteq F_p(\mu,C,0)$ is non-empty.
	To show compactness, simply take an arbitrary sequence $\{x_n\}_{n\in\Nbb}$ in $F_p(\mu,C,\varepsilon)$ and apply Lemma~\ref{lem:Fp-cpt-opt}.
\end{proof}

\begin{corollary}\label{cor:Fp-cts}
	Suppose $(X,d)$ is a HB space and $p\ge 1$.
	Also suppose that $\{(\mu_n,C_n,\varepsilon_n)\}_{n\in\Nbb}$ is a sequence in $\mathcal{P}_{p-1}(X)\times \closed(X)\times[0,\infty)$ with $\mu_n\to \mu$ in $\tau^{p-1}_w$, $Lt_{n\to\infty}C_n = C$, and $\varepsilon_n\to \varepsilon$ for some $(\mu,C,\varepsilon)\in\mathcal{P}_{p-1}(X)\times \closed(X)\times[0,\infty)$.
	Then, $\dvechaus(F_p(\mu_n,C_n,\varepsilon_n),F_p(\mu,C,\varepsilon))\to 0$.
\end{corollary}

\begin{proof}
	It suffices to show for all subsequences $\{n_k\}_{k\in\Nbb}$ that there exists a further subsequence $\{k_j\}_{j\in\Nbb}$ such that we have
	\begin{equation*}
		\dvechaus(F_p(\mu_{n_{k_j}},C_{n_{k_j}},\varepsilon_{n_{k_j}}),F_p(\mu,C,\varepsilon))\to 0.
	\end{equation*}
	To show this, use Lemma~\ref{lemma:Fp-cpt} 
	to get for each $k\in\Nbb$ a point $x_k\in F_p(\mu_{n_k},C_{n_k},\varepsilon_{n_k})$ satisfying
	\begin{equation*}
		d(x_k,F_p(\mu,C,\varepsilon)) = \dvechaus(F_p(\mu_{n_{k}},C_{n_{k}},\varepsilon_{n_{k}}),F_p(\mu,C,\varepsilon)).
	\end{equation*}
	Then use Lemma~\ref{lem:Fp-cpt-opt} to get a sequence $\{k_j\}_{j\in\Nbb}$ and a point $x\in F_p(\mu,C,\varepsilon)$ with $x_{k_j}\to x$.
	We of course have
	\begin{align*}
		\dvechaus(&F_p(\mu_{n_{k_j}},C_{n_{k_j}},\varepsilon_{n_{k_j}}),F_p(\mu,C,\varepsilon)) \\
		&= d(x_{k_j},F_p(\mu,C,\varepsilon)) \\
		&\le d(x_{k_j},x) \to 0,
	\end{align*}
	so the result is proved.
\end{proof}

Corollary~\ref{cor:Fp-cts} directly implies the following result for one-sided SLNN for relaxed constrained Fr\'echet means, strengthening the one-sided SLLN results from \cite{EvansJaffeSLLN,SchoetzSLLN}.

\begin{theorem}\label{thm:SLLN}
	Let $(X,d)$ be a HB space, $p\ge 1$, and $\mu\in\Pcal_{p-1}(X)$.
	Let $C$ be any closed set and $\{C_n\}_{n\in\Nbb}$ be any random closed sets. Fix $\varepsilon\geq 0$ and any random relaxation scales $\{\varepsilon_n\}_{n\in\Nbb}$. Then, $\dvechaus(F_p(\bar \mu_n,C_n, \varepsilon_n),F_p(\mu,C,\varepsilon))\to 0$ a.s.
	on $\{\varepsilon_n\to \varepsilon\}\cap\left\{\textrm{Lt}_{n\in\Nbb}C_n = C\right\}$.
\end{theorem}

Theorem~\ref{thm:mean-SLLN}
is a direct consequence of Theorem~\ref{thm:SLLN}. As another consequence, we were also able to deduce a strong consistency result for medoids.

\begin{proof}[Proof of Theorem~\ref{thm:medoid-SLLN}]
	Immediate from Theorem~\ref{thm:SLLN} and the fact that $\kurlimit_{n\in\Nbb}\supp(\bar \mu_n) = \supp(\mu)$ a.s. (see \cite[Proposition~4.2]{EvansJaffeSLLN}).
\end{proof}

For another consequence, we remark that these methods can be used to show a large deviation upper bound for the empirical Fr\'echet mean sets, in the one-sided Hausdorff distance $\dvechaus$, under a much weaker moment assumption than originally expected.
Indeed, one can combine Corollary~\ref{cor:Fp-cts} with the methods of \cite[Subsection~4.2]{JaffeClustering} (see also \cite[Theorem~4.6]{EvansJaffeSLLN}) to show that $\int_{X}\exp(\alpha d(x,y))\, d\mu(y)$ for some $x\in X$ and all $\alpha > 0$ suffices for a large deviations upper bound.
We leave the details to the interested reader.

\section{Proofs omitted from the main general results (Subsection~\ref{subsec:main_proofs})}
\label{sec:proof_general_results}

\begin{proof}[Proof of Lemma~\ref{lemma:finite_sigma}]
For any $x,x'\in X$, we have $\Var_\mu(d^p(x,Y_1)-d^p(x',Y_1)) = R_\mu(x,x',x,x')$.
By Lemma~\ref{lemma:Fp-cpt}, the set $F_p(\mu)$ is compact, and, by Lemma~\ref{lem:Z-R-cts}, the function $R_{\mu}$ is continuous.
Therefore, $R_{F_p(\mu),\mu} = R_{\mu}|_{F_p(\mu)\times F_p(\mu)}$ has a finite maximum value.
\end{proof}

\begin{proof}[Proof of Lemma~\ref{lem:sigma_pos}]
Assume that $F_p(\mu)$ is not a singleton, then take distinct $x,x'\in F_p(\mu)$ and write $r := d(x,x') > 0$.
Now use $x\in \supp(\mu)$ to get $\mu(B_{r/3}(x)) > 0$, and observe that $d(x,Y_1) \leq r/3 <2r/3 \le d(x',Y_1) $ when $Y_1\in B_{r/3}(x)$. Hence, with nonzero probability, $d(x,Y_1)\neq d(x,Y_1)$. Therefore, $\sigma_p(\mu) > 0$.
\end{proof}

We are now ready to give the missing proofs of our main results.

\begin{proof}[Proof of Theorem~\ref{thm:gaussian_tail_one_sided}]
Fix $\delta>0$ and denote $K^\delta:=(F_p(\mu))^\delta$ the $\delta$-thickening of $F_p(\mu)$. Since $X$ is HB and $F_p(\mu)$ is compact by Lemma~\ref{lemma:Fp-cpt}, it follows that $K^\delta$ is compact as well.
By Proposition~\ref{prop:functional_CLT}, $\{n^{-1/2}G_{K^\delta,n}\}_{n\in\Nbb}$ converges to a Gaussian process $G_{K^\delta}\sim \Gcal_\mu$ on $K^\delta\times K^\delta$ in distribution.
Because $K^\delta\times K^\delta$ is compact and the Gaussian process takes values in $C(K^\delta\times K^\delta)$, $G_{K^\delta}$ is a.s. bounded.
Also, define $\sigma_p^2(\mu;\delta) := 2\cdot  \sup_{x,x'\in K^\delta} \Ebb[(G_{K^\delta}(x,x'))^2] = 2\cdot  \sup_{x,x'\in K^\delta} R_\mu(x,x')$, which is finite.
As a result, we can apply the Borell-TIS inequality \cite[Theorem 2.1.1]{adler2007random} which shows that $M_p(\mu;\delta):=\Ebb[\|G_{K^\delta}\|_\infty]<\infty$, and that we have
\begin{equation*}
	\Pbb(\|G_{K^\delta}\|_\infty \geq c ) \leq \exp\left(-\frac{(c-M_p(\mu;\delta))^2}{\sigma_p^2(\mu;\delta)}\right).
\end{equation*}
Therefore, by convergence in distribution, for any $\eta>0$,
\begin{align*}
	\limsup_{n\to\infty} \Pbb_{\mu}(\|n^{-1/2}G_{K^\delta,n}\|_\infty \geq c-\eta ) \leq \exp\left(-\frac{(c-\eta-M_p(\mu;\delta))^2}{\sigma_p^2(\mu;\delta)}\right).
\end{align*}
Now use Theorem~\ref{thm:mean-SLLN} 
to get $\dvechaus(F_p(\bar\mu_n,\varepsilon_n),F_p(\mu))\to 0$ almost surely. In particular, we have convergence in probability, which implies $         \Pbb_{\mu}(F_p(\bar\mu_n,\varepsilon_n)\subseteq K^\delta) \underset{n\to\infty}{\longrightarrow} 1.$ Observe that
\begin{align*}
	&\{F_p(\bar\mu_n,\varepsilon_n)\subseteq K^\delta\} \cap  \{\|n^{-1/2}G_{K^\delta,n}\|_\infty < c-\eta\} \\
	&\subseteq \{\forall x\in F_p(\mu),\forall x'\in F_p(\bar\mu_n,\varepsilon_n), G_{K^\delta,n}(x,x') \leq (c-\eta)\sqrt n\}\\
	&\subseteq \{\forall x\in F_p(\mu),\forall x'\in F_p(\bar\mu_n,\varepsilon_n), W_p(\bar\mu_n,x,x') \leq (c-\eta)n^{-1/2}\}\\
	&\subseteq \{F_p(\mu)\subseteq F_p(\bar\mu_n,\varepsilon_n)\} \cup \{\varepsilon_n\sqrt n \leq c-\eta\},
\end{align*}
where in the second inclusion we use the fact that for any $x\in F_p(\mu)$ and $x'\in X$, $W_p(\mu, x, x') \leq 0$.
By hypothesis,  we have $\liminf_{n\to\infty}\varepsilon\sqrt n \geq c$ a.s. which implies $\Pbb(\varepsilon_n\sqrt n \leq c-\eta)\to 0$ as $n\to\infty$.
As a result, applying the union bound yields
\begin{align*}
	\limsup_{n\to\infty}\Pbb_{\mu}(F_p(\mu)\not\subseteq F_p(\bar\mu_n,\varepsilon_n)) &\leq \limsup_{n\to\infty} \left[\Pbb_{\mu}(\|n^{-1/2}G_{K^\delta,n}\|_\infty \geq c-\eta ) \right.\\
	&\left.+ (1-\Pbb_{\mu}(F_p(\bar\mu_n,\varepsilon_n)\subseteq K^\delta)) + \Pbb_{\mu}(\varepsilon_n\sqrt n \leq c-\eta) \right]\\
	&\leq \exp\left(-\frac{(c-\eta-M_p(\mu;\delta))^2}{\sigma_p^2(\mu;\delta)}\right).
\end{align*}
This holds for all $\eta>0$, hence
\begin{equation*}
	\limsup_{n\to\infty}\Pbb_{\mu}(F_p(\mu)\not\subseteq F_p(\bar\mu_n,\varepsilon_n)) \leq \exp\left(-\frac{(c-M_p(\mu;\delta))^2}{\sigma_p^2(\mu;\delta)}\right).
\end{equation*}
Further, this holds for any $\delta>0$. Now by a.s. continuity of $G\sim\Gcal_\mu$ (see \cite[Theorem~1.3.5]{adler2007random} and the remarks thereafter) on, say, $K^1\times K^1$, and the dominated convergence theorem, we have that $M_p(\mu;\delta)=\Ebb[\|G|_{K^\delta\times K^\delta}\|_\infty] \to \Ebb[\|G|_{F_p(\mu)\times F_p(\mu)}\|_\infty]:=M_p(\mu)$ as $\delta\to 0$. Similarly, by continuity of $R_\mu$, $\sigma_p^2(\mu;\delta) \to \sigma^2(\mu;0) =\sigma^2_p(\mu)$ as $\delta\to 0$, where
\begin{equation*}
	\sigma_p(\mu) = \sqrt 2 \cdot \sup_{x,x'\in F_p(\mu)} \sqrt{ R_\mu(x,x')} 
	= \sqrt 2 \cdot \sup_{x,x'\in F_p(\mu)} \sqrt{\Var_\mu(d^p(x,Y_1)-d^p(x',Y_1))}.
\end{equation*}
As a result, we showed that
\begin{equation*}
	\limsup_{n\to\infty}\Pbb_{\mu}( F_p(\mu)\not\subseteq F_p(\bar\mu_n,\varepsilon_n)) \leq \exp\left(-\frac{(c-M_p(\mu))^2}{\sigma_p^2(\mu)}\right),
\end{equation*}
which ends the proof of the first claim of the theorem.
In particular, this shows that
\begin{equation*}
	\limsup_{n\to\infty}\Pbb_{\mu}(\dvechaus(F_p(\mu), F_p(\bar\mu_n,\varepsilon_n))>0) \leq \exp\left(-\frac{(c-M_p(\mu))^2}{\sigma_p^2(\mu)}\right).
\end{equation*}
Now, by Theorem~\ref{thm:mean-SLLN} 
because almost surely, $\dvechaus(F_p(\bar\mu_n,\varepsilon_n),F_p(\mu))\to 0$, we have
\begin{equation*}
	\Pbb_{\mu}(\dvechaus( F_p(\bar\mu_n,\varepsilon_n),F_p(\mu))>\delta ) \underset{n\to\infty}{\longrightarrow} 0.
\end{equation*}
Summing the two above equations yields the second claim of the theorem.

Last, we give an upper bound on the value $M_p(\mu)$ based on covering numbers.
Recall that $G_{F_p(\mu)}$ is a Gaussian process on $F_p(\mu)\times F_p(\mu)$ on which we can define the pseudo-metric $d_{G}$ such that for any $x,x',x'',x'''\in F_p(\mu)$,
\begin{equation*}
	d_{G}((x,x'),(x'',x''')) = \sqrt{\Ebb[(G_K(x,x')-G_K(x'',x'''))^2]}
	= \sqrt{\Var_{\mu}(Z_1(x,x')-Z_1(x'',x'''))},
\end{equation*}
where using the proof of Lemma~\ref{lem:Z-R-cts},
\begin{align*}
	&|Z_1(x,x')-Z_1(x'',x''')| 
	\leq pd(x,x'')\left(d^{p-1}(x,Y_i)+d^{p-1}(x'',Y_i)\right)\\
	&+ pd(x',x''')\left(d^{p-1}(x',Y_i)+d^{p-1}(x''',Y_i)\right) 
	+ 2p(d(x,x'') + d(x',x''')) \left(\sup_{x\in F_p(\mu)}W_{p-1}(\mu,x)\right)\\
	&\leq pd_{max}\left(d^{p-1}(x,Y_i)+d^{p-1}(x'',Y_i)+ d^{p-1}(x',Y_i)  
	+d^{p-1}(x''',Y_i) + 4\cdot \sup_{x\in F_p(\mu)}W_{p-1}(\mu,x)\right),
\end{align*}
where $d_{max}=\max(d(x,x''),d(x',x'''))$. The first supremum term is finite because $x\mapsto W_{p-1}(\mu,x)$ is continuous on the compact set $F_p(\mu)$, as long as $\mu\in \mathcal{P}_{p-1}(X)$, and that the same argument holds for the supremum of the $x\mapsto W_{2p-2}(\mu,x)$.
Therefore, with $W_k = \sup_{x\in F_p(\mu)} W_k(\mu,x)$, we obtain
\begin{equation*}
	\Var(Z_1(x,x')-Z_1(x'',x''')) \leq 32 p^2 d_{max}^2 (W_{2p-2} + W_{p-1}^2) \leq 64 p^2 d_{max}^2 W_{2p-2},
\end{equation*}
where in the last inequality we used the Cauchy-Schwarz inequality.
Hence,
\begin{equation*}
	d_{G}((x,x'),(x'',x''')) \leq 8 p \sqrt{W_{2p-2}}\max(d(x,x''), d(x',x''')) .
\end{equation*}
Therefore, Dudley's inequality \cite[Theorem~1.3.3]{adler2007random} gives us the bound
\begin{align*}
	M_p(\mu)&\leq 24 \int_0^\infty \sqrt{\log \CovNum(u; (F_p(\mu))^2,d_{G})} du\\
	&\leq 24 \int_0^\infty \sqrt{\log ( \CovNum(u/(8 p \sqrt{W_{2p-2}}); F_p(\mu),d)^2)} du\\
	&= 192  p \sqrt{2\cdot \sup_{x\in F_p(\mu)} W_{2p-2}(\mu,x)}  \int_{0}^{\infty}\sqrt{ \log \CovNum(u; F_p(\mu),d)}\, du.
\end{align*}
Last, the integral term is finite because of the assumption that $(X,d)$ is HBD.
\end{proof}

Note that the assumption of $\liminf_{n\to\infty}\varepsilon_n\sqrt{n} \ge c$ a.s. was only needed to get that $\Pbb_{\mu}(\varepsilon_n\sqrt n \leq c-\eta)\to 0$ as $n\to \infty$ for each $\eta>0 $.
As such, the result still holds under this weaker assumption.

\begin{proof}[Proof of Theorem~\ref{thm:weak_convergence}]
We start by proving claim (i), so suppose that $\varepsilon_n\to 0$ and $1/(\varepsilon_n\sqrt n) \to 0$ hold in probability.
Next, we fix $\delta > 0$, and we claim that there exists some $\varepsilon > 0$ with $F_p(\mu,\varepsilon) \subseteq (F_p(\mu))^{\delta}$.
Towards a contradiction, assume this were not true; then, for all $k\in\Nbb$ there exists $x_k\in F_p(\mu,2^{-k})$ such that $d(x_k,F_p(\mu))>\delta$.
But because $F_p(\mu,1)$ is compact from Lemma~\ref{lemma:Fp-cpt}, there must exist $x\in F_p(\mu)$ and a subsequence $\{k_{\ell}\}$ such that $x_{k_{\ell}}\to x$.
By continuity of the functional $W_p(\mu,\cdot,\cdot)$ \cite[Lemma~3.7]{EvansJaffeSLLN}, this implies $x\in F_p(\mu)$.
But, we also have $d(x,F_p(\mu))\ge\delta$, a contradiction.
Thus, we can get $\varepsilon> 0$ such that $F_p(\mu,\varepsilon) \subseteq (F_p(\mu))^{\delta}$.
In other words, $\dvechaus(F_p(\mu,\varepsilon),F_p(\mu))\leq \delta$.

Now recall from Theorem~\ref{thm:SLLN} 
that $\dvechaus(F_p(\bar\mu_n,\varepsilon),F_p(\mu,\varepsilon) ) \to 0$
almost surely, hence
\begin{equation*}
	\lim_{n\to\infty}\Pbb_{\mu}(\dvechaus(F_p(\bar\mu_n,\varepsilon),F_p(\mu,\varepsilon)) \geq \delta) = 0.
\end{equation*}
Also, using the triangle inequality for the one-sided Hausdorff distance, we get
\begin{align*}
	\{\varepsilon_n \leq \varepsilon \} \cap \{\dvechaus(F_p(\bar\mu_n,\varepsilon),F_p(\mu,\varepsilon)) < \delta  \} 
	&\subseteq \{ \dvechaus(F_p(\bar\mu_n,\varepsilon_n),F_p(\mu,\varepsilon)) < \delta  \}\\
	&\subseteq \{ \dvechaus(F_p(\bar\mu_n,\varepsilon_n),F_p(\mu)) < \delta + \dvechaus(F_p(\mu,\varepsilon),F_p(\mu)) \}\\
	&\subseteq \{ \dvechaus(F_p(\bar\mu_n,\varepsilon_n),F_p(\mu)) < 2\delta\}.
\end{align*}
Combining the previous displays along with $\Pbb_{\mu}(\varepsilon_n > \varepsilon)\to 0$, we get \begin{equation}\label{eq:onesided_convergence_probability}
	\lim_{n\to\infty}\Pbb(\dvechaus(F_p(\bar\mu_n,\varepsilon_n),F_p(\mu)) \geq  2\delta) = 0,
\end{equation}
by the union bound. Next, for any $c>M_p(\mu)$, Theorem~\ref{thm:gaussian_tail_one_sided}
and $\varepsilon_n\in \omega_p(n^{-1/2})$ imply that
\begin{align*}
	\limsup_{n\to\infty} \Pbb_\mu(&\dvechaus(F_p(\mu),F_p(\bar\mu_n,\varepsilon_n)) \geq \delta) \\
	&\leq \limsup_{n\to\infty} (\Pbb_\mu(\dvechaus(F_p(\mu),F_p(\bar\mu_n,cn^{-1/2})) \geq \delta) + \Pbb_{\mu}(\varepsilon_n\leq c n^{-1/2}))\\
	&\leq \limsup_{n\to\infty} \Pbb_\mu(\dhaus(F_p(\mu),F_p(\bar\mu_n,cn^{-1/2})) \geq \delta) \leq \exp\left(-\frac{(c-M_{p}(\mu))^2}{\sigma_{p}^2(\mu)}\right),
\end{align*}
Since this equation holds for any $c>M_p(\mu)$, we showed that
\begin{equation*}
	\lim_{n\to\infty} \Pbb_\mu(\dvechaus(F_p(\mu),F_p(\bar\mu_n,\varepsilon_n)) \geq \delta) = 0.
\end{equation*}
Combining this equation with \eqref{eq:onesided_convergence_probability} gives $\dhaus(F_p(\bar\mu_n,\varepsilon_n),F_p(\mu))\to 0$ in probability.

We now turn to claim (ii), so suppose that $\sigma_p(\mu)>0$.
Since $F_p(\mu)$ is compact, there exists $x,x'\in F_p(\mu)$ such that $\sigma_p(\mu)^2/2 =R_{\mu}(x,x',x,x') = \Var_\mu(d^p(x,Y_1) - d^p(x',Y_1))>0$.
Now for $G\sim \Gcal_\mu$, the random variable $G(x,x')$ follows a Gaussian distribution $N(0,\frac{1}{2}\sigma_p(\mu)^2)$. 
In fact, the process $\tilde G(x'') = G(x'',x')-G(x,x')$ for $x''\in B_1(x)$ is a Gaussian process;
since $Z_1(x_i,x')-Z_1(x,x') = Z_1(x_i,x)$ almost surely (in fact, everywhere), we see that $\tilde{G}$ has covariance function for any $x_1,x_2\in B_1(x)$,
\begin{align*} 
	\Cov(\tilde G(x_1),\tilde G(x_2))
	&=\Cov(Z_1(x_1,x')-Z_1(x,x'), Z_1(x_2,x')-Z_1(x,x') )\\
	&= \Cov(Z_1(x_1,x), Z_1(x_2,x) ) = R_\mu(x_1,x,x_2,x).
\end{align*}

Next for each $\eta>0$, we give probabilistic estimates on $\|\tilde G|_{B_\eta(x)}\|_\infty$.
For any $x_1,x_2\in B_1(x)$, define
\begin{equation*}
	d_{\tilde G}(x_1,x_2) := \sqrt{\Ebb_{\mu}[(\tilde G(x_1) - \tilde G(x_2))^2]} = \sqrt{\Var_{\mu}(Z_1(x_1,x_2))}.
\end{equation*}
Now observe that we have
\begin{align*}
	|Z_1(x_1,x_2)| &\leq p d(x_1,x_2)(d^{p-1}(x_1,Y_1)+d^{p-1}(x_2,Y_1) + \Ebb[d^{p-1}(x_1,Y_1)+d^{p-1}(x_2,Y_1)])\\
	&\leq 2p d(x_1,x_2)(d^{p-1}(x,Y_1)+ W_{p-1}(\mu,x) +2c_{p-1})
\end{align*}
hence
\begin{equation}\label{eq:useful_eq}
	d_{\tilde G}(x_1,x_2) \leq 2p M(x) d(x_1,x_2),
\end{equation}
where $M(x) :=(3(W_{2p-2}(\mu,x)+W_{p-1}(\mu,x)^2 + 4c_{p-1}^2))^{1/2} < \infty$.
Consequently, similarly to the proof of Theorem~\ref{thm:gaussian_tail_one_sided}, we use Dudley's inequality \cite[Theorem~1.3.3]{adler2007random} to get
\begin{align*}
	\Ebb[&\|\tilde G|_{B_\eta(x)}\|_\infty] \leq 24 \int_0^\infty \sqrt{\log\CovNum(u;B_\eta(x),d_{\tilde G})}du\\
	&\leq 48 p M(x) \int_0^\infty \sqrt{\log\CovNum(u;B_\eta(x),d)}du \leq  48 p M(x) \int_0^\eta \sqrt{\log\CovNum(u;B_1(x),d)}du.
\end{align*}
Since $X$ is HBD, the integral is finite.
Thus, by dominated convergence, we can choose $0<\eta(\delta)\leq 1/(2\sqrt 2\cdot p M(x) \delta)$ such that $\Ebb[\|\tilde G|_{B_{\eta(\delta)}(x)}\|_\infty] \leq \delta$.
Next, \eqref{eq:useful_eq} gives
\begin{equation*}\label{eqn:B-TIS}
	\begin{split}
		2 \cdot \sup_{x_1\in B_{\eta(\delta)}(x)} \Var(\tilde G(x_1)) = 2\cdot  \sup_{x_1\in B_{\eta(\delta)}(x)} \Var_{\mu}(Z_1(x_1,x)) \leq 8p^2 M(x)^2 {\eta(\delta)}^2\leq \delta^3.
	\end{split}
\end{equation*}
We now apply the Borell-TIS inequality \cite[Theorem 3.6.1]{adler2007random} together with the two previous equations which imply $   \Pbb(\|\tilde G|_{B_{\eta(\delta)}(x)}\|_\infty \geq \delta + c) \leq e^{-(c/\delta)^2}$
for any $c\geq 0$.
Suppose that $\varepsilon_n$ is not $\omega(n^{-1/2})$. There exists $M \geq 2$ such that $\varepsilon_n\sqrt n \leq  M$ for infinitely many $n\in\Nbb$.
Next note that $G(x,x')\sim N(0,\sigma_p(\mu)^2/2)$ implies $\beta:=\Pbb(G(x,x') > 2M)>0$.
Now observe by the definition of $\tilde{G}$ that
\begin{equation*}
	\{G(x,x') > 2M\} \cap\{ \|\tilde G|_{B_{\eta(\delta)}(x)}\|_\infty \leq M\}
	\subseteq \{ \forall x''\in B_{\eta(\delta)}(x), G(x'',x')> -M+2M = M\}.
\end{equation*}
Combining this with $\Pbb(\|\tilde G|_{B_{\eta(\delta)}(x)}\|_\infty \geq \delta + 1) \leq e^{-1/\delta^2}$, choosing $\delta^\star = \min(1,1/\sqrt{\log (2/\beta)})$ gives
\begin{equation*}
	\Pbb(\forall x''\in B_{\eta(\delta^\star)}(x), G(x'',x')>  M) \geq \beta - \Pbb(\|\tilde G|_{B_{\eta(\delta^\star)}(x)}\|_\infty \geq \delta^\star + 1) \geq \frac{\beta}{2}.
\end{equation*}
For simplicity, let $r = \eta(\delta^\star)>0$. Recall that Proposition~\ref{prop:functional_CLT} 
implies that $\{n^{-1/2}G_{K^1,n}\}_{n\in\Nbb}$ converges in distribution to $G\sim \Gcal_\mu$ on $K^1\times K^1 = F_p(\mu)\times F_p(\mu)$ so that
\begin{align*}
	\lim_{n\to\infty} \Pbb_{\mu}\left(F_p\left(\bar\mu_n,\frac{M}{\sqrt n}\right) \cap B_{\frac{r}{2}}(x) = \emptyset\right) &\geq \lim_{n\to\infty}\Pbb_{\mu}\left(\forall x''\in  B_r(x), \frac{1}{n}G_n(x'',x')>  \frac{M}{\sqrt n}\right)\\
	& = \Pbb(\forall x''\in B_r(x), G(x'',x')>  M) \geq \frac{\beta}{2}.
\end{align*}
Also, if $\varepsilon_{n_k} \le Mn_k^{-1/2}$ then we have
\begin{align*}
	\left\{F_p\left(\bar\mu_{n_k},\frac{M}{\sqrt {n_k}}\right)\cap B_{\frac{r}{2}}(x) = \emptyset\right\} &\subseteq \{F_p(\bar\mu_{n_k},\varepsilon_{n_k})\cap B_{r/2}(x) = \emptyset\}
	\subseteq \left\{\dvechaus(F_p(\mu),F_p(\bar\mu_{n_k},\varepsilon_{n_k}))\geq \frac{r}{2}\right\}.
\end{align*}
Combining the two previous displays, along with the fact that $\varepsilon_n \le Mn^{-1/2}$ infinitely often,
\begin{equation*}
	\limsup_{n\to\infty} \Pbb_{\mu}(\dvechaus(F_p(\mu),F_p(\bar\mu_n,\varepsilon_n)\ge r/2) \ge \lim_{n\to\infty}\Pbb_{\mu}\left(F_p\left(\bar\mu_n,\frac{M}{\sqrt n}\right) \cap B_{\frac{r}{2}}(x) = \emptyset\right) \geq \frac{\beta}{2}>0.
\end{equation*}
In particular, $F_p(\bar\mu_n,\varepsilon_n)$ is not weakly consistent, which ends the proof.
\end{proof}

\begin{proof}[Proof of Proposition~\ref{prop:example_square}]
We start by noting that $(X,d)$ is a compact Dudley space, and that $F_p(\mu) = \{(1,0),(-1,0)\}$. Further, by symmetry, if $(1,0)\in F_p(\bar\mu_n,\varepsilon_n)$, then $(-1,0)\in F_p(\bar\mu_n,\varepsilon_n)$, which implies $\dvechaus(F_p(\mu),F_p(\bar\mu_n,\varepsilon_n)) = 0$. Thus, we have
\begin{align*}
	\left\{ (1,0)\notin F_p(\bar\mu_n,\varepsilon_n) \right\} = \left\{\dvechaus(F_p(\mu),F_p(\bar\mu_n,\varepsilon_n)) = 1\right\},\\
	\left\{ (1,0)\in F_p(\bar\mu_n,\varepsilon_n) \right\} = \left\{\dvechaus(F_p(\mu),F_p(\bar\mu_n,\varepsilon_n)) =0\right\} .
\end{align*}
For simplicity, we will denote by $k_n = \#\{i:Y_i=(0,1)\}- \frac{n}{2}$ and $x_\eta:=(\cos(\pi+\pi\eta/2),\sin(\pi+\pi\eta/2)))$ for any $\eta\in[-1,1]$.
The usual law of large numbers implies that, on an event $E_1$ of probability one, we have $k_n=o(n)$. Then, for any $\eta$, we have
\begin{align*}
	W_p(\bar\mu_n,x_\eta) -W_p(\bar\mu_n,(1,0))  &= \left(\frac{1}{2}+\frac{k_n}{n}\right) \left[(1-\eta)^p-1\right] +\left(\frac{1}{2}- \frac{k_n}{n}\right)\left[(1+\eta)^p-1 \right]\\
	&= -2p \eta\frac{k_n}{n} + O(\eta^2).
\end{align*}
Also, from the above equation, we note that since $p>1$, $W_p(\bar\mu_n,x)$ has a unique minimizer
\begin{equation*}
	\eta_n = \frac{u_n-1}{u_n+1},\quad u_n=\left(\frac{n+2k_n}{n-2k_n}\right)^{\frac{1}{p-1}}.
\end{equation*}
Thus, we have $\eta_n = \frac{2}{p-1} \frac{k_n}{n}+ O((\frac{k_n}{n})^2)$ and $W_p(\bar\mu_n,(1,0)) - m_p(\bar\mu_n) = \frac{4p}{p-1} (\frac{k_n}{n})^2(1 +\delta_n)$, where $\delta_n =  O(\frac{k_n}{n})$. Thus, we have now
\begin{align*}
	\left\{\dvechaus(F_p(\mu),F_p(\bar\mu_n,\varepsilon_n)) = 1\right\} = \{W_p(\bar\mu_n,(1,0)) - m_p(\bar\mu_n) >\varepsilon_n \}.
\end{align*}
Now, by the central limit theorem, $\frac{k_n}{\sqrt n}\to N(0,\frac{1}{4})$ in distribution.
As a result, for any $c>0$, 
\begin{equation*}
	\Pbb_{\mu}\left(W_p(\bar\mu_n,(1,0)) - m_p(\bar\mu_n) > \frac{4p}{p-1}\frac{c^2}{n} \right) \underset{n\to\infty}{\longrightarrow} \sqrt{\frac{2}{\pi}}\int_c^\infty e^{-2t^2}dt.
\end{equation*}
Further, this shows that $\Pbb_{\mu}(\dhaus(F_p(\mu),F_p(\bar\mu_n,\varepsilon_n))\geq 1)$ converges to $0$ if and only if one has $\dhaus(F_p(\mu),F_p(\bar\mu_n,\varepsilon_n)) \to 0$ in probability if and only if $\varepsilon_n = \omega(n^{-1})$.
It remains to check the other one-sided Haussdorff distance, for which it is already known that if $\varepsilon=o(1)$, 
then one has $\dhaus(F_p(\bar\mu_n,\varepsilon_n),F_p(\mu)) \to 0$ a.s.
Also, $W_p(\mu,\cdot)$ is continuous at $(-1,0)$. Therefore, if $\varepsilon\notin o(1)$, there exists an increasing sequence $(n_k)_k$ and $\eta> 0$ sufficiently small such that for all $k\geq 1$, we have $ \varepsilon_{n_k} > W_p(\mu,x_\eta) - m_p(\mu)$. Then, by the weak law of large numbers,
\begin{equation*}
	\Pbb_{\mu}(\dhaus(F_p(\bar\mu_{n_k},\varepsilon_{n_k}),F_p(\mu)) \geq \eta ) \geq \Pbb_{\mu}(x_\eta\in F_p(\bar\mu_{n_k},\varepsilon_{n_k})) \underset{n\to\infty}{\longrightarrow} 1.
\end{equation*}
This ends the proof of the claims for weak $\dhaus$-consistency.

We now turn to strong $\dhaus$-consistency.
The usual law of iterated logarithms shows that on an event $E_2$ of probability one, the set $\{(2n\log\log n)^{-1/2}k_n\}_{n\in\Nbb}$ has closure $[-\frac{1}{2},\frac{1}{2}]$.
Further, note that for $\varepsilon_n = c\log\log n/n$, we have
\begin{equation*}
	\{(1,0)\in F_p(\bar\mu_n,\varepsilon_n)\} = \left\{(2n\log\log n)^{-1/2} k_n \leq \sqrt{\frac{c (p-1)}{8p(1+\delta_n)}} \right\}.
\end{equation*}
Therefore, if $c>c^*$, on $E_1\cap E_2$, we have $\delta_n\to 0$ and for $n$ sufficiently large, $(1,0)\in F_p(\bar\mu_n,\varepsilon)$, hence $F_p(\bar\mu_n,\varepsilon_n)$ is strongly $\dhaus$-consistent (recall that $\varepsilon_n\to 0$). On the other hand, if $c<c^*$, on $E_1\cap E_2$, infinitely often we have $(1,0)\notin F_p(\bar\mu_n,\varepsilon)$ which implies that $F_p(\bar\mu_n,\varepsilon_n)$ is not strongly $\dhaus$-consistent. This ends the proof of the proposition.
\end{proof}

\begin{proof}[Proof of Proposition~\ref{prop:ex-two-point}]
Without loss of generality, we suppose $q>\frac{1}{2}$ so that $F_p(\mu) = \{x_1\}$. Now writing $\varepsilon_n = 2q-1 - c_n n^{-1/2}$, note that
\begin{align*}
	\Pbb_{\mu}(F_p(\bar\mu_n,\varepsilon_n)\neq F_p(\mu)) &=
	\Pbb_{\mu}(W_p(\bar\mu_n,x_2)\leq W_p(\bar\mu_n,x_1)+\varepsilon_n) \\
	&= \Pbb_{\mu}\left(\frac{2}{n}\sum_{i=1}^n \mathbf{1}[Y_i=x_1]-1 \leq \varepsilon_n\right)\\
	&=\Pbb_{\mu}\left(\frac{1}{\sqrt n}\sum_{i=1}^n (\mathbf{1}[Y_i=x_1] -q) \leq -\frac{c_n}{2} \right).
\end{align*}
By the central limit theorem, $n^{-1/2}\sum_{i=1}^n (\mathbf{1}[Y_i=x_1]-q) \to N(0,q(1-q))$. Therefore, we have $\Pbb_{\mu}(F_p(\bar\mu_n,\varepsilon_n)\neq F_p(\mu)) \to 0$ if and only if $c_n\to \infty$ almost surely. Now observe that one has $\dhaus(F_p(\mu),F_p(\bar\mu_n,\varepsilon_n))\in\{0,1\}$. Therefore, $F_p(\bar\mu_n,\varepsilon_n)$ is weakly $\dhaus$-consistent if and only if one has $\Pbb_{\mu}(F_p(\bar\mu_n,\varepsilon_n)\neq F_p(\mu))$ vanishes, proving the first claim.

We next turn to strong consistency.
Writing $\varepsilon_n = 2q-1 - c_n (\log\log n)^{1/2} n^{-1/2}$ and using the same arguments as above,
\begin{align*}
	\Pbb_{\mu}&(\dhaus(F_p(\bar\mu_n,\varepsilon_n),F_p(\mu))\to 0)\\
	&= \Pbb_{\mu}\left( \frac{1}{\sqrt{n\log\log n}}\sum_{i=1}^n (\mathbf{1}[Y_i=x_1]-q) > -\frac{c_n}{2} \text{ for $n$ suff. large}\right)
\end{align*}
Thus, the law of iterated logarithm readily implies the second claim since the variance of $\mathbf{1}[Y=x_1]-q$ is $q(1-q)$.
\end{proof}

\section{Proofs omitted from the phylogenetic application (Subsection~\ref{subsec:proof_phylo})}
\label{sec:phylogenomic}

\begin{proof}[Proof of Lemma~\ref{lemma:tropical_HBD}]
For simplicity, we identify $\Rbb^k/\Rbb\mathbf{1}$ to $\Rbb^{k-1}$ by setting the first coordinate to 0. For any $x, y\in\Rbb^{k-1}$, we then have
\begin{equation*}
	\dtrop(x,y) = \max\left(\|x-y\|_\infty,\max_{1\leq i<j\leq k-1} |x_i-x_j-y_i+y_j|\right).
\end{equation*}
In particular, note that
\begin{equation*}
	\|x-y\|_\infty \leq \dtrop(x,y) \leq 2\|x-y\|_\infty.
\end{equation*}
Since $(\Rbb^{k-1},\|\cdot\|_\infty)$ is HBD, the space $(\Rbb^{k-1},\dtrop)$ is HBD as well.
\end{proof}

We next turn to the proof of Theorem~\ref{thm:phylo-alg}. Algorithm~\ref{alg:phylo-estimator} implements an adapted form of the procedure described in Theorem~\ref{thm:adaptive-consistency}. The first computation in line 17 uses a sub-optimal relaxation $\varepsilon_{1,n,\delta} = m_1(\bar\mu_n) n^{-1/2}\log\log n$. In the following steps, the algorithm refines the potential pre-factor for the optimal relaxation rate $n^{-1/2}(\log\log n)^{1/2}$. In the following, we explain line 22 from Algorithm~\ref{alg:phylo-estimator}, in which we compute an upper bound for $\sigma_1(\bar\mu_n,\varepsilon_{s,n,\delta})$. We start with the case without relaxation.

\begin{proof}[Proof of Lemma~\ref{lemma:linear_frechet_mean}]
It suffices to show that $\dtrop(\cdot\,, z):[x,x']\to \R$ is affine for all $x, x'\in F_1(\mu)$, where $[x,x']:=\{(1-t)x+tx':0\le t \le 1\}$ denotes the line segment connecting $x$ and $x'$.
To show this, take an arbitrary $x, x'\in F_1(\mu)$ and $0\le t\le 1$, and set $x^\ast = (1-t)x+ tx'$.
Then take arbitrary $r>0$, and use $z\in\supp(\mu)$ to get $\mu(B_r(z))>0$.
Observe by construction that $\dtrop(\cdot\,, u):F_1(\mu)\to\R$ is convex for any $u\in\R^k/\Rbb\mathbf{1}$, hence that we have
\begin{equation}\label{eqn:phylo0}
	0 \le (1-t)(\dtrop(x,u)-\dtrop(x^{\ast},u)) + t(\dtrop(x',u) - \dtrop(x^{\ast},u)).
\end{equation}
Now integrate \eqref{eqn:phylo0} over $(\R^k/\Rbb\mathbf{1})\setminus B_r(z)$ with respect to $\mu$ to get
\begin{equation}\label{eqn:phylo1}
	\begin{split}
		0 \leq (1-t)&\int_{(\R^k/\Rbb\mathbf{1})\setminus B_r(z)}(\dtrop( x, u)-\dtrop(x^{\ast},u))\,d\mu(u) \\
		&+ t\int_{(\R^k/\Rbb\mathbf{1})\setminus B_r(z)}(\dtrop( x', u)-\dtrop(x^{\ast},u))\,d\mu(u).
	\end{split}
\end{equation}
Next add
\begin{equation}\label{eqn:phylo2}
	\mu(B_r(z))\Bigg((1-t)(\dtrop(x,z)-\dtrop(x^{\ast},z))+t(\dtrop(x',z)-\dtrop(x^{\ast},z))\Bigg)
\end{equation}
to both sides of \eqref{eqn:phylo1}, use $x,x'\in F_1(\mu)$, rearrange, and use the triangle inequality to get
\begin{align*}
	\mu(B_r(z))&\Bigg((1-t)(\dtrop(x,z)-\dtrop(x^{\ast},z))+t(\dtrop(x',z)-\dtrop(x^{\ast},z))\Bigg) \\
	&\le (1-t)W_1(\mu, x,x^{\ast}) + tW_1(\mu, x',x^{\ast}) \\
	&+(1-t)\int_{B_r(z)}((\dtrop(x,z) -\dtrop( x^{\ast}, z))-(\dtrop(x,u) - \dtrop(x^{\ast},u)))\,d\mu(u)\\
	&+t\int_{B_r(z)}((\dtrop(x',z) -\dtrop(x^{\ast}, z))-(\dtrop(x',u) - \dtrop(x^{\ast},u)))\,d\mu(u)\\
	&\le(1-t)\int_{B_r(z)}(r-(-r))\,d\mu(u)\\
	&+t\int_{B_r(z)}(r-(-r))\,d\mu(u)\\
	&\le 2\mu(B_r(z))\cdot r.
\end{align*}
Dividing the above by $\mu(B_r(z)) > 0$ implies
\begin{equation}\label{eqn:phylo3}
	(1-t)(\dtrop(x,z)-\dtrop(x^{\ast},z))+t(\dtrop(x',z)-\dtrop(x^{\ast},z) \leq 2r.
\end{equation}
Finally, by taking $r\to 0$ in \eqref{eqn:phylo3} and combining with \eqref{eqn:phylo0} for $u=z$, we get
\begin{equation}
	\dtrop(x^{\ast},z) = (1-t)\dtrop(x,z) + t\dtrop(x',z).
\end{equation}
This proves the result.
\end{proof}

\begin{lemma}\label{lemma:convexity}
If $\mu\in\Pcal(\Rbb^k/\Rbb\mathbf{1})$, then the function $\Var_{\mu}(\dtrop(\cdot\,, Y_1)-\dtrop(\cdot\,, Y_1)):F_1(\mu)\times F_1(\mu)\to [0,\infty)$ is convex.
Consequently, its maximum is attained at a pair of extreme points of $F_1(\mu)$.
\end{lemma}

\begin{proof}
Since $\Rbb^k/\Rbb\mathbf{1}$ is a separable metric space, one has $\mu(\supp(\mu))=1$.
Also, for $x,x'\in F_1(\mu)$, we have $W_1(\mu,x,x') = W_1(\mu,x',x) = 0$.
Thus:
\begin{align*}
	&\Var_{\mu}(\dtrop( x, Y_1) - \dtrop( x', Y_1)) \\
	&= \int_{\Rbb^k/\Rbb\mathbf{1}}  (\dtrop( x, u) - \dtrop( x', u))^2 \,d\mu(u) - \left(\int_{\Rbb^k/\Rbb\mathbf{1}} (\dtrop( x, u) - \dtrop( x', u)) \,d\mu(u) \right)^2\\
	&=   \int_{\supp(\mu)}  (\dtrop( x, u) - \dtrop( x', u))^2 \,d\mu(u).
\end{align*}
By Lemma~\ref{lemma:linear_frechet_mean}, for any $u\in \supp(\mu)$, the function $(x, x')\mapsto \dtrop( x, u) - \dtrop( x', u)$ is affine on $F_1(\mu)\times F_1(\mu)$, hence its square is convex.
This shows that the map $(x, x')\mapsto \Var_{\mu} (\dtrop( x, Y_1) - \dtrop( x', Y_1))$ is a convex combination of convex functions, hence convex.
\end{proof}

The convexity of the variance functional, unfortunately, does not extend to relaxed Fr\'echet mean sets.
To get around this, we need to introduce some notation:
For each $z\in\Rbb^k/\Rbb\mathbf{1}$, there exists a partition of $\Rbb^k/\Rbb\mathbf{1}$ into finitely many polyhedra $P_z=\{A_{z,m}: m\in M_{z}\}$ such that $\dtrop(\cdot\,, z):\Rbb^k/\Rbb\mathbf{1}\to\R$ is linear on each $A_{z,m}$; 
for a set $S\subseteq \Rbb^k/\Rbb\mathbf{1}$, we write $P_S:=\{A_{S,m}: m\in M_{S}\}$ for the coarsest common refinement of the partitions $\{P_z\}_{z\in S}$.

\begin{proposition}\label{prop:relaxed}
Suppose that $A\subseteq \Rbb^k/\Rbb\mathbf 1$ is a compact convex set with extreme points $\{v_i:i\in I\} := \textnormal{ex}(A)$ and that $\mu\in\Pcal(\Rbb^k/\Rbb\mathbf{1})$, and write $\{v_{A,j}: j\in I_{A}\}:=A\cap\bigcup_{m\in M_{\supp(\mu)}}\textnormal{ex}(A_{\supp(\mu),m})$ for the set of extreme points of the polyhedra in $P_{\supp(\mu)}$ which fall in $A$.
Then, the function $\E_{\mu}[(\dtrop(\cdot\,,Y_1)-\dtrop(\cdot\,,Y_1))^2]:A\times A\to [0,\infty)$ is maximzed at a pair of points in $\{v_i: i\in I\}\cup\{v_{A,j}: j\in J\}$.
\end{proposition}

\begin{proof}
By the same argument as in the proof of Lemma~\ref{lemma:convexity}, the function $(x,  x')\mapsto \E_{\mu}[(\dtrop( x,  z) - \dtrop(x', z))^2]$ is convex on $A_{\supp(\mu),m}\times A_{\supp(\mu),m'}$ for all $m,m'\in M_{\supp(\mu)}$, so its maximum must be achieved on $\textnormal{ex}(A_{\supp(\mu),m}\times A_{\supp(\mu),m'})=\textnormal{ex}(A_{\supp(\mu),m})\times \textnormal{ex}(A_{\supp(\mu),m'})$.
Because of the intersection with $A$, the result follows.
\end{proof}

We will apply this result to to the case $A=F_1(\bar \mu_n)$, where $\bar\mu_n = \frac{1}{n}\sum_{i=1}^{n}\delta_{Y_i}$ denotes the empirical measures of the first $n\in\Nbb$ samples.
To do this, observe that the extreme points of all the polyhedra in $P_{\supp(\bar \mu_n)} = P_{\{Y_1,\ldots, Y_n\}}$ can be found as the projection onto the first $k$ coordinates of the extreme points of the following polyhedron
\begin{equation*}
\left\{( v, c)\in\Rbb^k\times \Rbb^n: c_l \geq v_i-v_j -
(Y_l)_i + (Y_l)_j, \; l\in[n],i,j\in[k]\right\}.
\end{equation*}
As a result, we can find the extreme points of $F_1(\mu,\varepsilon)$ together with extreme points of the partition $P_{\{Y_1,\ldots, Y_n\}}$ falling in $F_1(\mu,\varepsilon)$ by taking the projection onto the first $k$ coordinates of the extreme points of the polyhedron
\begin{equation*}
\left\{( v, c)\in\Rbb^k\times \Rbb^n: \frac{1}{n}\sum_{l\in[n]}  c_l \leq m_1(\bar \mu_n) + \varepsilon,\text{ and } c_l \geq v_i-v_j -(Y_l)_i + (Y_l)_j, \; l\in[n],i,j\in[k]\right\}.
\end{equation*}
If we denote by $\{v_i:i\in I\}$ the resulting set of extreme points, then Proposition~\ref{prop:relaxed} shows that
\begin{equation*}
\sigma_1(\mu,\varepsilon) \leq \sqrt{2 \max_{j,j'\in I}\frac{1}{n}\sum_{i=1}^{n}(\dtrop( v_j, Y_i)-\dtrop( v_{j'}, Y_i))^2 }.
\end{equation*}
We are now ready to prove Theorem~\ref{thm:phylo-alg}.

\begin{proof}[Proof of Theorem~\ref{thm:phylo-alg}]
We will denote by $s^\ast$ the index of the terminal relaxation rate $\varepsilon_{s^\ast,n,\delta}$.
Algorithm~1 implements the strategy described above in the computation of line 22.
As a result, by Proposition~\ref{prop:relaxed}, for $s\geq 1$,
\begin{equation*}
	\sqrt{2 c_{s,n}} \geq \sigma_1(\bar\mu_n,\varepsilon_{s,n,\delta}).
\end{equation*}
For convenience, let $\varepsilon_{n,\delta'}:=(1+\delta'/2)^{1/2}\sigma_1(\mu) n^{-1/2}(\log\log n)^{1/2}$ for $0<\delta'<\delta$. Using the same arguments as in the proof of Theorem~\ref{thm:adaptive-consistency} leading to \eqref{eq:useful_3}
in the main body, the event
\begin{equation*}
	E_{\delta'} = \{F_1(\mu)\subseteq F_1(\bar\mu_n,\varepsilon_{n,\delta'})\text{ for suff. large }n\in\Nbb\} \cap  \left\{\liminf_{n\to\infty} \sigma_1(\bar\mu_n,\varepsilon_{n,\delta'}) \geq \sigma_1(\mu)\right\}
\end{equation*}
has full probability. Because $\varepsilon_{1,n,\delta}$ uses a suboptimal relaxation rate $n^{-1/2} \log\log n$, there exists $n_{\delta'}$ such that for $n\geq n_{\delta'}$, one has $\varepsilon_{1,n,\delta}\geq \varepsilon_{n,\delta'}$. Fix $n\geq n_0$ such that $\sigma_1(\bar\mu_n,\varepsilon_{n,\delta'}) \geq \sqrt{(2+\delta')/(2+\delta )} \sigma_1(\mu)$. We aim to show that for all $s\geq 1$, one has $\varepsilon_{s,n,\delta} \geq \varepsilon_{n,\delta'}$. We do this by induction. Because $n\geq n_{\delta'}$, this is true for $s=1$. Now suppose that for $s\geq 1$ one has $\varepsilon_{s,n,\delta} \geq \varepsilon_{n,\delta'}$. Using Proposition~\ref{prop:relaxed}, we have that
\begin{align*}
	c_{s,n} &= \sqrt{\max_{x,x'\in F_1(\bar\mu_n,\varepsilon_{s,n,\delta})} \frac{1}{n}\sum_{i=1}^n (\dtrop(x,Y_i) - \dtrop(x',Y_i))^2}\\
	&\geq \sqrt{\max_{x,x'\in F_1(\bar\mu_n,\varepsilon_{n,\delta'})} \frac{1}{n}\sum_{i=1}^n (\dtrop(x,Y_i) - \dtrop(x',Y_i))^2}\\
	&\geq \frac{\sigma_1(\bar\mu_n,\varepsilon_{n,\delta'})}{\sqrt 2}\\
	&\geq \sqrt{\frac{1+\delta'/2}{2+\delta}} \sigma_1(\mu).
\end{align*}
As a result, we obtain
\begin{equation*}
	\varepsilon_{s+1,n,\delta} = c_{s,n}\sqrt{\frac{(2+\delta) \log\log n}{n}}
	\geq \sigma_1(\mu) \sqrt{\frac{(1+\delta'/2)\log\log n}{n}}
	= \varepsilon_{n,\delta'},
\end{equation*}
where in the last inequality we used the assumption $\sigma_1(\bar\mu_n,\varepsilon_{n,\delta'}) \geq \sqrt{(2+\delta')/(2+\delta )} \sigma_1(\mu)$.
In particular, the final relaxation satisfies $\varepsilon_{s^\ast,n,\delta} \geq \varepsilon_{n,\delta'}$. Putting everything together, we obtain
\begin{equation*}
	E_{\delta'} \subseteq \left\{ \liminf_{n\to\infty} \sigma_1(\bar\mu_n,\varepsilon_{n,\delta'}) \geq \sigma_1(\mu)\right\}
	\subseteq \left\{\varepsilon_{s^\ast,n,\delta} \geq \varepsilon_{n,\delta'} \text{ for suff. large } n\in\Nbb\right\}.
\end{equation*}
This shows that, on $E_{\delta'}$, for sufficiently large $n\in\Nbb$ one has $F_1(\mu)\subseteq F_1(\bar\mu_n,\varepsilon_{n,\delta'})\subseteq F_1(\bar\mu_n,\varepsilon_{s^\ast,n,\delta})$, as well as
\begin{equation*}
	\liminf_{n\to\infty} \varepsilon_{s^\ast,n,\delta} \sqrt{\frac{n}{\log\log n}} \geq \liminf_{n\to\infty} \varepsilon_{n,\delta'} \sqrt{\frac{n}{\log\log n}} = \sqrt{\frac{2+\delta'}{2}} \sigma_1(\mu).
\end{equation*}
As a result, we consider the event $E:=\bigcap_{l\in\Nbb} E_{\delta -2^{-l}}$ which has full probability.
On $E$ and for any $\delta > 0$, we have
\begin{equation*}
	\liminf_{n\to\infty} \varepsilon_{s^\ast,n,\delta} \sqrt{\frac{n}{\log\log n}} \geq  \sqrt{\frac{2+\delta}{2}} \sigma_1(\mu),
\end{equation*}
and $F_1(\mu)\subseteq F_1(\bar\mu_n,\varepsilon_{s^\ast,n,\delta})$ for sufficiently large $n\in\Nbb$.

We now give upper bounds on $\varepsilon_{s^\ast,n,\delta}$. To do so, we use the same strategy as above but for the relaxation $\varepsilon_{1,n,\delta}$. The proof of Theorem~\ref{thm:adaptive-consistency} 
shows that the event
\begin{equation*}
	F = \left\{\limsup_{n\to\infty} \sigma_1(\bar\mu_n,\varepsilon_{1,n,\delta}) \leq \sigma_1(\mu) \right\}
\end{equation*}
has full probability.
Next, by construction of the relaxed Fr\'echet mean set, on this set, the Fr\'echet mean functional has values in $[m_1(\bar\mu_n),m_1(\bar\mu_n)+\varepsilon_{1,n,\delta}]$ so that for any $x,x'\in F_1(\bar\mu_n,\varepsilon_{1,n,\delta})$, we have
\begin{equation*}
	\E_{\bar \mu_n}[(\dtrop(x,Y_1) - \dtrop(x',Y_1))^2] \leq \Var_{\bar \mu_n}(\dtrop(x,Y_1) - \dtrop(x',Y_1)) + (\varepsilon_{1,n,\delta})^2.
\end{equation*}
Using Proposition~\ref{prop:relaxed}, we obtain
\begin{equation*}
	c_{1,n} = \sqrt{\max_{x,x'\in F_1(\bar\mu_n,\varepsilon_{1,n,\delta})} \frac{1}{n}\sum_{i=1}^n (\dtrop(x,Y_i) - \dtrop(x',Y_i))^2} \leq \sqrt{\frac{\sigma_1(\bar\mu_n,\varepsilon_{1,n,\delta})^2 }{2} + (\varepsilon_{1,n,\delta})^2}.
\end{equation*}
In particular, on $F$, we obtain
\begin{equation*}
	\limsup_{n\to\infty} \varepsilon_{2,n,\delta}\sqrt{\frac{n}{\log\log n}} \leq \sqrt{\frac{2+\delta}{2}} \limsup_{n\to\infty} \sigma_1(\bar\mu_n,\varepsilon_{1,n,\delta}) \leq \sqrt{\frac{2+\delta}{2}}\sigma_1(\mu).
\end{equation*}
Because we always have $\varepsilon_{s^\ast,n,\delta} \leq \varepsilon_{2,n,\delta}$, on $F$ this shows that
\begin{equation*}
	\limsup_{n\to\infty} \varepsilon_{s^\ast,n,\delta}\sqrt{\frac{n}{\log\log n}} \leq \sqrt{\frac{2+\delta}{2}}\sigma_1(\mu).
\end{equation*}
In particular, because $F$ has full probability via Theorem~\ref{thm:consistency}, we also obtain with probability one, $\dvechaus(F_1(\bar\mu_n,\varepsilon_{s^\ast,n,\delta}),F_1(\mu))\to 0$. Putting the upper and lower bounds together shows that with full probability, $\dhaus(F_1(\bar\mu_n,\varepsilon_{s^\ast,n,\delta}),F_1(\mu))\to 0$ and
\begin{equation*}
	\lim_{n\to\infty} \varepsilon_{s^\ast, n,\delta} \sqrt{\frac{n}{\log\log n}} = \sqrt{\frac{2+\delta}{2}}\sigma_1(\mu).
\end{equation*}
Finally, taking the intersection of these events over all $\delta\in\{2^{-l}:l\in\Nbb\}$ shows
\begin{equation*}
	\lim_{\delta\to 0}\lim_{n\to\infty} \varepsilon_{s^\ast, n,\delta} \sqrt{\frac{n}{\log\log n}} = \sigma_1(\mu),
\end{equation*}
which ends the proof.
\end{proof}

\end{appendix}

\end{document}